\documentclass[11pt]{amsart}[11pt]    
\pdfoutput=1
\usepackage[maxbibnames=99, style=alphabetic, maxalphanames=99, minalphanames=6, 
sorting=nyt, giveninits=true]{biblatex}
\AtEveryBibitem{
    \clearfield{issn}
  \clearfield{isbn}
  \clearfield{doi}
  \clearlist{location}
  \clearfield{month}
  \clearfield{day}
  \clearfield{series}
  \clearlist{language}
  \clearlist{translator}
  \renewbibmacro{in:}{}
}
    \usepackage{graphicx}
    \usepackage[english]{babel}
    \usepackage{graphicx}
    \usepackage{framed}
    \usepackage[normalem]{ulem}
   
    \usepackage{hyperref}
    \usepackage[capitalize]{cleveref}
    \usepackage{comment}
    \usepackage{tikz}
    \usetikzlibrary{matrix,calc}
    \usepackage{mathtools}
    \usepackage{tikz-cd}
   
    \usepackage[utf8]{inputenc}
    \usepackage[top=1 in,bottom=1in, left=1 in, right=1 in]{geometry}
       \usepackage{stmaryrd}
    \usepackage{bbold}

    \usepackage[all]{xy}
    \usepackage{mathrsfs}
\usepackage[mathscr]{eucal}

\usepackage{amsmath}
\usepackage{amsthm}
\usepackage{amssymb}
\usepackage{latexsym}
\usepackage{graphicx}
\newcommand{\red}[1]{{\color{red}#1}}
\newcommand{\blue}[1]{{\color{blue}#1}}

    \makeatletter
\renewcommand{\subsection}{\@startsection{subsection}{2}{0pt}{-3ex
plus -1ex minus -0.2ex}{-2mm plus -0pt minus
-2pt}{\normalfont\bfseries}} \makeatother

\numberwithin{equation}{subsection}
\theoremstyle{plain}
    \newtheorem{Theorem}[equation]{Theorem}
    \newtheorem{Corollary}[equation]{Corollary}

\newtheorem{Lemma}[equation]{Lemma} 
\newtheorem{Proposition}[equation]{Proposition} 

\theoremstyle{definition}

\newtheorem{defn}[equation]{Definition}
\newtheorem*{Definition}{Definition} 

\theoremstyle{remark}

\newtheorem{examp}[equation]{Example}

\newtheorem{rem}[equation]{Remark}
\newtheorem{rems}[equation]{Remarks}
\newtheorem{claim}[equation]{Claim}
\newtheorem{Remark}[equation]{Remark} 

     \allowdisplaybreaks

\DeclareMathOperator{\Ad}{{\mathrm{Ad}}}
\DeclareMathOperator{\ad}{{\mathrm{ad}}}

\DeclareMathOperator{\Spec}{{\mathrm{Spec}}}
\DeclareMathOperator{\Sym}{{\mathrm{Sym}}}
\DeclareMathOperator{\mmod}{{\text{-}{\mathrm{mod}}}}
\DeclareMathOperator{\Rep}{{\mathrm{Rep}}}
\DeclareMathOperator{\Ind}{{\mathrm{Ind}}}
\DeclareMathOperator{\Res}{{\mathrm{Res}}}
\DeclareMathOperator{\Ker}{{\mathrm{Ker}}}
\DeclareMathOperator{\bimod}{{\mathrm{\text{-}\mathrm{bimod}}}}
\DeclareMathOperator{\comod}{{\mathrm{\text{-}\mathrm{comod}}}}
\DeclareMathOperator{\act}{{\mathrm{act}}}
\DeclareMathOperator{\coact}{{\mathrm{coact}}}
\DeclareMathOperator{\Av}{{\mathrm{Av}}}
\renewcommand{\o}{\otimes }
\newcommand{\td}{\text{-}}

\newcommand{\LieAlgebraofUnipotentRadicalOfOppositeBorel}{\mathfrak{n}}
\newcommand{\unipotentRadicalOfOppositeBorel}{N}
\newcommand{\LieAlgebraOfUnipotentRadicalOfBorel}{\bar{\mathfrak{n}}}
\newcommand{\BorelSubgroup}{\bar{B}}
\newcommand{\unipotentRadicalOfBorel}{\bar{N}}
\newcommand{\oppositeOfBorelSubgroup}{B}
\newcommand{\characterlatticeforT}{{\mathbb{X}}^*(T)}

\newcommand{\LTd}{\LT^{\ast}}
\newcommand{\LGd}{\mathfrak{g}^*}
\newcommand{\Symt}{\operatorname{Sym}(\mathfrak{t})}
\newcommand{\LG}{\mathfrak{g}}
\newcommand{\LT}{\mathfrak{t}}

\newcommand{\regularCentralizersForThisGroupDefaultstoG}[1][G]{Z^{\text{reg}}_G}

\newcommand{\hb}{{\hbar}}
\newcommand{\khb}{{k[\hb]}}
\newcommand{\C}{\mathcal{C}} 
\newcommand{\D}{\mathcal{D}} 
\renewcommand{\O}{\mathcal{O}} 
\newcommand{\F}{\mathcal{F}} 

\newcommand{\QCoh}{\operatorname{QCoh}} 
\newcommand{\Ga}{\Gamma }

\newcommand{\modwithdash}{\mmod} 
\newcommand{\op}{{\text{op}}} 
\newcommand{\Z}{\mathcal{Z}} 
\newcommand{\Avpsil}{\text{Av}_*^{\psi_{\ell}}} 
\newcommand{\Avpsir}{\text{Av}_*^{\psi_{r}}} 
\newcommand{\Ngo}{\operatorname{{\mathsf{Ng\mathaccent"705E o}}}}
\newcommand{\ups}{\Upsilon }
\newcommand{\ssh}{/\!/ }
\newcommand{\iso}{{\;\stackrel{_\sim}{\to}\;}}

\newcommand{\uhg}{{U_\hb\g}}
  \newcommand{\Dh}{{D_\hb}}
  \newcommand{\Dht}{{D_{\lhd,\hb}}}
\newcommand{\obl}{{\operatorname{Obl}}}
\newcommand{\zhg}{{Z_\hb\g}}
\newcommand{\jh}{{{\mathfrak J}_\hb}}
\newcommand{\jht}{{\J_{\lhd,\hb}}}
\newcommand{\jt}{{\J_{\lhd}}}
\newcommand{\erem}{\hfill$\lozenge$\end{rem}}
\newcommand{\eerem}{\hfill$\lozenge$\end{rem}\vskip 3pt }

\newcommand{\beq}{\begin{equation}\label}
  \newcommand{\eeq}{\end{equation}}
\newcommand{\fc}{{\mathfrak c}}
\def\ccirc{{{}_{\,{}^{^\circ}}}}
\DeclareMathOperator{\Hom}{\mathrm{Hom}}

\DeclareMathOperator{\GL}{\mathrm{GL}}

\newcommand{\sk}{{\operatorname{sk}}}

\newcommand{\pt}{{\operatorname{pt}}}

\newcommand{\De}{\Delta }

\newcommand{\nilDAHAAtHBARISONE}{\mathbb{H}}

\renewcommand{\o}{\otimes }

\newcommand{\wt}{\widetilde }

\newcommand{\Id}{{\operatorname{Id}}}

\newcommand{\BX}{{{\mathsf Q}}}
\newcommand{\BM}{{\mathbf M}}
\newcommand{\quantizationOfCentralizersHBAREQUALSONE}{{\mathbf N}}

\newcommand{\fn}{{\mathfrak n}}

\newcommand{\scr}[1]{\mathscr{#1}}

\newcommand{\we}{_{\text{weak}}}
\newcommand{\Mi}{{\mathbb{M}}}
\newcommand{\MiuraBimodule}{\Mi}
\newcommand{\Dpsirhbar}{D^{\psi_{r}}_{\hbar}}
\newcommand{\Dpsilhbar}{D^{\psi_{\ell}}_{\hbar}}
\newcommand{\Dpsirzero}{\mathcal{O}(T^*_{\psi}(G/\unipotentRadicalOfOppositeBorel))}

\newcommand{\biWhittakerDifferentialOperatorsHBAR}{\mathfrak{J}_{\hbar}}

\newcommand{\Dpsil}{D^{\psi_{\ell}}}
\newcommand{\Dpsir}{D^{\psi_{r}}}
\newcommand{\biWhittakerDifferentialOperatorsonG}{\mathfrak{J}}
\newcommand{\biWhittakerDifferentialOperatorsZERO}{J}
\newcommand{\freeFunctorASYMPTOTICSETTING}{{\operatorname{Fr}}}
\newcommand{\DhbarG}{{D_{\hbar}}}
\newcommand{\Zhbarg}{Z_{\hbar}\LG}
\newcommand{\Zhbarh}{Z_{\hbar}\LH}
\newcommand{\Uhbarg}{U_{\hbar}\LG}
\newcommand{\AsymptoticHarishChandraBimodulesForTHISGROUPDefaultsToG}{\mathcal{HC}_{\hbar}}
\newcommand{\kostantWhittakerForNONFILTEREDtoZGBIMOD}
{\varkappa}
\newcommand{\kostantWhittakerVIEWEDASLEFTZGMODULE}{\varkappa_{\ell}}

\newcommand{\kostantWhittakerToWComod}{\kappa}

\newcommand{\kostantWhittakerInCLASSICALSETTINGwithGRADEDLIFTtoREPOFWZERO}{\kappa_0}
\newcommand{\kostantWhittakerHBARandVIEWEDASLEFTZGMODULE}{\varkappa_{\hbar, \ell}}
\newcommand{\kostantWhittakerHBARtoZGBIMOD}{\varkappa_{\hbar}} 
\newcommand{\kostantWhittakerASYMPTOTICtoComodForASYMPTOTICbiWhit}{\kappa_{\hbar}}
\newcommand{\LieAlgebraofUnipotentRadicalOfOppositeBorelPSI}{\LieAlgebraofUnipotentRadicalOfOppositeBorel^{\psi}}

\newcommand{\LieAlgebraofUnipotentRadicalOfOppositeBorelPSIell}{\mathfrak{n}^{\psi}_{\ell}}
\newcommand{\LieAlgebraofUnipotentRadicalOfOppositeBorelPSIl}{\LieAlgebraofUnipotentRadicalOfOppositeBorelPSIell}
\newcommand{\LieAlgebraofUnipotentRadicalOfOppositeBorelPSIr}{\mathfrak{n}^{\psi}_{r}}

\def\ccirc{{{}_{^{\,^\circ}}}}

\newcommand{\la}{\lambda}

\newcommand{\J}{{\mathfrak J}}

\newcommand{\bone}{{\boldsymbol{1}}}

\newcommand{\dd}{{\mathscr{D}}}
\newcommand{\oo}{{\mathcal{O}}}

\newcommand{\U}{{U}}

\newcommand{\back}{\backslash }

\renewcommand{\t}{{\mathfrak t}}

\newcommand{\Gr}{{\mathsf{Gr}}}

\newcommand{\al}{\alpha }

\newcommand{\vkap}{\varkappa }

\newcommand{\kap}{\kappa }

\newcommand{\eps}{\epsilon }

\newcommand{\g}{\mathfrak{g}}

\newcommand{\h}{\mathfrak{h}}

\newcommand{\inv}{^{-1}}

\newcommand{\cf}{{\mathcal F}}

\newcommand{\en}{{\enspace}}

\newcommand{\vi}{${\en\sf {(i)}}\;$}
\newcommand{\vii}{${\;\sf {(ii)}}\;$}
\newcommand{\viii}{${\sf {(iii)}}\;$}
\newcommand{\iv}{${\sf {(iv)}}\;$}
\newcommand{\vv}{${\sf {(v)}}\;$}

\newcommand{\sset}{\subseteq}

\newcommand{\intoo}{\,\xymatrix{\ar@{^{(}->}[r]&}\,}
\newcommand{\ontoo}{\,\xymatrix{\ar@{->>}[r]&}\,}
\newcommand{\into}{\,\hookrightarrow\,}

\newcommand{\mto}{\mapsto}
\newcommand{\onto}{\,\twoheadrightarrow\,}

\newcommand{\hdot}{{\:\raisebox{2pt}{\text{\circle*{1.5}}}}}

\newcommand{\freeFunctorForNONFILTEREDHarishChandraBimodules}{\text{Fr}}
\newcommand*{\HarishChandraBimodulesForTHISGROUPDefaultsToG}{\mathcal{HC}}
\newcommand{\LH}{\mathfrak{h}}
\newcommand{\bfN}{\mathbf{N}}
\newcommand{\differentialOperatorsOnT}{D(T)}

\newcommand{\differentialOperatorsOnG}{D(G)}
\newcommand{\quantizationOfCentralizersHBAR}{\mathbf{N}_{\hbar}}

\newcommand{\End}{\operatorname{End}}
\newcommand{\PhiTHISGROUPRho}[1]{\mathbf{\Phi}_{#1, {V}}}
\newcommand{\IndTG}{\operatorname{Ind}_T^G}

\newcommand{\kostantWhittakerSTANDIN}{K}
\newcommand{\fieldOfFractionsForSymt}{\operatorname{Q}}

    \title{Quantization of the universal centralizer and central D-modules}
   \author{Tom Gannon and Victor Ginzburg}    

\begin{document}

\begin{abstract}
The group scheme of universal centralizers of a complex reductive group $G$
has a quantization called the \textit{spherical nil-DAHA}. The category of modules over this ring
is equivalent, as a symmetric monoidal category, to the category of bi-Whittaker $D$-modules
on $G$. We construct a braided monoidal equivalence of
this category with a full monoidal subcategory of the abelian category of
Ad $G$-equivariant $D$-modules, establishing a $D$-module abelian counterpart of an equivalence established by Bezrukavnikov and Deshpande, in a different way. 

As an application of our methods, we prove conjectures of Ben-Zvi and Gunningham by relating this equivalence to parabolic induction and prove a conjecture of Braverman and Kazhdan in the $D$-module setting.
\end{abstract}

\maketitle

{\small
\tableofcontents
}
\section{Main results}

We work over the field  of complex numbers to be denoted $k$, put $\o=\o_k$,
and write ${\mathcal O}(X)$, resp. $D(X)$, for the algebra of regular functions,
resp. differential operators, on
a variety $X$.
\subsection{Knop-Ng\^o morphism}
\label{main1}
Let $G$ be a connected
reductive group with Lie algebra $\LG$.
Let  ${\mathfrak c}:=\LG^*\ssh G=\Spec((\Sym\LG)^G)$ be the categorical quotient
 of  the coadjoint action
of $G$ on $\g^*$ (to be denoted `$\Ad$'), and let
  $Z=\{(g,\xi)\in G\times \g^*\mid \Ad g(\xi)=\xi\}$ 
denote the commuting variety, viewed as a reduced closed subscheme of $G\times \g^*$.
The projection map $pr: Z \to \LGd$ makes $Z$
a group scheme over $\LGd$ such that the fiber  $pr\inv(\xi)$ for $\xi\in\LG^*$
 is the stabilizer of $\xi$ in $G$.
Restricting to the set $\LG^*_{\text{reg}}\subseteq\LG^*$  
of regular elements, one obtains
the group scheme $Z_{\text{reg}}=pr\inv(\LG^*_{\text{reg}})\to \LG^*_{\text{reg}}$
of \lq{}regular centralizers\rq{}.

In his proof of the fundamental lemma, Ng\^o constructed a morphism from a smooth affine abelian group scheme $J\to{\mathfrak c}$ called
the \lq universal centralizer\rq{}, equipped with
a  canonical isomorphism $\LG^*_{\text{reg}}  \times_{{\mathfrak c}} J\iso
Z_{\text{reg}}$
of  $G$-equivariant group schemes over $\LG^*_{\text{reg}}$
that fits into commutative diagram
 
 \begin{equation}\label{Classical Ngo Map on Regular Locus}
  \xymatrix{
    \LG^*_{\text{reg}}  \times_{{\mathfrak c}} J\  \ar@{^{(}->}[rr]
    \ar[d]^<>(0.5){\cong}&&\   \LG^* \times_{{\mathfrak c}} J\ar[d]^<>(0.5){j}\\
Z_{\text{reg}} \   \ar@{^{(}->}[rr] && \  Z
   }
\end{equation}
of morphisms of  $G$-equivariant group schemes, \cite[Section 2.1]{NgoLeLemmeFondamentalPourLesAlgebresdeLie}.
The scheme $J$, resp. its group theoretic analogue, has been  considered earlier
by Kostant \cite{KostantTheSolutiontoaGeneralizedTodaLatticeandRepresentationTheory},
resp. Lusztig  \cite{LusztigCoxeterOrbitsandEigenspacesofFrobenius},
using a less canonical construction involving slices.
Similar constructions in the setting of spherical varieties were also considered by Knop \cite{KnopAutomorphismsRootSystemsandCompactificationsofHomogeneousVarieties}.

The group scheme structure
makes $\mathcal{O}(Z)$, resp. $\mathcal{O}(J)$, a Hopf algebra over
$\Sym(\LG)=\mathcal{O}(\LG^*)$, resp. $\mathcal{O}({\mathfrak c})$, and
there is an isomorphism
$\mathcal{O}(\LG^* \times_{{\mathfrak c}} J)=\Sym(\LG)\otimes_{{\mathcal{O}}({\mathfrak c})} \mathcal{O}(J)$
of Hopf algebras over
$\Sym(\LG)$.
Therefore, it follows from the above diagram that 
restriction  of regular functions from $Z$ to $Z_{\text{reg}}$
can be factored  as a composition of  Hopf algebra maps
\begin{equation}\label{class ngo alg}
\mathcal{O}(Z)\xrightarrow{{\mathsf{kn}}}
\Sym(\LG)\otimes_{\mathcal{O}({\mathfrak c})} \mathcal{O}(J)\xrightarrow{}\mathcal{O}(Z_{\text{reg}}).
\end{equation} We  refer to the first map in \labelcref{class ngo alg} as the \textit{Knop-Ng\^o morphism}.

Let $\mathcal{O}(J)\mmod$, resp. $\mathcal{O}(Z)\mmod^{G}$, 
be the abelian category of $\mathcal{O}(J)$-modules, resp. $G$-equivariant
$\mathcal{O}(Z)$-modules.
The group scheme structure gives each of these categories a monoidal structure with respect to convolution.
Using Diagram 
\labelcref{Classical Ngo Map on Regular Locus} we obtain a chain of monoidal functors
\begin{equation}\label{class ngo functor}
  \mathcal{O}(J)\mmod
  \ \xrightarrow{\Sym(\LG)\otimes_{\mathcal{O}({\mathfrak c})}({\text{-}})}\
\mathcal{O}(\LG^*\times_{\mathfrak{c}}J)\mmod^{G}\ \xrightarrow{j_*}\ 
\mathcal{O}(Z)\mmod^{G},
\end{equation}
where the second functor $j_*$, which corresponds to push-forward of quasi-coherent
sheaves along the embedding $j$, makes an $\mathcal{O}(\LG^*\times_{\mathfrak{c}} J)$-module
into an $\mathcal{O}(Z)$-module via the algebra map ${\mathsf{kn}}$.

%

One of the goals of this paper, inspired to a large extent by the ideas of Ben-Zvi and Gunningham
\cite{BenZviGunninghamSymmetriesofCategoricalRepresentationsandtheQuantumNgoAction},
is to construct a quantization of the algebra morphism ${\mathsf{kn}}$
in \eqref{class ngo alg} and 
the composite functor
in  \labelcref{class ngo functor}.

The algebras that appear
in \eqref{class ngo alg} each have a well known \textit{quantization},
that is, non-commutative deformation.
The quantization of  $\Sym(\LG)$,
resp. $\mathcal{O}({\mathfrak c})$, is the enveloping algebra
$U\LG$, resp.  the center $Z\LG$ of $U\LG$.
The  quantization of $\mathcal{O}(G\times \LG^*)$
is the algebra $ D(G)$, where
we have used the identification $G\times \LG^*\cong T^*G$.
The left ideal $ D(G)\ad\LG\sset D(G)$ quantizes the
ideal of definition of
the {\em not necessarily reduced} subscheme $Z_{\text{scheme}}\sset G\times \LG^*$.
Therefore,
the $D(G)$-module
${\mathbf N}:= D(G)/D(G)\ad\LG$ provides a quantization of 
the functions on this subscheme, viewed as an $\mathcal{O}(T^* G)$-module.
We show that the coalgebra structure on ${\mathcal O}(Z)$ can be
  deformed to the structure of a coalgebra object on ${\mathbf N}$, cf. \cref{Definition of Ring Object Coring Object Symmetric and Left Comodule}.


\newcommand{\gv}{{G^\vee}}
The algebra $\mathcal{O}(J)$ is known to have
a  deformation to a noncommutative algebra 
${\mathfrak J}$ called {\em spherical nil DAHA}.
This algebra
may be defined as a quantum Hamiltonian reduction (aka bi-Whittaker reduction)
of $D(G)$ with respect to  $(N\times N_{op},\, \psi\times\psi)$,
where $N$ is the unipotent radical of a Borel
subgroup $B\sset G$,
the groups $N$ and $N_{op}$ act on $G$ by left and right
translations, respectively, and  $\psi$ is  a nondegenerate character of $N$,
see \cite{BezrukavnikovFinkelbergMirkovicEquivariantHomologyandKTheoryofAffineGrassmanniansandTodaLattices}, \cite{BezrukavnikovFinkelbergEquivariantSatakeCategoryandKostantWhittakerReduction}, and also \cref{Monoidal Structure on BiWhit Mod and Monoidality}
for more details. 
It was shown in \textit{loc. cit.} that $\biWhittakerDifferentialOperatorsonG$ is isomorphic, via geometric Satake,
to the equivariant Borel-Moore homology of the affine Grassmannian
$\mathsf{Gr}_{\gv}$ of the Langlands dual group as an algebra.
The coproduct on the
equivariant Borel-Moore homology of the affine Grassmannian induced,
via the K\"unneth formula, by
the diagonal embedding
$\mathsf{Gr}_{\gv}\into \mathsf{Gr}_{\gv}\times \mathsf{Gr}_{\gv}$, therefore equips $\biWhittakerDifferentialOperatorsonG$ with a cocommutative coalgebra structure.
In \cref{Kostant-Whittaker Reduction Is Monoidal Localization Section} we will give an algebraic construction of this
coproduct and an associated
symmetric monoidal structure on the category ${\mathfrak J}\mmod$.

In view of the above, it is natural to expect that the quantization of the  map
$\mathsf{kn}$ is a map of the form
\begin{equation}\label{quant ngo alg}
    {\mathbf N} \to
    U\LG\otimes_{Z\LG}{\mathfrak J}.
  \end{equation}
  
  A primary difficulty in  constructing such a map is that
  it is not  {\em a priori} clear what kind of structure to expect from
  the map in \labelcref{quant ngo alg}:
  the domain of the map
  has the structure
  of a $ D(G)$-module 
  while the codomain is a $(U\LG, {\mathfrak J})$-bimodule that
  has no immediately visible  $D(G)$-action.
  
  We construct the quantized Knop-Ng\^o
  morphism and the quantized Knop-Ng\^o  functor
  using the monoidal category $\HarishChandraBimodulesForTHISGROUPDefaultsToG$ of
Harish-Chandra $(U\LG, U\LG)$-bimodules, a category which serves as a bridge between the categories of ${\mathfrak J}$-modules and $D(G)$-modules. 
To this end, given a ${\mathfrak J}$-module $E$ we make   $U\LG\otimes_{Z\LG} E$
a Harish-Chandra bimodule with `diagonal'  action of
the algebra $U\LG\otimes U\LG_{op}$ defined by the formula
  $u_1\otimes u_2: u\otimes e\mapsto u_1uu_2\otimes e$.
  A key idea is  that the ${\mathfrak J}$-module structure on $E$ provides
  the Harish-Chandra bimodule $U\LG\otimes_{Z\LG}E$
  with a canonical {\em nontrivial} central structure.
  That is, we show that there is
    a canonical lift of $U\LG\otimes_{Z\LG}E$ to an object of  the Drinfeld center  
  of $\HarishChandraBimodulesForTHISGROUPDefaultsToG$. 
The Drinfeld center has been identified in \cite{BezrukavnikovFinkelbergOstrikCharacterDModulesViaDrinfeldCenterofHarishChandraBimodules}
with the category $\differentialOperatorsOnG\mmod^{\Ad G}$ of $G$-equivariant $\differentialOperatorsOnG$-modules, where $G$ acts by the adjoint action. We will actually use a
strengthening of this result
that
identifies the latter category
with the {\em centralizer}  of the category $\Rep G$ in $\HarishChandraBimodulesForTHISGROUPDefaultsToG$, see \cref{Radj for Drinfeld Center via Relative Setting}. 
We therefore obtain the structure of a
$G$-equivariant $\differentialOperatorsOnG$-module on $U\LG\otimes_{Z\LG}E$.

Let $W$ be the Weyl group and $D(T)^W$ the algebra of $W$-invariant
differential operators on the maximal torus $T$.
In \cref{Miura Bimodule and Comodules Subsection}, cf. also \cite{Gin},  we will see that there is a natural algebra embedding $D(T)^W \into \biWhittakerDifferentialOperatorsonG$.
There is also an algebra isomorphism 
$(\End_{\differentialOperatorsOnG}{\quantizationOfCentralizersHBAREQUALSONE})_{\text{op}} \cong \differentialOperatorsOnT^W$ constructed by Harish-Chandra 
via the `radial parts' map, cf. \cite[Remark (iv)]{GinzburgParabolicInductionandtheHarishChandraDModule}
for a short exposition.  This gives a left $\differentialOperatorsOnT^W$-action on
any $\biWhittakerDifferentialOperatorsonG$-module and 
a right $\differentialOperatorsOnT^W$-action
on ${\quantizationOfCentralizersHBAREQUALSONE}$ that commutes with the
(left) $\differentialOperatorsOnG$-action.

Our first result, whose first part is proved in
\cref{Definition of Ngo Functor Subsection}
and second part is proved in \cref{Proof of First Theorem Part Two}, reads
 
\begin{Theorem}\label{Theorem Asserting Existence of Ngo Functor as Braided Monoidal Fully Faithful Functor to Center}
  \vi   The functor $\biWhittakerDifferentialOperatorsonG\mmod \to {\HarishChandraBimodulesForTHISGROUPDefaultsToG},\, E\mapsto U\LG \otimes_{Z\LG} E$,
  has a canonical lift
  to a fully faithful, braided monoidal, exact functor
  $\ups: {\biWhittakerDifferentialOperatorsonG}\mmod \to \differentialOperatorsOnG\mmod^{\Ad G}$.
  
  \vii For any $\biWhittakerDifferentialOperatorsonG$-module $E$, the map
  $D(G) \o E \to U\LG \otimes_{Z\LG} E, \, u \o e\mto u(1\o e)$,
  given by the $D(G)$-action yields an isomomorphism of $D(G)$-modules:
  \[\mathbf{N}\o_{D(T)^W} E \iso\ups(E).
    \]
  \end{Theorem}


  Applying the theorem in the special case where
$E$ is a rank one free ${\mathfrak J}$-module yields the desired $\differentialOperatorsOnG$-module
structure on the object  $U\LG \otimes_{Z\LG} {\mathfrak J}=\ups({\mathfrak J})$
that appears in \eqref{quant ngo alg}. Using this, we prove the following result, which confirms a conjecture of Nadler:

\begin{Corollary}\label{Morphism Which Quantizes the Ngo Morphism as Map of Coalgebra Objects in Central HCG} The map
  \[{\quantizationOfCentralizersHBAREQUALSONE}=
  {\quantizationOfCentralizersHBAREQUALSONE}\o_{D(T)^W} {D(T)^W}\ \xrightarrow{} \
  {\quantizationOfCentralizersHBAREQUALSONE}\o_{D(T)^W} {\mathfrak J}\cong
  {U\LG}\otimes_{Z\LG} {\mathfrak J}
\]
  gives a $D(G)$-module map of coalgebras in $\HarishChandraBimodulesForTHISGROUPDefaultsToG$ that quantizes the morphism $\mathsf{kn}$ in \eqref{class ngo alg}.
\end{Corollary}

\begin{Remark}
    A geometric construction of
     the object $\ups(\mathfrak{J})$
     that makes sense in the
      constructible setting (eg. of $\ell$-adic sheaves) in terms of averaging functors and the Grothendieck-Springer resolution will be given in  \cref{Perverse Sheaf Counterpart Remark}.
\end{Remark}

\newcommand{\derivedCategoryOfWEquivariantDTModules}{\mathscr{D}^W(T)}
\newcommand{\derivedCategoryOfWEquivariantDTModulesVANISHINGCONDITION}{\mathscr{D}^W_{\circ}(T)}
\newcommand{\derivedCategoryofAdGEquivariantDgModules}{\mathscr{D}^G(G)}

\smallskip

\begin{rem}[{\em Comparison to work of
  Ben-Zvi and Gunningham}]
The idea to relate the category $\biWhittakerDifferentialOperatorsonG$-mod to the category $\differentialOperatorsOnG\mmod^G$ has been implemented in a different way by Ben-Zvi and Gunningham \cite{BenZviGunninghamSymmetriesofCategoricalRepresentationsandtheQuantumNgoAction}, who have obtained a result resembling our \cref{Theorem Asserting Existence of Ngo Functor as Braided Monoidal Fully Faithful Functor to Center}(i). However, their approach is different from ours in that 
they work with the full triangulated category of $G$-equivariant $D$-modules on $G$ rather
than with the abelian heart of a particular $t$-structure. The categorical center construction does not produce reasonable results when applied to triangulated categories,
and the formalism of $\infty$-categories (or, more precisely, \textit{DG categories} as defined in \cite[Chapter 1.10]{GRI}) has been employed to overcome this difficulty. Using the exactness of Kostant-Whittaker reduction allows us to bypass that difficulty and work with abelian categories; thus the technique of the present paper is more elementary than that of \cite{BenZviGunninghamSymmetriesofCategoricalRepresentationsandtheQuantumNgoAction}.
Also, the proof of \cite[Conjecture 2.9]{BenZviGunninghamSymmetriesofCategoricalRepresentationsandtheQuantumNgoAction} discussed below, which was one of our main motivations for this project, is unlikely to be accessible by the methods of \cite{BenZviGunninghamSymmetriesofCategoricalRepresentationsandtheQuantumNgoAction} alone.
\end{rem}

\newcommand{\derivedCategoryOfNEquivariantDModulesOnBaseAffineSpace}{\mathscr{D}^N({G/N})}
\newcommand{\derivedCategoryofVeryCetnralDGModules}{\mathscr{D}^G_{\text{vc}}(G)}
\newcommand{\abeliandCategoryOfVeryCentralAdGEquivariantDmodulesonG}{\mathscr{D}^G_{\operatorname{vc}}(G)^{\heartsuit}}
\newcommand{\derivedBiWhittakerCategory}{\mathscr{H}_{\psi}}

\subsection{Parabolic induction and very central $D$-modules}
\label{par-ind-intro}    Given an algebraic group $H$ and  a smooth  $H$-variety $X$, we write
${\scr D}^H(X)$ for the $H$-equivariant derived category of $D$-modules on $X$.
If $X$ is affine, we may (and will) identify sheaves of $D$-modules  with
$D(X)$-modules, so the abelian heart of ${\scr D}^H(X)$
gets identified with the category 
$D(X)\mmod^H$ of (strongly) $H$-equivariant $D(X)$-modules. In the special case $X=T$ and $H=W$, the category
$\derivedCategoryOfWEquivariantDTModules$ 
is equivalent
to the derived category of its heart $D(T)\mmod^W$ and the latter category
is  equivalent
to $D(T)^W\mmod$,
thanks to the Morita equivalence
\begin{equation}\label{Morita Equivalence}D(T)^W \sim D(T)\rtimes W.
\end{equation} 
  It turns out that our functor $\ups$ can be expressed
  in terms of the more familiar functor
  $\derivedCategoryOfWEquivariantDTModules \to \derivedCategoryofAdGEquivariantDgModules$   of parabolic induction. The latter functor is exact, \cite{BezrukavnikovYomDinOnParabolicRestrictionofPerverseSheaves}, so it restricts to an exact  functor $\Ind_T^G: \differentialOperatorsOnT\mmod \to \differentialOperatorsOnG\mmod^{G}$ between the respective abelian categories. Furthermore,
for any $\cf\in \derivedCategoryOfWEquivariantDTModules$ the object $\Ind_T^G(\cf)$ comes equipped
with a canonical $W$-action, see \cite[Section 3.2]{ChenOnTheConjecturesofBravermanKazhdan}. Let $\Ind_T^G(\cf)^W$ denote the direct summand
of $W$-invariants.

\cref{Theorem Asserting Existence of Ngo Functor as Braided Monoidal Fully Faithful Functor to Center} combined with \cite[Theorem 1.6]{GinzburgParabolicInductionandtheHarishChandraDModule} yields the following result, which proves \cite[Conjecture 2.9]{BenZviGunninghamSymmetriesofCategoricalRepresentationsandtheQuantumNgoAction}:

\begin{Corollary}\label{Jcor} The functor $\ups$ is isomorphic to the  composite functor
  \begin{equation}\label{Big Composite for Ngo}
    \biWhittakerDifferentialOperatorsonG\mmod\xrightarrow{\ \operatorname{Obl}^{{\mathfrak J}}_{\differentialOperatorsOnT^W} \ }
    \differentialOperatorsOnT^W\mmod \xrightarrow[\ \, \text{\em equivalence}\  \,]{
      \text{\em Morita}\ }
    \differentialOperatorsOnT\mmod^W\xrightarrow{\  \Ind_T^G \ }
      \differentialOperatorsOnG\mmod^G.
    \end{equation}
  \end{Corollary}

   In \cite{BenZviGunninghamSymmetriesofCategoricalRepresentationsandtheQuantumNgoAction}, \cite{ChenAVanishingConjecturetheGLnCase} and \cite{BezrukavnikovDeshpandeVanishingSheavesandtheGeometricWhittakerModel}, the authors considered
   a full monoidal subcategory of $\derivedCategoryofAdGEquivariantDgModules$, sometimes referred to as `{\em very central $D$-modules}', defined as follows.
   Let $p: G\to G/N$ be the natural projection.
  Recall that the Harish-Chandra functor may be defined as the
   push-forward functor
   $\int_p: \derivedCategoryofAdGEquivariantDgModules \to \derivedCategoryOfNEquivariantDModulesOnBaseAffineSpace$ for  equivariant derived categories of $D$-modules.
   An object $M\in \derivedCategoryofAdGEquivariantDgModules$ is called
   very central     if the support of $\int_p M$ is contained in the
    closed $B$-orbit $B/N\subseteq G/N$, 
 ~\cite[Definition 2.12]{BenZviGunninghamSymmetriesofCategoricalRepresentationsandtheQuantumNgoAction}.

 \cref{Main Theorem 2} below
  provides, in particular, a remarkably simple interpretation of the heart
  of the category of very central $D$-modules.

  To explain this,
we introduce the full abelian subcategory
$D(T)\mmod^W_\circ$  of $D(T)\mmod^W$, resp. $D(G)\mmod^G_\circ$ of 
$D(G)\mmod^G$, whose objects $M$ are \textit{induced from their invariants} in the sense 
that the following map induced by the action of the subalgebra
$\Sym(\t)\subseteq D(T)$, resp. $U\g\subseteq D(G)$, 
\begin{equation}\label{circ}
  \Sym(\t)\otimes_{\Sym(\t)^W} M^W\to M,\en
  \text{resp.}\en
  U\g\otimes_{Z\g} M^G\to M,
  \end{equation}
  is a bijection. We remark that, the group $G$ being connected
    and the $D(G)$-module $M$ being {\em strongly} $G$-equivariant, the space 
    $M^G$ above equals
    $\{m\in M\mid \xi_lm = \xi_rm\textcolor{white}{z}  \forall \xi\in\g\}$, where $\xi_l$, resp. $\xi_r$,
    is the right, resp. left, invariant vector field associated to $\xi$.
    Therefore, replacing  in \eqref{circ}
    the algebra $U\LG$ of right invariant differential operators with the algebra $U\LG_{\text{op}}$ of left invariant differential operators leads to an equivalent
    property.

    In  \cref{Essential Image of Ngo Functor} we prove the following theorem that,
    combined with \cref{Theorem Asserting Existence of Ngo Functor as Braided Monoidal Fully Faithful Functor to Center}, confirms
    \cite[Conjecture 2.14]{BenZviGunninghamSymmetriesofCategoricalRepresentationsandtheQuantumNgoAction}:

    \begin{Theorem}\label{Main Theorem 2} For $M\in \differentialOperatorsOnG\mmod^G$, the following are equivalent:

\vi $M \in \differentialOperatorsOnG\mmod^G_\circ$.

\vii The support of the $D$-module  $\int_p M$ is contained in $B/N$, i.e.
  $M$ is very central;

\viii $M$ is contained in the essential image of the functor $\ups$.
\vskip 3pt
\noindent
Moreover, the the following functors  yield     braided monoidal  equivalences
    \begin{equation}\label{equiequi}
           D(T)\mmod^W_\circ\
      \xleftarrow[\operatorname{Obl}^{{\mathfrak J}}_{\differentialOperatorsOnT^W}]{\sim}\
      {\mathfrak J}\mmod\ 
      \xrightarrow[\ \ups\ ]{\sim} \
      \differentialOperatorsOnG\mmod^G_\circ.
      \end{equation}

      In particular, $\differentialOperatorsOnG\mmod^G_\circ$ is a {\em symmetric}
    monoidal category.
  \end{Theorem}

   The left equivalence in \labelcref{equiequi} has been proved in
    \cite[Theorem 1.5.1]{Gin}, which states that $D(T)\mmod^W_\circ$ is exactly the essential image of the functor $\operatorname{Obl}^{{\mathfrak J}}_{\differentialOperatorsOnT^W}$, cf. also \cite{Lo}, \cite{LoRemark}, \cite{GannonClassificationofNondegenerateGCategories}, \cite{GannonUniversalCategoryOandGelfandGraevAction}, \cite{GannonDescentToTheCoarseQuotientForPseudoreflectionAndAffineWeylGroups}, \cite{BezrukavnikovDeshpandeVanishingSheavesandtheGeometricWhittakerModel} for closely related results.
The composed equivalence
      \[\differentialOperatorsOnG\mmod^G_\circ
      \cong D(T)\mmod^W_\circ\]
    has been obtained earlier by Bezrukavnikov and Deshpande, see \cite[Corollary 1.7]{BezrukavnikovDeshpandeVanishingSheavesandtheGeometricWhittakerModel}.
    Our  proof based on the equivalences in
    \eqref{equiequi}  is totally different from the one  in {\em loc. cit.}; 
    in particular, we bypass results from \cite[Section 6]{ChenOnTheConjecturesofBravermanKazhdan} and \cite{BezrukavnikovIonovTolmachovVarshavskyEquivariantDerivedCategoryofaReductiveGroupasaCategoricalCenter} used in 
\cite{BezrukavnikovDeshpandeVanishingSheavesandtheGeometricWhittakerModel}.

  \begin{Remark}
Recall that convolution of $D(G)$-modules is given by the formula
$M_1 \star M_2={M_1\otimes_{U\g} M_2}$, 
cf. \cite[Lemma 2.1]{BezrukavnikovFinkelbergOstrikCharacterDModulesViaDrinfeldCenterofHarishChandraBimodules} and \cref{Monoidality and Tensor Product Subsection};
in particular, the $U\LG$-bimodule structure on $M_1 \star M_2$ is completely determined by the $U\LG$-bimodule structure on $M_1$ and $M_2$. Therefore, by the definition of the functor $\ups$
one has an isomorphism 
\begin{equation}\label{ngo conv}
  M\star \ups(E) \,=\, M\otimes_{U\g} (U\g\otimes_{Z\LG} E)\,\cong\,
  M\otimes_{Z\LG} E,\quad E\in \J\mmod,\ M\in \mathcal{HC},
\end{equation}
of Harish-Chandra bimodules.  This formula, which has no analogue in the
  $\ell$-adic setting,
plays an important role in our arguments.
  \end{Remark}

\subsection{The exactness and vanishing conjectures}\label{Exactness and Vanishing Conjectures Subsection}As an application of our methods, we give proofs of conjectures of Braverman and Kazhdan in the $D$-module setting. Recall that, motivated by Langlands functoriality, Braverman and Kazhdan \cite{BravermanKazhdanGammaSheavesonReductiveGroups} 
defined, for a fixed choice of additive character $\psi: \mathbb{G}_a \to \mathbb{A}^1$
and a finite dimensional representation $G^\vee\to GL(V)$
  of the Langlands dual group $G^\vee$ with $V^{T^{\vee}} = \{0\}$, a certain $\Ad G$-equivariant perverse sheaf $\Phi_V$ on $G$,
    known as the \textit{Bessel sheaf} or $\gamma$\textit{-sheaf}.
    It was shown by Chen \cite{ChenAVanishingConjecturetheGLnCase} that such $\gamma$-sheaves lie in the essential image of the composite of functors in \labelcref{Big Composite for Ngo}. 
  
  Braverman and Kazhdan proposed several conjectures involving 
  $\gamma$-sheaves, in particular, the {\em Vanishing conjecture}
    and
    the {\em Exactness conjecture}. The $D$-module counterpart (of
    the generalization given in \cite{ChenAVanishingConjecturetheGLnCase})
    of the Vanishing conjecture reads
\begin{equation}\label{vanish conj}
        \textsf{Vanishing conjecture}:
   \qquad\; \F \in \derivedCategoryOfWEquivariantDTModulesVANISHINGCONDITION\enspace\Longrightarrow
    \enspace \operatorname{supp}\Big(\int_p \Ind_T^G(\F)^W\Big)\,\subseteq\, B/N.
  \end{equation} where $\derivedCategoryOfWEquivariantDTModulesVANISHINGCONDITION$ is the full triangulated subcategory of the derived category of $W$-equivariant $D$-modules on $T$ whose cohomologies lie in $D(T)\mmod^W_\circ$. The vanishing conjecture (in the original form as well as in the $D$-module setting)
  was proved by Chen \cite{ChenOnTheConjecturesofBravermanKazhdan},
  cf. also   \cite{ChenAVanishingConjecturetheGLnCase}, \cite{ChenOnTheConjecturesofBravermanKazhdan}, Laumon-Letellier \cite{LaumanLetellierNoteonaConjectureofBravermanKazhdan}, and Ng\^o \cite{NgoHankelTransformLanglandsFunctorialityandFunctionalEquationofAutomorphicLFunctions} for partial earlier results and also \cite{BezrukavnikovIonovTolmachovVarshavskyEquivariantDerivedCategoryofaReductiveGroupasaCategoricalCenter}, \cite{BezrukavnikovDeshpandeVanishingSheavesandtheGeometricWhittakerModel}.
A generalization of \eqref{vanish conj} where the Borel $B$ is replaced by an arbitrary parabolic was subgroup
  was found by Chen in \cite{ChenTowardTheDepthZeroStableBernsteinCenterConjecture}. We will use \cref{Main Theorem 2} to
    deduce \eqref{vanish conj}, as well as its parabolic generalization, as a corollary, see Remark \ref{Can Rederive Vanishing Conjecture Remark}.
        Thus, we obtain
  an alternate proof of the  general parabolic vanishing conjecture  in the $D$-module setting.
  \smallskip

 As a simple application of \labelcref{ngo conv},  in \cref{Proof of BK Exactness Conjecture Subsection} we also
 prove the following result, which is a generalization of 
 \cite[Conjecture 1.8]{BravermanKazhdanGammaSheavesonReductiveGroups}
 for the $D$-module analogue $\boldsymbol{\Phi}_V$
 of the $\gamma$-sheaf $\Phi_V$:
  

 \begin{Theorem}\label{Braverman-Kazhdan
     Exactness Conjecture}[\textsf{Exactness Conjecture}]
   Let ${V}$ be a representation of $G^{\vee}$ such that ${V}^{T^{\vee}} = \{0\}$. Then the functor $D(G)\mmod\to D(G)\mmod,\,M \to \PhiTHISGROUPRho{G}\star M$,
   of convolution with    $\PhiTHISGROUPRho{G}$, is exact.
\end{Theorem}

\begin{rems}\label{BK rem}
\vi In \cite[Conjecture 1.8]{BravermanKazhdanGammaSheavesonReductiveGroups}, Braverman and Kazhdan require the representation $V$ be $\sigma$\textit{-positive} for some $\sigma$ (in the sense of \cite{ChenAVanishingConjecturetheGLnCase}) which is stronger than our assumption  that ${V}^{T^{\vee}} = \{0\}$. 
  
   \vii It is likely that one can deduce the Exactness conjecture in the $\ell$-adic setting from the corresponding result in the 
  $D$-module setting 
  using the Riemann-Hilbert correspondence and standard comparison theorems: the proof of the vanishing conjecture
  in \cite{ChenOnTheConjecturesofBravermanKazhdan}
uses these comparison theorems to reduce the general case of the conjecture to the special
  case of  a certain $D(G)$-module which is a character sheaf, or in other words has nilpotent singular support, see \cite[Section 6]{ChenOnTheConjecturesofBravermanKazhdan}. 

  \viii The arguments of \cite{BravermanKazhdanGammaSheavesonReductiveGroups} show that,
  that in the special case  where
  $G = \GL({V})$ and the representation $G^\vee\to GL(V)$ is the identity, one has $\Phi_V\cong\mathrm{tr}^*(\psi)$, the pullback of the exponential $D$-module on
  $\mathbb{A}^1$ via the trace map $\mathrm{tr}: \GL_n \to \mathbb{A}^1$.
  \end{rems}

  \subsection{Transfer of very central $D$-modules}\label{transfer}
  \newcommand{\hv}{{H^\vee}}
  \newcommand{\gc}{{G^\vee}}
 \newcommand{\frj}{{\mathfrak J}}

\newcommand{\bfoo}{{\mathbf O}}
\newcommand{\bfok}{{\mathbf K}}
Let $\bfoo=k[\![t]\!]$, resp. $\bfok=k(\!(t)\!)$, so
$\Gr_{\gc}=\gc(\bfok)/\gc(\bfoo)$ is
the affine  Grassmannian.
The group ${{\mathbb G}_m}$ acts on  these objects by loop rotation.
Convolution of equivariant Borel-Moore homology gives the space
$
H^{BM}_{{{\mathbb G}_m}}(\gc(\bfoo)\backslash\gc(\bfok)/\gc(\bfoo))$ $=H^{BM}_{{{\mathbb G}_m}} (\gc(\bfoo)\backslash\Gr_{\gc})$
an  algebra structure and 
it was shown in \cite{BezrukavnikovFinkelbergEquivariantSatakeCategoryandKostantWhittakerReduction} that this algebra is isomorphic  to an asymptotic counterpart
 $\frj_{G,\hb}$ of the algebra ${\mathfrak J}_G$, cf. Section
 \ref{Monoidal Structure on BiWhit Mod and Monoidality}.
  The algebra embedding
  $H^\hdot_{{{\mathbb G}_m}}(\gc(\bfoo)\backslash\pt)\into
  H^{BM}_{{{\mathbb G}_m}}(\gc(\bfoo)\backslash\Gr_{\gc})$ corresponds via the isomorphism of \cite{BezrukavnikovFinkelbergEquivariantSatakeCategoryandKostantWhittakerReduction}
 and the canonical isomorphism $H^\hdot_{{{\mathbb G}_m}}(\gc(\bfoo)\backslash\pt)\cong k[\g^\vee]^{G^\vee}[\hb]\cong \Zhbarg$,
 to the natural algebra embedding
 $\Zhbarg\into \frj_{G,\hb}$.
  Writing $(\td)_{\hb=1}:=k[\hb]/(\hb-1)\o_{k[\hb]}(\td)$ for the specialization
 at $\hb=1$, one obtains an isomorphism
 $H^{BM}_{{{\mathbb G}_m}\ltimes\gc(\bfoo)} (\Gr_{\gc})_{\hb=1}\cong \frj_G$.

 Now let $G$ and $H$ be a pair of reductive groups
 with Lie algebras $\g$ and $\h$, respectively.
  Below we use subscripts to distinguish objects associated to $G$ from
 those associated to $H$; for example, we write $\frj_G$ and $\frj_H$ for the
 quantizations of the universal centralizers for $G$ and $H$, respectively.

 Let $f:  H^\vee\to \gc$ be a morphism of the Langlands dual groups
 and $df: \h^\vee \to \g^\vee$ the differential of $f$.
  The induced morphism $\hv(\bfoo)\to\gc(\bfoo)$
  gives  an $H^\vee(\bfoo)$-action on $\Gr_{\gc}$ and one has canonical isomorphisms \begin{equation}\label{Borel Moore As Free Module}H^{BM}_{{{\mathbb G}_m}}(\hv(\bfoo)\backslash\gc(\bfok)/\gc(\bfoo))\cong
  \Zhbarh\otimes_{\Zhbarg} H^{BM}_{{{\mathbb G}_m}}(\gc(\bfoo)\backslash\gc(\bfok)/\gc(\bfoo))
  \cong \Zhbarh\otimes_{\Zhbarg} \frj_{G,\hb},\end{equation}
where $Z\h$ is viewed as a $Z\g$-module via the composition of algebra maps
\[Z\g\iso k[\g^\vee]^{G^\vee} \xrightarrow{\ (df)^* \ }  k[\h^\vee]^{H^\vee} \iso Z\h.\] 
The Borel-Moore homology group on the left hand side of \labelcref{Borel Moore As Free Module} above comes equipped with the natural structure of a left module over the algebra
  \[H^{BM}_{{{\mathbb G}_m}}(\hv(\bfoo)\backslash\hv(\bfok)/\hv(\bfoo))\,
    \o_{k[\hb]}\,H^{BM}_{{{\mathbb G}_m}}(\gc(\bfoo)\backslash\gc(\bfok)/\gc(\bfoo))_{op}\ \cong\
  \frj_{H,\hbar}\,
  \o_{k[\hb]}\,(\frj_{G,\hb})_{op},
  \] completely analogously to the construction of the algebra structure of the equivariant Borel-Moore homology of $\Gr_{G^{\vee}}$. Taking specializations at $\hb=1$ we conclude that the map $f$ equips the tensor product $Z\h\otimes_{Z\g} \frj_G$ with the structure of a
  $(\frj_H,\frj_G)$-bimodule, see also \cite[Section 5.4]{GinzburgPointwisePurityDerivedSatakeandSymplecticDuality}.

We define the (abelian categorical) {\em transfer functor}
${\mathscr T}_f: D(G)\mmod^G_\circ\to  D(H)\mmod^H_\circ $ as a composition of
the following  chain of functors
\[
  D(G)\mmod^G_\circ\xrightarrow[\cong]{\  (\ups_G)\inv\  }
  \frj_G\mmod\xrightarrow{\  (Z\h\otimes_{Z\g}\frj_G)\o_{\frj_G}(\td)\ }
  \frj_H\mmod
  \xrightarrow[\cong]{\  \ups_H \  }
  D(H)\mmod^H_\circ.
\]

We sketch a proof of the following Proposition in Section
\ref{Proof of BK Exactness Conjecture Subsection}:

\begin{Proposition}\label{transfer cor} 
  \vi The functor ${\mathscr T}_f$  is monoidal and for any  $M\in D(G)\mmod^G_\circ$ there is a canonical isomorphism
  ${\mathscr T}_f(M)\cong {U\h\otimes_{Z\g}M^G}$
  of Harish-Chandra $(U\h,U\h)$-bimodules.

  \vii There is
  an isomorphism of functors ${\mathscr T}_{g\ccirc f}\cong
  {\mathscr T}_g\ccirc {\mathscr T}_f$, for any morphisms
  $L^\vee \xrightarrow{g} H^\vee \xrightarrow{f} G^\vee$ of reductive groups.

  \viii Let $f: H^\vee\to \GL({V})$ be a finite dimensional representation with trivial zero weight space, i.e. ${V}^{T_H^{\vee}}=0$ for $T_H^{\vee}$ a maximal torus of $H^\vee$. Then $\boldsymbol{\Phi}_{V, H}$ is isomorphic to ${\mathscr T}_f(\mathrm{tr}^*(\psi))$, cf. Remark \ref{BK rem}{\em{(iii)}}.
\end{Proposition}

\begin{rem}\cite{ChenFunctorialTransferforReductiveGroupsandCentralComplexes}, the author defines a constructible analogue of the functor ${\mathscr T}_f$ using combinatorics of root systems. We expect that, using the identification of the asymptotic bi-Whittaker differential operators and the nil-DAHA of \cite{Gin} and \cite{Lo}, the bimodule
  $Z\h\otimes_{Z\g} \frj_G$ can be identified with the $\D$-module analogue of the combinatorial construction given in \cite{ChenFunctorialTransferforReductiveGroupsandCentralComplexes}.
\end{rem}


  \subsection{Outline} In \cref{Preliminaries Section}, we recall the definition of the category $\HarishChandraBimodulesForTHISGROUPDefaultsToG$
  of Harish-Chandra bimodules, and discuss coalgebra type structures on
the ring $D(G)$ of differential operators on $G$
  inherited from the convolution monoidal structure on its category of $D(G)$-modules.
  In \cref{Kostant-Whittaker Reduction Is Monoidal Localization Section}, we construct
the coalgebra structure on $\biWhittakerDifferentialOperatorsonG$ mentioned in the introduction and study
monoidal structure on the categories of $\biWhittakerDifferentialOperatorsonG$-modules
and $\biWhittakerDifferentialOperatorsonG$-comodules
from an algebraic standpoint.
In \cref{Bi-Whittaker Reduction and Biwhittaker Comodules}, we use this to give an explicit relationship between the monoidal categories
  $\HarishChandraBimodulesForTHISGROUPDefaultsToG$
  and the category of comodules for $\biWhittakerDifferentialOperatorsonG$. After proving some preliminary results on the \textit{Miura bimodule} in \cref{Kostant-Whittaker reduction and the Miura bimodule section}, we quantize the Knop-Ng\^o morphism in \cref{Radj for Drinfeld Center via Relative Setting} using a foundational result, \cref{theorem asserting abstract central functor in biWhittaker case}, whose proof is deferred to \cref{Hopf Algebroids and Central Functors Section}. In \cref{Comparison to Parabolic Induction Section} we complete the proofs of our remaining theorems.

\subsection{Acknowledgements} We thank Paul Balmer, David Ben-Zvi, Roman Bezrukavnikov, Tsao-Hsien Chen, Peter Crooks, Pavel Etingof, Joakim F\ae rgeman, Sam Gunningham, Ivan Losev, George Lusztig, David Nadler, Cris Negron, Sam Raskin, Aaron Slipper, Germ\'an Stefanich, and Harold Williams for interesting and useful discussions. We would especially like to thank Tsao-Hsien Chen for comments on an earlier version of this paper and for explaining his vision of central sheaves.

  \section{Monoidal structure on $D(G)$-modules}
\label{Preliminaries Section}

\subsection{Conventions}
\label{conventions}
Given an associative ring $R$ we will  use the following terminology
\begin{defn}\label{Definition of Ring Object Coring Object Symmetric and Left Comodule}   
\vi An $R$\textit{-ring object} is a ring $A$ equipped with  a ring morphism $R\to A$ whose image is not necessarily central.

\vii An $R$\textit{-coring object} is a coring object 
in the monoidal category of $R$-bimodules, that is, an $R$-bimodule
$A$ equipped with  $R$-bimodule morphisms $\De: A\to A\o_R A$ and $\eps: A\to R$,
called coproduct and counit, respectively, which is coassociative in the sense that $(\Delta \otimes \mathrm{id}_A) \circ \Delta = (\mathrm{id}_A \otimes \Delta)$ and counital in the sense that $(\epsilon \otimes \mathrm{id}_A) \circ \Delta = \mathrm{id}_A = (\mathrm{id}_A \otimes \epsilon)\circ \Delta.$
      
\viii  A ring, resp. coring,  object $A$ over
 a   commutative ring $R$ is called an $R$-{\em algebra}, resp. $R$-{\em coalgebra},
      if   $A$ is symmetric as an $R$-bimodule, that is,
      $ra=ar$ for all $r\in R, a\in A$.

 \iv A \textit{left comodule} over an $R$-coring object $A$ is the data of a left
  $R$-module $C$ along with a coaction map
  $\operatorname{coact}: C \to A \otimes_R C$
  of $R$-modules for which $(\Delta \otimes \text{id}_M) \circ \operatorname{coact} = (\text{id}_A \otimes \operatorname{coact}) \circ \operatorname{coact}$ and $(\text{id}_A \otimes \epsilon) \circ \operatorname{coact} = \text{id}_C$, where $\Delta$, resp. $\epsilon$, is the comultiplication, resp. counit of $A$. Right comodules are defined in a  similar way.

  \vv Given an $R$-coring object $A$, and a left, resp. right, $A$-comodule $C_{\ell}$,
  resp.  $C_r$,
 with coaction map $\text{coact}_{\ell}$, resp. $\text{coact}_{r}$, define the \textit{cotensor product}
 \begin{equation}\label{cotensor not}
   C_{r} \otimes_{R}^A C_{\ell} := \big\{\sum\nolimits_i\, c^r_i \otimes c^{\ell}_i \in C_{r} \otimes_{R} C_{\ell} \mid \sum\nolimits_i \,\text{coact}_r(c^r_i) \otimes c^{\ell}_i = \sum\nolimits_i \, c^r_i \otimes \text{coact}_{\ell}(c^{\ell}_i)\big\}.
\end{equation}
\end{defn}

There are natural analogues of the above definitions in the 
graded setting. 
For any $\mathbb{Z}$-graded ring, resp.  coring, object $A$ we let $A\mmod$, resp.
$A\comod$, denote the category of $\mathbb{Z}$-\textbf{graded}  left modules, resp. comodules.

The coordinate ring  $\oo(H)$ of an algebraic group $H$ is a $k$-coalgebra.
  By a linear action (aka representation) of  an algebraic group $H$
  in a not necessarily finite dimensional vector space $V$ we mean
a  coaction $V\to \mathcal{O}(H)\o V$, i.e.
  a left $\oo(H)$-comodule structure on $V$, where  $\oo(H)$ is viewed as a coalgebra.
  Let $\text{Rep}(H)$ denote the symmetric monoidal category  of 
representations of $H$. In what follows, we will use the fact that a linear action of the multiplicative group ${\mathbb G}_m$ is the same thing as a ${\mathbb Z}$-grading without further mention.

For any $k$-algebra $R$, the tensor product  
$A=R\otimes\oo(H)$ is an  $R$-coring object.
In this case, for  $R\otimes\oo(H)$-comodules $C_\ell$ and $C_r$
as in 
  \eqref{cotensor not}, we use simplified notation
  \[C_{r} \otimes_{R}^H C_{\ell}:=(C_{r} \otimes_{R} C_{\ell})^H=C_{r} \otimes_{R}^{A} C_{\ell}.
  \]
 Here, we view both  $C_r$ and $C_\ell$ as representations of the group $H$ so that
  the space on the right may  be identified
  with the invariants
  of the diagonal $H$-action on $C_{r} \otimes_{R} C_{\ell}$.

Recall the following

\begin{Definition}\label{monoidal defn}
\vi A \textit{lax monoidal functor}, resp. \textit{left lax monoidal functor}, between monoidal categories $\mathcal{C}$ and $\mathcal{C}'$ is the data a functor $F: \C \to \mathcal{C}'$, a morphism $\epsilon: \mathbf{1}_{\C'} \to F(\mathbf{1}_{\C})$, resp. $\epsilon': F(\mathbf{1}_{\C}) \to \mathbf{1}_{\C'}$ where $\mathbf{1}_{\C}, \mathbf{1}_{\C'}$ are the respective monoidal units, and morphisms $\mu_{X, Y}: F(X) \otimes F(Y) \to F(X \otimes Y)$, resp.  $\mu'_{X, Y}: F(X \otimes Y) \to F(X) \otimes F(Y)$ natural in $X, Y \in \C$ which satisfy an associativity and unit constraint given explicitly by the diagrams of \cite[Definition 2.1]{MonoidalFunctorDefinition};  

\vii  If $\epsilon$, resp. $\epsilon'$, is an isomorphism and $\mu_{X, Y}$, resp. $\mu'_{X, Y}$, is an isomorphism for all $X, Y \in \C$ we say that this lax monoidal functor is \textit{monoidal}.
\end{Definition}

\begin{Remark}
  What we refer to as a monoidal functor is sometimes referred to as a \textit{strongly monoidal functor}, see \cite[Remark 2.2]{MonoidalFunctorDefinition}.
  \end{Remark}

We will use the following standard result, see e.g. \cite[Proposition 3.84, Proposition 3.85]{AguiarMahajanMonoidalFunctors}:

\begin{Lemma}\label{Adjoint to Monoidal Functor is Appropriately Lax Monoidal}
    The right adjoint, resp. left adjoint, to any monoidal functor $F$ is naturally lax monoidal, left lax monoidal, in such a way that the unit and counit maps are compatible with the monoidal structure. Moreover, if $F$ is braided monoidal then the lax monoidality, resp. left lax monoidality, morphisms are compatible with the braiding.
\end{Lemma}

Let $\D_X$ denote  the sheaf of differential operators on
   a smooth variety $X$ and $D(X)=\Gamma(X,\D_X)$  the algebra of differential operators on $X$.
   We will  use the notations of \cite{HTT} for morphisms of (left) $D$-modules. In particular, for any morphism of smooth algebraic varieties $f$, we let $\int_f$ denote the induced pushforward on the derived category of $D$-modules, resp. $f^{\dagger}$ denote
   the shifted inverse image functor as in \cite[Section 1.5]{HTT}.

\newcommand{\Dt}{{D_{\lhd}}}
\subsection{}\label{Monoidality and Tensor Product Subsection}
Throughout, we fix a connected reductive group $G$ with Lie algebra $\g$.
We will also use the group
$G_\op$, resp. Lie algebra $\g_\op$, obtained by taking an opposite multiplication, resp.
Lie bracket.
We have an isomorphism $G\iso G_\op,\,g\mto g\inv$, resp.  $\g\iso\g_\op,\,\xi\mto -\xi$.
The group $G \times G_{\text{op}}$ acts on $G$ by the formula $(g_1, g_2)g = g_1gg_2$.
We  identify  $U\LG$, resp. $U\LG_{\text{op}}$, with the subalgebra of $D(G)$ 
    of $G_{\mathrm{op}}$-invariant, resp.  $G$-invariant,
    differential operators.
        The algebra
        $U\LG\cap U\LG_{\text{op}}$ of bi-invariant differential operators
    equals $Z\LG$, the center of $U\g$; furthermore, the resulting algebra map
    \begin{equation}\label{Comoment Map for G times Gop with Identifications}U\LG \otimes_{Z\LG} U\LG_{\text{op}} \to {D}(G)\end{equation}
    is known to be injective. Thus, 
    we may (and will) view a  ${D}(G)$-module as a
    $(\U\LG,U\g)$-bimodule.

   Given a $G$-action $\al: G \times X \to X$
   one defines a convolution  operation
   $\mathcal{M}, \mathcal{N}\mapsto\mathcal{M}\star_{_{\D}} \mathcal{N}
   :=\int_{\alpha} \mathcal{M}\boxtimes\mathcal{N}$,
   where $\mathcal{M}$, resp.  $\mathcal{N}$ and
   $\mathcal{M}\star_{_{\D}} \mathcal{N}$, are objects
   of the derived category of $D$-modules on $G$,
   resp. $X$. The morphism $\al$ being smooth and affine,
   it is a straightforward consequence of the definition of $\int_{\alpha}$
that there is a canonical isomorphism of (left) $U\LG$-modules,
     cf. \cite[Lemma 2.1]{BezrukavnikovFinkelbergOstrikCharacterDModulesViaDrinfeldCenterofHarishChandraBimodules},
   \begin{align}\label{bfo}
    &\phi:\,R\Gamma(X, \mathcal{M} \star_{_{\D}} \mathcal{N})\xrightarrow{\ \sim \ }
    R\Gamma(G, \mathcal{M})\,\stackrel{_L}\otimes_{U\g}\, R\Gamma(X, \mathcal{N}),
    \en\text{such that}\nonumber\\
    &\phi(f m)=f_{(0)} m_{(0)}\otimes f_{(1)}m_{(1)},\quad\forall f\in \mathcal{O}(X),\
      m\in R\Gamma(X, \mathcal{M} \star_{_{\D}} \mathcal{N}).
\end{align}
Here, we have used Sweedler notation to write the map
$\alpha^*: \mathcal{O}(X)\to\mathcal{O}(G\times X)=\mathcal{O}(G)\otimes
   \mathcal{O}(X)$  as $f\mapsto f_{(0)}\otimes f_{(1)}$ and $   m = m_{(0)}\otimes m_{(1)}$, and we use the ring morphism \labelcref{Comoment Map for G times Gop with Identifications},
   resp.  (co)moment map $U\LG \to D(X)$ induced by the $G$ action on $X$,
   to view $R\Gamma(G, \mathcal{M})$ as a
   right $U\LG$-module (by left multiplication by the subring $U\LG_{\text{op}} \subseteq D(G)$),
   resp.  $R\Gamma(X, \mathcal{N})$ as a left $U\LG$-module. 
 
In the special case where $G$ acts on itself by left translations, so
 $\alpha =m$ is the multiplication map,
 the convolution operation gives  the derived category of $\D_G$-modules
 a monoidal structure.
 Since $G$ is affine, the latter category is equivalent to   the derived category of  $D(G)$-modules
 via the functor $\Gamma(G,\td)$ and  formula \eqref{bfo}
 shows that  convolution of $\D_G$-modules
 goes via the equivalence to the derived tensor product
  of $U\LG$-bimodules.
Since the tensor product is right exact, the abelian category ${D}(G)\modwithdash$ acquires a
monoidal structure $(\td)\star(\td)=(\td)\otimes_{U\LG}(\td)$.
We will use simplified notation $D=D(G)$ and
define  the  {\em transfer $D$-module} as the convolution
\begin{equation}\label{Representation Theoretic Transfer}
  \Dt:= {D}\star {D}=D\o_{U\g} D,
\end{equation}
of two copies of the rank one free {\em left}
${D}$-module $D$. Thus, $\Dt$ is a left ${D}$-module
and  the action of $D_\op$ on each of the two copies of 
${D}$  by
right multiplicaton induces, by functoriality, 
a  $(D\o D)_\op$-action on $\Dt$ 
by ${D}$-module endomorphisms.
Thus, $\Dt$ acquires the structure of a
$({D}, {D}\o {D})$-bimodule. The convolution in ${D}\mmod$ is given in terms of this object
  by the formula
 \[
   M\star M' := \Dt
     \otimes_{_{{D} \otimes{D}}}\, (M\otimes M'),
\quad M,M'\in {D}\mmod,\]
where $M\otimes M'$ is viewed as a left ${D} \otimes{D}$-module.

   The definition of $\Dt$ agrees with the corresponding
  notion in the theory of $D$-modules, cf. \cite{HTT}.
  Specifically, let $K_X$ denote the dualizing sheaf of a smooth variety $X$.
  Then, using bi-invariant trivializations
  $K_G\cong {\mathcal O}_G$, resp.  $K_{G\times G}\cong {\mathcal O}_{G\times G}$,
  one can check that there is a natural isomorphism
  of $(D, D\otimes D)$-bimodules
  \[\Dt\cong \Gamma\big(G\times G, \, K_{G\times G}
    \otimes_{\mathcal{O}_{G\times G}}\, m^*(K^{\otimes -1}_G\otimes_{\mathcal{O}_G}  \D_G)\big). 
    \]

The monoidal unit of the category $D(G)\mmod$ is the simple 
skyscraper ${D}$-module $\boldsymbol{1}_{{D}\mmod}$ supported 
 at the unit element of $G$.
 The restriction of the
 ${D}$-action to the subalgebra $U\g\otimes_{Z\g} U\g_{\text{op}}$
 identifies  $\boldsymbol{1}_{{D}\mmod}$ with
 the diagonal  $U\LG$-bimodule $(U\LG)_{\text{diag}}$, the 
 monoidal unit of the category $U\g\bimod$.
 Thus, we have the following monoidal functors
\begin{equation}\label{Two fun}
    \dd_G\mmod\xrightarrow[\cong]{\ \Gamma(G,\td)\ } D(G)\mmod
    \xrightarrow{\ \operatorname{Obl}^{D(G)}_{U\g\o U\g_{op}} \ } U\g\bimod,
    \end{equation}
   where the second functor is induced by the map 
\eqref{Comoment Map for G times Gop with Identifications}.
 
\subsection{}
\label{hc bimod}
 Let $k[\hbar]$ denote the graded polynomial algebra with $\hbar$ assigned degree 2. We
write $U_\hb\h$ for the asymptotic enveloping algebra of a Lie algebra $\h$,
resp.  $D_\hbar(X)$ for the algebra of asymptotic differential operators on  a variety  $X$.
Thus,   $U_\hb\h$ and $D_\hb(X)$ are graded $\khb$-algebras defined as 
 the Rees algebras associated with the  PBW filtration  on $U\h$ and  the  filtration on
 $D(X)$ by the order of the differential operator, respectively.

 Let  $\Dh=D_\hb(G)$.  We define convolution of  $\Dh$-modules by
 mimicking constructions above.  To this end, let  $Z_\hb\g$ be the center of $U_\hb\g$.
 The map \eqref{Comoment Map for G times Gop with Identifications} induces
 a  graded $\khb$-algebra embedding $U_\hb\h\o_{Z_\hb\g} U_\hb\g_{op}\to\Dh$.
 We view $\Dh$ as a left, resp. right, $\uhg$-module
 using {\textbf left} multiplication by the elements of the first, resp. second,
 tensor factor of
 $\uhg \o_{Z_\hb\g}\uhg_{\mathrm{op}}$ and put  $\Dht:=\Dh\o_{\Uhbarg}\Dh$.
 This is a $(\Dh,\Dh\o\Dh)$-bimodule 
 and one defines convolution of graded $\Dh$-modules 
  by the formula
 $
   M\star M' =\Dht
     \bigotimes_{{\Dh} \otimes{\Dh}} (M\otimes M')$.
This operation gives    the category $\Dh\mmod$ of graded $\Dh$-modules 
a monoidal structure.

  One has a diagram
  $\g^*\xleftarrow{\mu} T^*G \xrightarrow{\mu_\op} \g^*_\op$,
  of moment maps for the Hamiltonian actions
  on $T^*G$ of $G$ and $G_\op$, respectively.
  Using this diagram and the identification $\g^*_\op\cong\g^*$ of vector spaces,
  one can give $T^*G$  the structure of a \textit{groupoid scheme} in the sense of \cite[Definition 39.13.1]{StacksProject}.  A quantum analogue of this structure is a
  \textit{Hopf algebroid} structure on $\Dh$ in the sense of \cite[
  Definition 4.1]{BohmHopfAlgebroids};
  the structures on $\Dh$ and its category of modules and comodules that we will primarily use are collected in \cref{Differential Operators on G Has Hopf Algebroid Structure and Comodules Has Monoidal Stucture} below.

  To state the proposition we define a `coproduct'  $\Delta: {\Dh}\to \Dht={\Dh}\o_{\uhg} {\Dh}$ and a   `counit'  $\eps: \Dh\to\uhg$ as follows.
  First,  the  left ${\Dh}$-action on $\Dht$
  gives an algebra map
\begin{equation}\label{takeuchi}
  {\Dh} \to \End_{(\Dh\otimes \Dh)_{\text{op}}}(\Dht)
  =\End_{(\Dh \otimes \Dh)_{\text{op}}}({\Dh} \otimes_{\Uhbarg} {\Dh}).
\end{equation}
Second, the  ${D}$-module $\boldsymbol{1}_{{D}\mmod}$
 comes equipped with a natural filtration.
 The Rees object  associated
 with this filtration  gives a graded $\Dh$-module $\boldsymbol{1}_{{\Dh}\mmod}$ 
 which 
 may be identified, as a $(\Uhbarg, \Uhbarg)$-bimodule, with  the diagonal bimodule
 $(\uhg)_{\text{diag}}$.
  Now, the desired map $\Delta: {\Dh}\to \Dht$, resp. $\eps: \Dh\to\uhg$,
  is defined as the map given by the  ${\Dh}$-action on
  the element $1 \otimes 1 \in \Dht={\Dh}\o_{\uhg} {\Dh}$,
  resp. $1\in\uhg=\boldsymbol{1}_{{\Dh}\mmod}$.

Below, we  use various notions introduced
in Definition \ref{Definition of Ring Object Coring Object Symmetric and Left Comodule}    in the special case $R=\Uhbarg$ and $A=\Dh$, where $\Dh$ is viewed as  a 
 $(\uhg,\uhg)$-bimodule
using left multiplication by the elements of the subalgebra
$\uhg \o_{Z_\hb\g}\uhg\subseteq\Dh$.

  \begin{Proposition}\label{Differential Operators on G Has Hopf Algebroid Structure and Comodules Has Monoidal Stucture}
\vi
    The coproduct map $\Delta$ and the counit map $\epsilon$
    equip  ${\Dh}$ with the structure of a (graded) $\uhg$-coring object.
    
    \vii Any (left) ${\Dh}$-comodule $C$ has the canonical structure of a
    $(\uhg,\uhg)$-bimodule where the left action of $\uhg$ comes from
    the  ${\Dh}$-comodule structure. 
        For any ${\Dh}$-comodules $C,C'$,
    the tensor product $C \otimes_{\Uhbarg}C'$ of $(U_\hb\g,\Uhbarg)$-bimodules has
    the natural structure of a ${\Dh}$-comodule.
    Specifically,
    the right $\uhg$-action $C\o \uhg\to C$, resp. $\Dh$-coaction
    $C \otimes_{\Uhbarg} C'\to {\Dh\,\o_{\Uhbarg}\,(C \otimes_{\Uhbarg}C')}$,
    is given by the formula
    \begin{equation}\label{comod1}
           c \o u\mapsto cu := \epsilon(c_{(0)}u)c_{(1)},
                              \quad\text{\em resp.},\quad
    c \otimes c' \mapsto c_{(0)}c'_{(0)} \otimes c_{(1)} \otimes c'_{(1)},
  \end{equation}
  where we use Sweedler notation
  $\operatorname{coact}(c)= c_{(0)}\otimes c_{(1)}$ for the coaction of $\Dh$.

   \viii The operation $(\td)\otimes_{\Uhbarg} (\td)$ 
   equips the category ${\Dh}\comod$ with a monoidal structure such that the
   functor \begin{equation}\label{Monoidal Forgetful Functor}\DhbarG\comod \to U_\hb\g \otimes_{k[\hbar]} U_\hb\g_{op}\mmod\end{equation} induced by the
$U_\hb\g$-bimodule structure on a $\DhbarG$-comodule   is monoidal.
\end{Proposition}

\begin{proof}
The associativity of the convolution
implies that $\Delta$ is coassociative, and the fact that $\Uhbarg
$ is the monoidal unit implies counitality; this shows (i). 
Moreover, it is straighforward to check that, for any $\partial\in \Dh,\,u\in \uhg,\,u'\in \uhg_{\text{op}}$, we have
$\epsilon(u u'\partial)=(u u'\partial)1=u u'\epsilon(\partial)1=
u\epsilon(\partial)u'$. We claim that the algebra and
the coalgebra object structures on $\Dh$ satisfy the axioms of a \textit{left} $\uhg$\textit{-bialgebroid} in the sense of \cite[Definition 3.3]{BohmHopfAlgebroids}.
By our above analysis, this claim amounts to the statement 
that the map $\Delta$ factors through
the so-called \textit{Takeuchi product}, cf. {\em loc cit}
This last statement is an immediate consequence of formula   \eqref{takeuchi}.\footnote{We first learned of this interpretation of the Takeuchi product in the MathOverflow answer \cite{GunninghamWhyShouldTensorofDxModulesOverOXBeDxModule}.}
Alternatively, the fact that $\DhbarG$ is a left $\Uhbarg$-bialgebroid follows from our construction of the monoidal structure on $\DhbarG$-mod in \cref{Monoidality and Tensor Product Subsection} and \cite[Theorem 3.13]{BohmHopfAlgebroids}. Statements (ii) and (iii) then follow from \cite[Theorem 3.18]{BohmHopfAlgebroids}.
\end{proof}

\begin{Remark}
An analogue of \cref{Differential Operators on G Has Hopf Algebroid Structure and Comodules Has Monoidal Stucture} where $\Dh$ is replaced with $D(G)$ and $\Uhbarg$ is replaced with $U\LG$ also holds, with a completely analogous proof. In fact, one can show that the monoidal category of comodules for $D(G)$ is equivalent to the category of \textit{Harish-Chandra bimodules} for $G$. We prove an asymptotic upgrade of this statement in the next section, see \cref{HC forget}.
\end{Remark}

Let $D(G)\ad\g$ be the left ideal of $D(G)$ generated by the vector fields
$\ad\xi,\,\xi\in\g$, where $\ad\xi$ denotes the
image of $\xi\o 1- 1\o\xi\in U\g\o_{Z\g}U\g_{op}$.
 It is clear that the $D(G)$-module $\mathbf{N}=D(G)/D(G)\ad\g$, cf.  \cref{main1},
  is the image of the  diagonal  $U\LG$-bimodule  $(U\LG)_{\text{diag}}$ under the
  functor $D(G)\o_{_{U\g\o U\g_\op}} (\td)$, and
 \begin{equation} \label{N as Fiber of D}\mathbf{N}\,\cong\, D(G)\o_{_{U\g\o U\g_\op}}  (U\LG)_{\text{diag}}.
 \end{equation}
  
  The diagonal bimodule
  is a coalgebra object in $U\g\bimod$ and the functor $D(G)\o_{_{U\g\o U\g_\op}} (\td)$ is a left adjoint of the monoidal
  functor $\mathrm{Obl}^{D(G)}_{U\g\o U\g_{\op}}$ in \labelcref{Two fun}, hence, it is left lax monoidal by \cref{Adjoint to Monoidal Functor is Appropriately Lax Monoidal}. We deduce

 \begin{Corollary}\label{N coalg} The $D(G)$-module
$\mathbf{N}$ has the natural structure of a coalgebra object in the monoidal category $\differentialOperatorsOnG\mmod$.
  \end{Corollary}
   The resulting comultiplication $\mathbf{N}\to \mathbf{N}\o_{U\g}\mathbf{N}$
  can also be described explicitly in terms of
  the  comultiplication $\Delta$ for $\differentialOperatorsOnG$
  considered in \cref{Monoidality and Tensor Product Subsection}. Specifically, it is straightforward to check that
  the composite \[\differentialOperatorsOnG \xrightarrow{\Delta} \differentialOperatorsOnG \otimes_{U\LG} \differentialOperatorsOnG \xrightarrow{} \quantizationOfCentralizersHBAREQUALSONE \otimes_{U\LG} \quantizationOfCentralizersHBAREQUALSONE\]
  sends the left ideal $D(G)\ad\g$ to zero.
  Hence, this composite factors through a map
  $\mathbf{N}\to \mathbf{N}\o_{U\g}\mathbf{N}$. 
  
\subsection{}
Let  $H$ be an algebraic group
with Lie algebra $\h$. 
A  {\em Hamiltonian} $H$-action on  an associative graded  $\khb$-algebra
${ A}$ is the data of an 
$H$-action $\act_{ A}:
H\times { A} \to { A},\, (h,a)\mto \act_{ A} h(a)$ by
graded $\khb$-algebra automorphisms
and
a graded $\khb$-algebra (comoment) map $\mu_A: U_\hb\h\to A$ such that
one has $\hb (d\act_{ A})(\xi)(a)=[\mu_A(\xi),\, a]$ for all $\xi\in\h,a\in { A}$,
where $d\act_{ A}$ is the differential of the $H$-action.
The action of $G$ on 
$U_\hb\g$ induced  by the adjoint action in $\g$ is Hamiltonian
with  the identity comoment map. 

Let ${A}$ be a graded $\khb$-algebra 
with a Hamiltonian $G$-action.  This action gives a
Hamiltonian action on the opposite algebra
$A_{op}$ such that the  comoment map $\mu_{A_{op}}: U_\hb\h\to A_{op}$  sends
$\xi$ to  $-\mu_A(\xi)$ for all $\xi\in \g$.  Given another graded $\khb$-algebra $B$
with Hamiltonian $H$-action,
one makes the diagonal $H$-action
 $h: a \o b\mto \act_{ A}(h)(a)\o \act_{B}(h)(b)$ on ${A} \otimes_{k[\hbar]} B$
 a
 Hamiltonian action by defining the comoment map
 $U_\hb \g \to {A} \otimes_{k[\hbar]} B$ by
the assignment $\xi\mto\mu_A(\xi)\o 1+1\o \mu_B(\xi)$
for all $\xi\in\g$.

Let $A\mmod^H\we$ be the abelian category of 
graded ${ A}$-modules $M$
equipped with an $H$-action $\act_M: H\times M\to M,\,(h,m) \mto \act_M h(m)$
by graded space automorphisms
such that $\act_M h(am)=\act_{ A}(h)(a)\act_M h(m)$, for all
$a\in { A},\,h\in H,\, m\in M$.
Let $A\mmod^H$ be the full subcategory of  $A\mmod^H\we$
whose objects $M$
have the property that for all $\xi\in \h, m\in M$ one has
$\hb (d\act_M)(\xi)(m)=\mu(\xi) m$, where $d\act_M: \h\o M\to M$ is the differential of $\act_M$.
More generally, given a Lie algebra homomorphism $\chi: \h\to k$,
one can define a  full subcategory $A\mmod^{H,\chi}$ of  $A\mmod^H\we$
whose objects $M$
have the property that for all $\xi\in \h, m\in M$ one has
$\hb (d\act_M)(\xi)(m)=(\mu(\xi) - \chi(\xi)) m$.
Objects of $A\mmod^H$, resp.  $A\mmod^{H,\chi}$,
are  usually referred to as  $(A,H)$-modules, resp.
$(A,H,\chi)$-modules.

We will primarily be interested in  the case of the algebra
$A=U_\hb\g \otimes_{k[\hbar]} U_\hb\g_{op}$ equipped with the
diagonal Hamiltonian action of the group $G$ induced by the adjoint
action of $G$ on $\g$.
Let $\mathcal{HC}_{\hb}:=(U_\hb\g \otimes_{k[\hbar]} U_\hb\g_{op})\mmod^G$ be
the category of
(asymptotic) {\em Harish-Chandra bimodules}.
We equip
this category with  monoidal structure
$(\td)\otimes_{\uhg} (\td)$.

\begin{Corollary}\label{HC forget}\vi The forgetful functor $\DhbarG\comod \to \Uhbarg\mmod$ lifts to an equivalence of categories \begin{equation}\label{DComod Is UgmodGweak}\DhbarG\comod \xrightarrow{\sim}  \Uhbarg\mmod^G\we\end{equation} and, moreover, the functor \labelcref{Monoidal Forgetful Functor} lifts to a monoidal equivalence of categories
  \begin{equation}\label{DComod to HChbar}\DhbarG\comod
    \xrightarrow{\sim} \mathcal{HC}_{\hb};\end{equation}
  in particular,
 the monoidal unit  $\boldsymbol{1}_{{D}\mmod}$ of $\Dh\mmod$
 goes to the diagonal $(\Uhbarg, \Uhbarg)$-bimodule  $(\uhg)_{\text{diag}}$.

  \end{Corollary}
  \begin{proof}
For any left comodule $C$ and $c\in C$, we may write the coaction
  $\operatorname{coact}(c)=c_{(0)} \otimes c_{(1)}$
  so that $c_{(0)} \in \mathcal{O}(G)$. This equips $C$ with an $\O(G)$-comodule structure. Conversely, given such a coaction, we may use the same coaction map and the isomorphism $\DhbarG \cong \Uhbarg \otimes \mathcal{O}(G)$ to give the $\DhbarG$-coaction, and these constructions evidently give mutually inverse equivalences of categories. 
 Moreover, for any $C \in \DhbarG\comod$ and any $\xi \in \LG$ we have
  \begin{align*}c \xi := \epsilon(c_{(0)} \xi)c_{(1)} &= \epsilon( \xi c_{(0)} -  \hbar\xi(c_{(0)}))c_{(1)} = \epsilon( \xi c_{(0)})c_{(1)} - \hbar\epsilon( \xi(c_{(0)}))c_{(1)}\\
                                                  &=  \xi\epsilon(c_{(0)})c_{(1)} - \hbar \epsilon( \xi(c_{(0)}))c_{(1)} =  \xi c - \hbar\epsilon( \xi(c_{(0)}))c_{(1)}.
  \end{align*}
  From this equality we see that \begin{equation}\label{HC Bimod Equation}\xi c - c \xi = \hbar \epsilon(\xi(c_{(0)}))c_{(1)}\end{equation} and
so the $U\LG$-bimodule structure in \cref{Differential Operators on G Has Hopf Algebroid Structure and Comodules Has Monoidal Stucture}(ii) is a Harish-Chandra bimodule, and the monoidal structure of the functor \labelcref{Monoidal Forgetful Functor} gives our desired monoidal structure. Finally, forgetting to $\Uhbarg\mmod_{\we}^G$ and applying the inverse functor to \labelcref{DComod Is UgmodGweak} gives the inverse functor to \labelcref{DComod to HChbar}.
\end{proof}

\begin{examp}\label{torus ex} Let $T$ be a torus with Lie algebra $\t$ and ${\mathbb{X}}^*(T)\subset \O(T)$  the weight lattice. Any object $F\in U_\hb\t\mmod^T\we$
  has a canonical weight space decomposition
  $F=\oplus_{\lambda\in {\mathbb{X}}^*(T)} \, F^{(\la)}$ with respect to the $T$-action.
 Here, each $F^{(\la)}$ is a
  $U_\hb\t$-bimodule and we write $u: f\mto u f$ for the left
  action of $u\in U_\hb\t$. Observe that the formula \labelcref{HC Bimod Equation} says that the right action of a given $\xi \in \LT$ on an element $c \in F^{(\la)}$ by the formula $c \xi = \xi c - \hbar(d\lambda)_e(\xi)c$ where $(d\lambda)_e \in \LTd$ is the differential of $\lambda$ at the identity. 
    \end{examp}

The action of  the group $G \times G_{\text{op}}$ on the algebra $\DhbarG$ is Hamiltonian
with moment map given by the composition
$U_\hb\g\o_{\khb} U_\hb\g_{op}\onto U_\hb\g\o_{Z_\hb} U_\hb\g_{op}
\xrightarrow{\ \eqref{Comoment Map for G times Gop with Identifications} \ }
D_\hb$.
In particular, the adjoint action of the group $G$ on itself induces a Hamiltonian
$G$-action on $\Dh$ and the corresponding category 
$\Dh\mmod^{\Ad G}$  of $(\Dh, G)$-modules is the category of 
strongly $\Ad G$-equivariant graded $\Dh$-modules.

The category $\mathcal{HC}_\hb$ contains the full subcategory
$(\uhg\o_{Z_\hb\g}\uhg_{op})\mmod^G$ whose objects are symmetric
as $(Z_\hb\g,Z_\hb\g)$-bimodules.
Restricting the action on an
any $\Ad G$-equivariant $\Dh$-module $M$ gives an object of
$(\uhg\o_{Z_\hb\g}\uhg_{op})\mmod^G$.
Let $\Dh\comod^\De$
be a full subcategory  of $\Dh\comod$ whose objects are symmetric
as $(Z_\hb\g,Z_\hb\g)$-bimodules, cf. \eqref{comod1}.
    Thus, using \cref{HC forget} we
    one obtains the following
diagram of monoidal functors
\beq{D diag of functors} 
\xymatrix{
\Dh\mmod^{\Ad G}\ \ar[r]^<>(0.5){\ \obl \ } &
\   (\uhg\o_{Z_\hb\g}\uhg_{op})\mmod^G \
&&&
\Dh\comod^{\De}\ \ar[lll]_<>(0.5){\ \text{Corollary \ref{HC forget}} \ }_<>(0.5){\cong} 
}
\eeq

\begin{rems}
  \vi We have canonical graded algebra isomorphisms $\uhg/(\hb)\cong\Sym\g\cong
  \O(\g^*)$, resp. $\Dh/(\hb)\cong \O(T^*G)$.
  Thus, an object of $\mathcal{HC}_\hb$ annihilated by $\hb$ is the
  same thing as an $\Ad G$-equivariant {\em symmetric}
  $(\Sym\g, \Sym\g)$-bimodule, equivalently,
  an $\Ad G$-equivariant quasi-coherent sheaf on $\g^*$.
  Similarly,  an object of $\Dh\mmod^{\Ad G}$ annihilated by $\hb$ is the
  same thing as an $\Ad G$-equivariant
  $\O(T^*G)$-module which is symmetric  as a $(\Sym\g, \Sym\g)$-bimodule,
 equivalently, an  $\Ad G$-equivariant quasi-coherent sheaf on 
  the {\em scheme-theoretic zero fiber} of the 
  morphism $\mu-\mu_{op}: T^*G \to \g^*$. This zero fiber is the
  (possibly nonreduced) commuting {\em scheme}.
  \vskip3pt

  \vii The variable $\hb$ can be specialized at $1\in k$ 
in all definitions and results of this section;
in particular, one similarly obtains the monoidal category
$\mathcal{HC}$ of (nonasymptotic) Harish-Chandra $U\g$-bimodules, and are nonasymptotic analogues of \cref{Differential Operators on G Has Hopf Algebroid Structure and Comodules Has Monoidal Stucture},
and Lemma
\ref{HC forget}, and diagram \eqref{D diag of functors}.
\end{rems} 
\newcommand{\overlinepartial}{\overline{\partial}}
\newcommand{\overlinenu}{\overline{\nu}}
\section{Kostant-Whittaker Reduction}\label{Kostant-Whittaker Reduction Is Monoidal Localization Section}

\newcommand{\fs}{{\mathfrak s}}
\newcommand{\coh}{{\operatorname{QCoh}}}
\newcommand{\reg}{_{\text{reg}}}
\newcommand{\gm}{{{\mathbb G}_m}}
From now on, we fix a pair $B, \BorelSubgroup$, of
opposite Borel subgroups of $G$,
so  $T := \BorelSubgroup \cap \oppositeOfBorelSubgroup$ is a maximal torus.
Let $N$ be the unipotent radical 
of $\oppositeOfBorelSubgroup$.
Let $\mathfrak{b}, \bar{\mathfrak{b}}, \mathfrak{t}$,  and $\mathfrak{n}$
denote the respective Lie algebras.
We a fix a nondegenerate
character
$\psi: \LieAlgebraofUnipotentRadicalOfOppositeBorel \to k$ and
a cocharacter $\gamma: \mathbb{G}_m\to T$ such that
$\gamma(z)(\psi)=z^{-2}\cdot\psi,\,\forall z\in\mathbb{G}_m$. 
\subsection{Classical Kostant-Whittaker reduction}
\label{Asymptotic and Classical Variants Subsection}
We identify $\g^*$ with $\g$ via the Killing form $(\td,\td)$. Choose an $\mathfrak{sl}_2$-triple   $e,h,f\in\g$
  such that $h$ lies in $\mathfrak{t}$ and 
$(e,\td)|_{\fn}=\psi$.
Let ${\mathfrak s}\sset \g^*$  be the
image under the isomorphism $\g\iso\g^*,\, \xi\mto (\xi,\td)$
of the Kostant slice through $e$ associated with this triple.
Then, $\fs$ is contained
$(\psi+\LieAlgebraofUnipotentRadicalOfOppositeBorel^\perp)\cap \g^*\reg$.
Let $Z_\fs\to \fs$ be the restriction to $\fs$ of
the group scheme $Z\reg\to\g\reg$ of regular centralizers.

We define a $\bullet$-action of $\gm$ on $\fn^*$
by
$z: \xi\mto z\bullet \xi= z^2\cdot \Ad \gamma(z)(\xi)$.
Then,  $\psi\in\fn^*$ is a $\gm$-fixed point and the affine space  $\psi+ \fn^\perp$, resp. $\fs$,
 is $\gm$-stable. 
 We let $\gm$ act on
 the group scheme $Z_\fs$ as follows
\begin{equation}\label{diag twist}
  z:\, (g,\xi)\mto (\gamma(z)g\gamma(z)\inv,\,z^2\cdot \Ad \gamma(z)(\xi)).
  \end{equation}

 Let  $\coh^{\gm\times G}(\g^*)$, resp. $\coh^{\gm\ltimes Z_\fs}(\fs)$, 
 be the abelian category of equivariant quasi-coherent sheaves
 where the group  $\gm\times G$  acts on $\g^*$ by
 $(z,g): \xi \mto z^2 \Ad g(\xi)$,
 resp. $\gm\ltimes Z_\fs$  acts on $\fs$ by $(z,g): s\mto z\bullet s$.
  Tensor product of $\mathcal O$-modules gives these categories a monoidal structure. One has a well defined  restriction functor   $i^*: \coh^{\gm\times G}(\g^*)\to \coh^{\mathbb{G}_m \ltimes Z_{\mathfrak{s}}}(\fs)$,
even though
 the  closed embedding $i$ is {\em not} $\gm$-equivariant.
Furthermore, this functor  is  exact and monoidal.

There is an alternate, somewhat more canonical construction
of the  functor $i^*$ in terms of `{\em classical Kostant-Whittaker reduction}' that we are going to quantize
in subsequent subsections.
To explain the construction, let
$T^*_\psi(G/N)=G\times_N(e+\fn^\perp)$
be the $\psi$-twisted cotangent bundle on $G/N$. We have a commutative diagram
 \[
   \xymatrix{
       T^*_\psi(G/N)  \ \ar[rrr]^<>(0.5){\ q:\,(g,\xi)\,\mto\,\Ad g(\xi)\  }\ar[drrrr]_<>(0.5){p}&&&\ 
     \g^*_{\text{reg}}\ \ar@{^{(}->}[r]^<>(0.5){\jmath} & \ \g ^*\ \ar[d]_<>(0.4){\varpi}\
& \ \fs\ \ar[l]_<>(0.5){i}\ar[dl]^<>(0.4){\cong}\\
&&&&\ \fc \ &
}
   \]
   Here, $\varpi$ is the coadjoint quotient morphism,
      $\jmath$ is the open embedding,
      $p=\varpi\ccirc \jmath\ccirc q$,
      and the composition $\varpi\ccirc i$ is an isomorphism.

The  $\gm$-action on $\g^*$, resp. $G\times \g^*\reg$, defined by
 $z: \xi\mto z^2\cdot\xi$, resp. by formula \eqref{diag twist},
 induces a  $\gm$-action on $\fc=\g^*/\!/G$, resp. on the universal centralizer
 $J$. The canonical map
$J\to\fc$ intertwines the  $\gm$-actions on $J$ and $\fc$.
We extend the   $\gm$-action on $\fc$ to a $\gm\ltimes J$-action
by letting $J$ act on $\fc$ trivially.
    Finally, we define a $\gm\times G$-action on $T^*_\psi(G/N)$
by letting $G$, resp. $\gm$,  act by $g': (g,\xi)\mto (g'g,\xi)$, resp.
$z: (g,\xi)\mto (g\gamma(z),\, z^2\cdot \Ad \gamma(z)\inv(\xi))$.
The map $q$, resp.  $p, \varpi$, and the isomorphism $\fs\iso\fc$,
commutes with the $G$-actions, resp. $\gm$-actions.
Furthermore, the map $p$ is a $G$-torsor,
\cite[Lemma 3.2.3(iii)]{GinzburgKazhdanDifferentialOperatorsOnBasicAffineSpaceandtheGelfandGraevAction}, and $q$ is a $Z\reg$-torsor.
 
   For any   ${\cf}\in \coh^{\gm\times G}(\g^*\reg)$
   and $\xi\in \g^*\reg$, the stabilizer of $\xi$ in $G$ acts
   in the fiber of $\cf$ at $\xi$. This upgrades to an action
   $p^*J\otimes q^*\cf\to q^*\cf$, giving the sheaf $q^*{\cf}$ 
   a canonical $G\times (\gm\ltimes p^*J)$-equivariant structure,
  where the $\gm$-action on $p^*J$ is induced from the one on $J$.

  It is clear that one has the following mutually inverse monoidal equivalences:

\[
  \xymatrix{
     \coh^{\gm\times G}(\g^*\reg) \ar@<0.5ex>[rr]^<>(0.5){q^*}
   &&
  \coh^{\gm\ltimes (G\times q^*Z\reg)}(T^*_\psi(G/N))\ \ar@<0.5ex>[rrr]^<>(0.5){(p_*(\td))^G}
  \ar@<0.5ex>[ll]^<>(0.5){(q_*(\td))^{Z\reg}}&&&
    \coh^{\gm\ltimes J}(\fc).\ \ar@<0.5ex>[lll]^<>(0.5){p^*}
  }
\]

We consider the  composite  $\kappa_0:\ \coh^{\gm\times G}(\g^*)\xrightarrow{\jmath^*}
\coh^{\gm\times G}(\g^*_{\text{reg}})\xrightarrow{(p_*(\td))^G\ccirc q^*}\coh^{\gm\ltimes J}(\fc)$.

\begin{Corollary}\label{KostantWhittaker In Classical Setting Induces Monoidal Equivalence with G equivariant QCoh on greg}
  \vi The  functor $\kappa_0$
   is an exact monoidal functor  and
  the functor $\jmath_*\big(((q_*p^*)(\td))^{Z_{\text{\em reg}}}\big)$ is a right adjoint of $\kappa_0$.

  \vii The  functor
  $\coh^{\gm\times G}(\g^*)/\Ker(\kappa_0)\to$  $\coh^{\gm\ltimes J}(\fc)$
    induced by $\kappa_0$ is an equivalence and the functor $i^*$ corresponds to $\kappa_0$ under the isomorphism
    $\fs\iso \fc$.\qed
  \end{Corollary}

  \subsection{Kostant-Whittaker reduction}
 \label{psi reduction}
 For any $M\in \Rep(G\times\mathbb{G}_m)$, a $G$-representation equipped
 with a $G$-stable $\mathbb{Z}$-grading $M=\oplus_j\,M_j$,
 there is an an associated  `{\em Kazhdan grading}'  $M= \oplus_{n\in {\mathbb Z}}\  M_{[n]}$
 defined as follows.  
 For each integer $i$, let  $M^{(i)}=\{m\in M\mid \gamma(z)m=z^i\cdot m\; \forall z\in \mathbb{G}_m\}$ be the weight space for the  $\mathbb{G}_m$-action on $M$ obtained by restricting the $G$-action on $M$ along
 the map $\gamma: \mathbb{G}_m\to G$. Thus, we have
 $M=\oplus_{i\in {\mathbb Z}}\,M^{(i)}$.
 The homogeneous 
 components of the Kazhdan grading are defined by the formula
\[
  M_{[n]}:=\oplus_{\{i,j\in {\mathbb Z}\,\mid \,  n=i+j\}} \ M^{(i)}\cap M_j,\qquad
  n\in \mathbb{Z}.
\]

Let $\fn^\psi$ be the image of the map
$\fn\to\uhg,\,\xi\mto\xi - \psi(\xi)$.
This is an $\Ad N$-stable graded Lie subalgebra
of $\uhg$ where $\uhg$ is equipped with the Kazhdan grading  associated
with the adjoint action of $G$.
Hence,   $\fn^\psi M$, resp. $M\fn^\psi$, 
is an $N$-stable graded subspace of $M$
for any $M\in \uhg\mmod^G\we$, resp. $M\in \uhg_{op}\mmod^G\we$.

Let ${ A}$ be a graded $k[\hb]$-algebra with a Hamiltonian $G$-action
with comoment map $\mu$. We equip
$A\o_{\khb} A_{op}$ with the diagonal Hamiltonian action of $G$
and consider the category $(A\o_{\khb} A_{op})\mmod^G$.
Any object $M$ of this category  is in particular a $(\uhg,\uhg)$-bimodule and
a $G$-representation with a $G$-stable
$\mathbb{Z}$-grading, hence, $M$  has an associated   Kazhdan grading.
It follows from the above that  the space  $M/M\fn^\psi$, resp. $M/\fn^\psi M$, of $\fn^\psi$-coinvariants
acquires the structure of  an  $({ A}, N,\psi)$-module,
resp. $({ A}_{op}, N,\psi)$-module.  Since $\psi([\LieAlgebraofUnipotentRadicalOfOppositeBorel, \LieAlgebraofUnipotentRadicalOfOppositeBorel]) = 0$, 
the assignment $\xi \mapsto \psi(\xi)$ also gives 
a character of an opposite Lie algebra $\LieAlgebraofUnipotentRadicalOfOppositeBorel_{\text{op}}$,
to be denoted by $\psi$ as well. Hence, the image of the map
$\fn_\op\to \Uhbarg_\op,\, \xi\mto \xi-\psi(\xi)$, is a Lie subalgebra to be denoted by
$\fn^\psi_r$.
Thus, in the special case 
of the diagonal $A$-bimodule  $M=A$,
one has the  Kazhdan grading on the algebra $A$,
the image of the composite
$\fn^\psi\into\uhg\xrightarrow{\mu} { A}$ generates  a left, resp. right,
homogeneous ideal 
${ A} \fn^\psi$, resp.  $\fn^\psi{ A}$,
and the quotient  $E^\psi_\ell ={ A}/{ A}\fn^\psi$, resp. $E^\psi_r ={ A}/\fn^\psi{ A}$,
has the structure of  an  $({ A}, N,\psi)$-module,
resp. $({ A}_{op}, N, \psi)$-module.
For any $F\in A\mmod^{N,\psi}$, resp.
  $F\in A_{op}\mmod^{N_{op},\psi}$,
and there is a canonical isomorphism
$\Hom_{A\mmod^{N,\psi}}(E^{\psi}_\ell, F)\cong F^N$, resp.
$\Hom_{A_{op}\mmod^{N_{op},\psi}}(E^{\psi}_r, F)\cong F^{N_{op}}$.
In particular,  we have
\begin{equation}\label{Kostant Whittaker as Endomorphisms}
  (\End_{{ A}\mmod^{N,\psi}} E^{\psi}_\ell)_\op\cong ({ A}/{ A}\fn^\psi)^N,\enspace\text{resp. }\enspace
  \End_{{ A}_{op}\mmod^{N,\psi}} E^{\psi}_r \cong ({ A}/\fn^\psi{ A})^N.
\end{equation}
Composition of endomorphisms provides, via the above isomorphisms,
algebra structures on the objects appearing on the right hand side of the isomorphisms.

  The Kazhdan grading and the standard (doubled) grading on $\uhg$ induce
the same grading on $\zhg$ and
the composite
$\zhg\into\uhg\to A$ induces an algebra map $\zhg\to({ A}/{ A}\fn^\psi)^N$,
resp. $\zhg\to({ A}/\fn^\psi{ A})^N$,
that respects each of these gradings. 

In the special case $A=\uhg$
these  maps are algebra isomorphisms due to Kostant
\cite{KostantonWhittakerVectorsandRepresentationTheory}:
\begin{equation}
  \label{Unit Map for Monoidality of Asymptotic Kostant Whittaker}
  Z_\hb\g \cong(\Uhbarg/\Uhbarg \cdot
  \fn^\psi)^{{\unipotentRadicalOfOppositeBorel}} \cong (\Uhbarg/\fn^\psi\Uhbarg)^{{\unipotentRadicalOfOppositeBorel}}
\end{equation}

For any $A$ as above and a right, resp. left,
 ${ A}$-module $M$,  the coinvariants 
  \begin{equation}\label{Epsi}
    M/M\fn^{\psi}\cong M \otimes_{{ A}} E^{\psi}_{\ell},\enspace\text{resp. }\enspace
    M/\fn^\psi M\cong E^\psi_r \o_{ A} M
  \end{equation}
acquire the natural structure of a right module over
$({ A}/{ A}\fn^\psi)^N$, respectively left module for $({ A}/\fn^\psi{ A})^N$,
via the action on $E^{\psi}_{\ell}$, resp. $E^\psi_r$.
Similarly, for any
$({ A}\o_{\khb} A_{op},\,G)$-module $M$
the space
  \[
  \vkap_\hb(M):=(M/M\fn^\psi)^{N}=\Hom_{{ A}\mmod^N}(E^\psi_\ell,\, M/M\fn^\psi)\]
    acquires the structure of   a graded left module over the algebra
    ~$({ A}/{ A}\fn^\psi)^N\otimes_{\khb}({ A}/{ A}\fn^\psi)^N_{{op}} $.
        Observe that if $M={ A}$ is the diagonal bimodule for ${ A}$ then
  $\vkap_\hb(M)$
  is the diagonal  bimodule for $({ A}/{ A}\fn^\psi)^N$,
  so we will write $({ A}/{ A}\fn^\psi)^N=\vkap_\hb({ A})$.


In the case $A=\uhg$, part (i) of the followoing proposition is known as \textit{Skryabin's equivalence}; part (ii) is implicit in
  \cite[Section 4]{BezrukavnikovFinkelbergEquivariantSatakeCategoryandKostantWhittakerReduction} and proved in \cite[Theorem 4.1.4(iii)]{GinzburgHarishChandraBimodulesforQuantizedSlodowySlices}. 
 
\begin{Proposition}\label{Asymptotic Kostant-Whittaker Reduction Is Monoidal and Exact}
Let $A$ be a graded $\khb$-algebra  equipped with a Hamiltonian $G$-action. 

\vi The functors \[A/A\mathfrak{n}^{\psi}\otimes_{\kappa(A)} (-) : \kappa(A)\mmod \xleftrightarrow{\sim}  A\mmod^{N, \psi}: (-)^{\mathfrak{n}^{\psi}}\] give mutually inverse exact equivalences of abelian categories.

\vii The assignment $M\mapsto \vkap_\hb(M)$
yields an exact monoidal functor
$(A \o_{\khb} A_{op})\mmod^G\to\vkap_\hb({ A})\bimod$. Moreover, $\vkap_\hb$ commutes with arbitrary direct sums.
 \end{Proposition}

 \begin{proof} Part (i) is proved completely analogously to the non-asymptotic case of \cite[Proposition 3.4]{Gin}. 
 
 We now show (ii). Restricting the action along the algebra map
   $\mu_A\o \mu_{A_{op}}:  \uhg \o_{\khb} \uhg_{op}\to A \o_{\khb} A_{op}$
   gives an exact functor $(A \o_{\khb} A_{op})\mmod^G\to\mathcal{HC}_\hb$.
   Thus, to prove the exactness statement we may assume without loss of generality
   that $A=\uhg$, in which case  the exactness follows from the results discussed in \cite[Section 3.1]{Gin}, cf. also
\cite[Theorem 4.1.4(iii)]{GinzburgHarishChandraBimodulesforQuantizedSlodowySlices}.

To prove monoidality, we must construct for all
$M,M'\in ({ A}, G)\bimod$ a functorial isomorphism
 $\kostantWhittakerHBARtoZGBIMOD(M) \otimes_{\vkap_{\hb}({A})}
    \kostantWhittakerHBARtoZGBIMOD(M') \xrightarrow{\sim}
    \kostantWhittakerHBARtoZGBIMOD(M \otimes_{{A}} M')$.
    To this end, put $\bar {A}:=({ A}/{ A}\fn^\psi)^N$ and $E:=E^\psi_\ell=
   {A}/{A}\fn^\psi$.
   Observe that the left action of $\fn^\psi$ on
   $M/M\fn^\psi$ is locally nilpotent. It follows that the natural map $E
  \otimes_{\bar{A}}  \kostantWhittakerHBARtoZGBIMOD(M)
  \to M \otimes_{{A}} E$ is an isomorphism,
  by a version of Skryabin's isomorphism, \cite[Proposition 3.1.4(i)]{Gin}.  
  We deduce  a chain of    isomorphisms
  \begin{align*}
                                          \big(E
    \otimes_{\bar {A}} \kostantWhittakerHBARtoZGBIMOD(M)\big) \otimes_{\bar{A}}
                                         \kostantWhittakerHBARtoZGBIMOD(M') 
 &   \xrightarrow{\sim}\
                                         (M \otimes_{{A}}
                                         E)\  \otimes_{\bar{A}} \kostantWhittakerHBARtoZGBIMOD(M')\\
                                       &\xrightarrow{\sim}
                                         M \otimes_{{A}} \big(E\otimes_{\bar{A}}
                                         \kostantWhittakerHBARtoZGBIMOD(M')\big)
                                         \xrightarrow{\sim}\ (M \otimes_{{A}} M') \otimes_{{A}}
                                         E.  \end{align*}

                                       The isomorphism of the proposition is given by applying $(\td)^{\unipotentRadicalOfOppositeBorel}$ to the composite of the above isomorphisms using that
                                       the functor $(E\o_{\bar{A}} \td)^N$, resp.
                                       $(\td\o_{A} E)^N$, is the identity. Finally, $\vkap_\hb$ commutes with arbitrary direct sums since it is the composite of a tensor product and an equivalence of categories.
 \end{proof}


There are also completely analogous constructions in the specialized
setting where $\hb=1$.



\subsection{}
\label{Monoidal Structure on BiWhit Mod and Monoidality}
We have the Lie subalgebra $\fn^\psi_\ell:=\fn^\psi\subset\Uhbarg\subset \Dh$,
resp. $\fn^\psi_r\subset\Uhbarg_\op\subset \Dh$,
and the  Kostant-Whittaker reduction of $\DhbarG$:
\[\Dpsilhbar :=  (\DhbarG/\DhbarG\LieAlgebraofUnipotentRadicalOfOppositeBorelPSIl)^{\unipotentRadicalOfOppositeBorel},
  \enspace\text{resp.}\enspace
  \Dpsirhbar := (\DhbarG/\LieAlgebraofUnipotentRadicalOfOppositeBorelPSIr\DhbarG)^{\unipotentRadicalOfOppositeBorel_{\text{op}}}.\]

The residual action of the group $G_\op$, resp. $G$,  on  $\Dpsil$, resp.  $\Dpsir$,
is Hamiltonian.
Therefore,
one can apply Kostant-Whittaker reduction to the algebra $\Dpsirhbar$,
  resp.  $\Dpsilhbar$ , using the Lie subalgebra
  $\fn^\psi_\ell\sset \Uhbarg\sset \Dpsirhbar$, resp. $\fn^\psi_r\sset \Uhbarg_\op\sset \Dpsilhbar$.
  This way one obtains the following algebras and algebra maps:
\begin{equation}\label{Algebra Isos For Different Descriptions of biWhittaker}
  (\Dpsirhbar/\Dpsirhbar\LieAlgebraofUnipotentRadicalOfOppositeBorelPSIl)^{\unipotentRadicalOfOppositeBorel}
  \xrightarrow{} (\Dh/(\fn^\psi_r \Dh+\Dh\fn^\psi_\ell)) ^{\unipotentRadicalOfOppositeBorel \times \unipotentRadicalOfOppositeBorel_{\text{op}}} \xleftarrow{} (\Dpsilhbar/\LieAlgebraofUnipotentRadicalOfOppositeBorelPSIr\Dpsilhbar)^{\unipotentRadicalOfOppositeBorel_{\text{op}}}.
\end{equation}
By  Skryabin's equivalence, the functor $(\td)^{\unipotentRadicalOfOppositeBorel}$ is exact on the subcategory of $\Uhbarg$-modules for which $\LieAlgebraofUnipotentRadicalOfOppositeBorelPSI$ acts locally nilpotently.
Since any exact functor on an abelian category commutes with all finite limits and colimits (see for example \cite[Lemma 12.7.2]{StacksProject}) it follows that the
both maps in \eqref{Algebra Isos For Different Descriptions of biWhittaker} are algebra isomorphisms.\footnote{Observe that $\Dpsirhbar/\Dpsirhbar\LieAlgebraofUnipotentRadicalOfOppositeBorelPSIl$ is the colimit $\mathrm{colim}(\Dpsirhbar \otimes \Dpsirhbar\fn^\psi_\ell \rightrightarrows{} \Dpsirhbar)$ of the zero map and the action map, and $\Dpsirhbar$ itself can be obtained as the limit $\lim(D(G)/\fn^\psi_rD(G) \rightrightarrows{} \mathcal{O}(N) \otimes D(G)/\fn^\psi_rD(G))$ of the coaction map and the identity map. Taking instead the corresponding colimit of the limit, we see that the fact that exact functors commute with all limits and colimits implies that the canonical map $  (\Dpsirhbar/\Dpsirhbar\LieAlgebraofUnipotentRadicalOfOppositeBorelPSIl)
  \xrightarrow{} (\Dh/(\fn^\psi_r \Dh+\Dh\fn^\psi_\ell)) ^{\unipotentRadicalOfOppositeBorel_{\text{op}}}$ from the colimit of the limit to the limit of the colimit is an isomorphism. Taking $N$-invariants we obtain our above map is an isomorphism.}
We let $\biWhittakerDifferentialOperatorsonG_\hb$ denote the resulting algebra, referred to as the algebra
of (asymptotic) \textit{bi-Whittaker differential operators}.

It is clear that the maps $\zhg\to (\Dpsirhbar/\Dpsirhbar\LieAlgebraofUnipotentRadicalOfOppositeBorelPSIl)^{\unipotentRadicalOfOppositeBorel}$ and
$\zhg\to(\Dpsilhbar/\LieAlgebraofUnipotentRadicalOfOppositeBorelPSIr\Dpsilhbar)^{\unipotentRadicalOfOppositeBorel_{\text{op}}}$ correspond to each other
via the identification given by \eqref{Algebra Isos For Different Descriptions of biWhittaker}, so one obtains an algebra homomorphism $u_\J: \zhg\to \jh$.
This homomorphism is a quantization of the algebra map
$\O(\fc)\to \O(J)$ induced by the canonical morphism $J\to\fc$.

We are now going to construct a  $Z\g$-coalgebra structure on $\J$ that
quantizes the $\O(\fc)$-coalgebra structure  on $\O(J)$ induced by the group
scheme structure on $J$.   We will then  consider the monoidal
category of $\J$-comodules and use the functor $\vkap_\hb$ to construct 
a commutative diagram of monoidal functors, cf. \cref{Asymptotic Kostant Whittaker Reduction is Monoidal Localization}:
\beq{diag of functors}
\xymatrix{
  \Dh\mmod^{\Ad G}\  
  \ar[rr]^<>(0.5){\operatorname{Obl}} \ar[d]^<>(0.5){\vkap_\J}&& \  (\uhg\o_{Z_\hb}\uhg_{op})\mmod^{G}\
  \ar@{^{(}->}[r]^<>(0.5){\ \eqref{D diag of functors} \ }
  \ar[d]^<>(0.5){\kap_\hb} &
\    \Dh\comod\  \ar[d]^<>(0.5){\kappa_\hb} \\
\    \J\mmod\    \ar[r]^<>(0.5){\operatorname{Obl}}  & \  Z_\hb\g\mmod\  & \  \J\comod^\De\   \ar[l]_<>(0.5){\operatorname{Obl}} \ar@{^{(}->}[r]&\ \J\comod\  
   }
 \eeq
 where $\J\comod^\De$ is the full subcategory of $\J\comod$ whose objects are
 symmetric as $(Z_\hb\g, Z_\hb\g)$-bimodules, cf.
 Proposition \ref{Asymptotic Enhancement of All Preliminary Results}(ii).

%

We now describe the functor $\vkap_\J$ on the left more explicitly. Observe that, from the discussion in \cref{psi reduction}, since the action map $\DhbarG \otimes M \to M$ is a map of $\DhbarG$-modules (say) for any $M \in \DhbarG\mmod$, the coinvariants $M/\fn^\psi_rM \cong \DhbarG/\fn^\psi_r\DhbarG \otimes_{\DhbarG} M$ acquires a module structure for $\Dpsirhbar$ and, moreover, if $M \in \DhbarG\mmod^G$ then one readily checks that $\fn^\psi_{\ell}$ acts locally nilpotently on $M/\fn^\psi_rM$. Thus we may define $\vkap_\J$ as the composite
\beq{vkap mod}
\vkap_\J:\  D_\hb\mmod^G\  \xrightarrow{\  M \mto 
  M/\fn^\psi_r M \ }  \ \Dpsirhbar\mmod^{N, \psi}
\
\xrightarrow{\  (\td)^{N} \ } \
\J_\hb\mmod
\eeq  of the coinvariants and the invariants for the adjoint action of $N$.

As in \cref{HC forget}, the $\Dh$-module $\bone_{\Dh\mmod}$ supported at the unit of $G$ can be identified with $\uhg$, the diagonal $(\uhg,\uhg)$-bimodule.
Hence we find
$\vkap_\J(\bone_{\Dh\mmod})=(\uhg/\fn^\psi \uhg)^N=Z_\hb\g$.
  Thus, the $  Z_\hb\g$-action on itself by multiplication has a canonical
  extension to a $\J_\hb$-action. In particular, the action on the
  element $1\in Z_\hb\g$ gives a $\J_\hb$-module map $\eps: \J_\hb\to Z_\hb\g$,
  which is a quantization of the
  restriction map $\O(J)\to \O(\fc)$ induced by the unit
  section $\fc\into J$ of the group scheme $J$.

\begin{Proposition}\label{Asymptotic Enhancement of All Preliminary Results}
  \vi There is a left $\biWhittakerDifferentialOperatorsHBAR$-module structure on the tensor product \[\jht := \biWhittakerDifferentialOperatorsHBAR \otimes_{\Zhbarg} \biWhittakerDifferentialOperatorsHBAR\] of left $\Zhbarg$-modules which commutes with the obvious right module structure for the ring $\biWhittakerDifferentialOperatorsHBAR \otimes \biWhittakerDifferentialOperatorsHBAR$.
  The maps
    $\Delta: \biWhittakerDifferentialOperatorsHBAR \to \jht,\,
    u\mto u(1\o 1)$ and $\eps$ equip $\J_\hb$ with a graded $\Zhbarg$-coalgebra structure.
Furthermore, the operation
       \begin{equation}\label{Tensor for BiWhit Mod Via Transfer}
         E_1,E_2\mto E_1\star E_2:=\jht\o_{\jh\o\jh}(E_1\o E_2) \ (\cong E_1\o_\zhg E_2)
       \end{equation}
      gives     the category $\jh\mmod$ of graded left $\jh$-modules
      a  monoidal structure such that the forgetful functor $\jh\mmod \to \zhg\mmod$ is monoidal. \vskip 3pt

      \vii The $\zhg$-module structure on any left $\jh$-comodule
      (with the coalgebra structure on $\jh$ from {\em (i)})  can be extended to a $\zhg$-bimodule structure. Moreover, the category of $\biWhittakerDifferentialOperatorsHBAR$-comodules can be equipped with the canonical monoidal structure such that the forgetful functor $\jh\comod\to\zhg\bimod$ is monoidal. \vskip 3pt

    \viii If $M \in \DhbarG\mmod$, $M' \in \Dpsirhbar\mmod$, and $E \in \biWhittakerDifferentialOperatorsHBAR\mmod$, then there is a natural (graded) $\Dpsirhbar$-module structure on $M \otimes_{\Uhbarg} M'$ and $M' \otimes_{\Zhbarg} E$. Moreover, this assignment equips $\Dpsirhbar\mmod$ with the structure of a $(\DhbarG\mmod, \biWhittakerDifferentialOperatorsHBAR\mmod)$-bimodule category. 
      
    \iv One can equip $\Dpsirhbar$, viewed   as a $\uhg$-module, resp. $\zhg$-module, by left multiplication by the subalgebra $\uhg$, resp. $\zhg$, with the structure of a left $\Dh$-comodule, resp. right $\jh$-comodule. Combined together, this makes $\Dpsirhbar$ a $(\Dh,\jh)$-bicomodule.
\end{Proposition}

\begin{proof} In \cref{Monoidality and Tensor Product Subsection}, we equipped $\Dht := \DhbarG \otimes_{\Uhbarg} \DhbarG$ with a $(\DhbarG, \DhbarG \otimes\DhbarG)$-bimodule structure. The coinvariants $\DhbarG \otimes_{U\LG} (\DhbarG/\DhbarG\LieAlgebraofUnipotentRadicalOfOppositeBorelPSIl)$ obtains an induced $(\DhbarG, \DhbarG \otimes\Dpsilhbar)$-bimodule structure, where the right action of any $\overline{\partial} \in \Dpsilhbar$ is given by the action of any coset representative. To avoid excessive repetition, if an action of $\Dpsirhbar$ (respectively $\Dpsilhbar$ or $\biWhittakerDifferentialOperatorsHBAR$) is obtained from the action of any coset representative in $\DhbarG$, respectively $\DhbarG$ or $\Dpsirhbar$, for the remainder of this proof we will simply say the action is \textit{induced} from the associated $\DhbarG$, respectively $\DhbarG$ or $\Dpsirhbar$, module. We will use similar terminology for bimodules, so, summarizing our discussion above: the $(\DhbarG, \DhbarG \otimes\DhbarG)$-bimodule structure on $\DhbarG \otimes_{\Uhbarg} \DhbarG$ induces a $(\DhbarG, \DhbarG \otimes\Dpsilhbar)$-bimodule structure on 
$\DhbarG \otimes_{\Uhbarg} \DhbarG/\DhbarG\LieAlgebraofUnipotentRadicalOfOppositeBorelPSIl$. From this, one can similarly equip the invariants \[(\DhbarG \otimes_{U\LG} \DhbarG/\DhbarG\LieAlgebraofUnipotentRadicalOfOppositeBorelPSIl)^{\LieAlgebraofUnipotentRadicalOfOppositeBorelPSIr \otimes 1}\] for the right action of $\LieAlgebraofUnipotentRadicalOfOppositeBorelPSIr \subset \DhbarG$ with an induced $(\DhbarG, \Dpsirhbar \otimes\Dpsilhbar)$-bimodule structure. 

Next, observe that the $N_{\mathrm{op}}$-action on $G$ given by the formula $n\cdot g := gn$ equips $\DhbarG$ with the structure of an asymptotic Harish-Chandra bimodule for the group $N_{\mathrm{op}}$ such that the cotensor product $\DhbarG \otimes_{U\LG}^N \DhbarG/\DhbarG\LieAlgebraofUnipotentRadicalOfOppositeBorelPSIl$ (defined in \eqref{cotensor not}) has the property that \[(\DhbarG \otimes_{U\LG} \DhbarG/\DhbarG\LieAlgebraofUnipotentRadicalOfOppositeBorelPSIl)^{\LieAlgebraofUnipotentRadicalOfOppositeBorelPSIr \otimes 1} = \DhbarG \otimes_{U\LG}^N \DhbarG/\DhbarG\LieAlgebraofUnipotentRadicalOfOppositeBorelPSIl.\] Therefore, the invariants $\DhbarG \otimes_{U\LG}^{\unipotentRadicalOfOppositeBorel} (\DhbarG/\DhbarG\LieAlgebraofUnipotentRadicalOfOppositeBorelPSIl)$ acquires the structure of a $(\DhbarG, \Dpsirhbar \otimes \Dpsilhbar)$-bimodule. The monoidality of Kostant-Whittaker reduction for Harish-Chandra bimodules (\cref{Asymptotic Kostant-Whittaker Reduction Is Monoidal and Exact}) gives a canonical isomorphism \[\Dpsirhbar \otimes_{\Zhbarg} \Dpsilhbar \xrightarrow{\sim} \DhbarG \otimes_{\Uhbarg}^{\unipotentRadicalOfOppositeBorel} (\DhbarG/\DhbarG\LieAlgebraofUnipotentRadicalOfOppositeBorelPSIl)\] and, using this isomorphism, we transfer our above bimodule structure to a $(\DhbarG, \Dpsirhbar \otimes_{\Zhbarg} \Dpsilhbar)$-bimodule structure on $\Dpsirhbar \otimes_{\Zhbarg} \Dpsilhbar$. From this bimodule structure, we obtain an induced $(\Dpsirhbar, \Dpsirhbar \otimes \Dpsilhbar)$-bimodule structure on $\LieAlgebraofUnipotentRadicalOfOppositeBorelPSIl\Dpsirhbar\backslash \Dpsirhbar \otimes_{\Zhbarg} \Dpsilhbar$. From this, we obtain an induced $(\Dpsirhbar, \Dpsirhbar \otimes \biWhittakerDifferentialOperatorsHBAR)$-bimodule structure on the invariants \begin{equation}\label{Transfer for Left Biwhit Action on Dpsil Mod}
(\Dpsirhbar \otimes_{\Zhbarg} \Dpsilhbar/\Dpsilhbar\LieAlgebraofUnipotentRadicalOfOppositeBorelPSIr)^{\LieAlgebraofUnipotentRadicalOfOppositeBorelPSIr} \xleftarrow{\sim} \Dpsirhbar \otimes_{\Zhbarg} \biWhittakerDifferentialOperatorsHBAR
\end{equation} for left multiplication by $\LieAlgebraofUnipotentRadicalOfOppositeBorelPSIr \subseteq \DhbarG$. Now, if $M' \in \Dpsirhbar\mmod$ and $E \in \biWhittakerDifferentialOperatorsHBAR\mmod$, observe that the module structures give rise to an isomorphism \[(\Dpsirhbar \otimes_{\Zhbarg} \biWhittakerDifferentialOperatorsHBAR) \otimes_{\Dpsirhbar \otimes \biWhittakerDifferentialOperatorsHBAR} (M' \otimes E) \xrightarrow{\sim} M' \otimes_{\Zhbarg} E\] and so, using this isomorphism and the left $\Dpsirhbar$-module structure on $\Dpsirhbar \otimes_{\Zhbarg} \biWhittakerDifferentialOperatorsHBAR$, we equip $M' \otimes_{\Zhbarg} E$ with a $\Dpsirhbar$-module structure. 

One can trace through a completely analogous construction starting with the $(\DhbarG, \DhbarG \otimes \DhbarG \otimes \DhbarG)$-bimodule $\DhbarG \otimes_{\Uhbarg} \DhbarG \otimes_{\Uhbarg} \DhbarG$ to obtain a $(\Dpsirhbar, \DhbarG \otimes \Dpsirhbar \otimes \biWhittakerDifferentialOperatorsHBAR)$-bimodule structure on $\DhbarG \otimes_{\Uhbarg} \Dpsirhbar \otimes_{\Zhbarg} \biWhittakerDifferentialOperatorsHBAR$, and construct an assignment \begin{equation}\label{Assignment for Bimodule Category Action}(M, M', E) \mapsto M \otimes_{\Uhbarg} M' \otimes_{\Zhbarg} E \in \Dpsirhbar\mmod\end{equation} for any $M \in \DhbarG\mmod, M' \in \Dpsirhbar\mmod$, and $E \in \biWhittakerDifferentialOperatorsHBAR\mmod$. We also observe that the $\DhbarG$-module structure on $\Uhbarg$ induces a $\Dpsirhbar$-module structure on $\Uhbarg/\Uhbarg\cdot \LieAlgebraofUnipotentRadicalOfOppositeBorel^{\psi}$ and, from this, we obtain an induced $\biWhittakerDifferentialOperatorsHBAR$-module structure on $\Zhbarg \xrightarrow{\sim} \kostantWhittakerHBARtoZGBIMOD(\Uhbarg)$. From this and the assignment in \labelcref{Assignment for Bimodule Category Action}, we obtain that \[M \otimes_{\Uhbarg} M' \cong M \otimes_{\Uhbarg} M' \otimes_{\Zhbarg} \Zhbarg \in \Dpsirhbar\mmod\] for any $M \in \DhbarG\mmod$ and $M' \in \Dpsirhbar\mmod$.

Now, from the $(\Dpsirhbar, \Dpsirhbar \otimes \biWhittakerDifferentialOperatorsHBAR)$-bimodule structure on \labelcref{Transfer for Left Biwhit Action on Dpsil Mod}, we obtain an induced $(\biWhittakerDifferentialOperatorsHBAR, \Dpsirhbar \otimes \biWhittakerDifferentialOperatorsHBAR)$-bimodule structure on the coinvariants $\LieAlgebraofUnipotentRadicalOfOppositeBorelPSIl\Dpsirhbar\backslash \Dpsirhbar \otimes_{\Zhbarg} \biWhittakerDifferentialOperatorsHBAR$. From this, the invariants \[(\LieAlgebraofUnipotentRadicalOfOppositeBorelPSIl\Dpsirhbar\backslash \Dpsirhbar \otimes_{\Zhbarg} \biWhittakerDifferentialOperatorsHBAR)^{\LieAlgebraofUnipotentRadicalOfOppositeBorelPSIl \otimes 1} \xleftarrow{\sim} \biWhittakerDifferentialOperatorsHBAR \otimes_{\Zhbarg} \biWhittakerDifferentialOperatorsHBAR\] for the right action of $\LieAlgebraofUnipotentRadicalOfOppositeBorelPSIl \subseteq \Dpsirhbar$ obtains an induced $(\biWhittakerDifferentialOperatorsHBAR, \biWhittakerDifferentialOperatorsHBAR \otimes \biWhittakerDifferentialOperatorsHBAR)$-bimodule structure. 
Therefore, using the isomorphism in \labelcref{Tensor for BiWhit Mod Via Transfer}, we equip the tensor product of two graded $\biWhittakerDifferentialOperatorsHBAR$-modules with a graded $\biWhittakerDifferentialOperatorsHBAR$-module structure. This gives our desired monoidal operation. The unit of the $\biWhittakerDifferentialOperatorsHBAR$-module structure is given by the induced $\biWhittakerDifferentialOperatorsHBAR$-module structure on $\Zhbarg$ constructed above. Starting with the $(\DhbarG, \DhbarG \otimes \DhbarG \otimes \DhbarG)$-bimodule $\DhbarG \otimes_{\Uhbarg} \DhbarG \otimes_{\Uhbarg} \DhbarG$ and tracing through the above constructions, one can construct associativity isomorphisms and construct the monoidal structure on $\biWhittakerDifferentialOperatorsHBAR\mmod$. From the monoidal structure on $\biWhittakerDifferentialOperatorsHBAR\mmod$, one can prove the remainder of (i) and (ii) completely analogously as in the proof of \cref{Differential Operators on G Has Hopf Algebroid Structure and Comodules Has Monoidal Stucture}. 

One can also use closely parallel arguments to equip $\DhbarG \otimes_{\Uhbarg} \Dpsirhbar \otimes_{\Zhbarg}\biWhittakerDifferentialOperatorsHBAR$ with a $(\Dpsirhbar, \DhbarG \otimes \Dpsirhbar \otimes \biWhittakerDifferentialOperatorsHBAR)$-bimodule structure, starting with the $(\DhbarG, \DhbarG \otimes \DhbarG \otimes \DhbarG)$-bimodule $\DhbarG \otimes_{\Uhbarg} \DhbarG \otimes_{\Uhbarg} \DhbarG$. Using this, one can show that the assignment \labelcref{Assignment for Bimodule Category Action} gives $\Dpsilhbar\mmod$ admits the structure of a $(\biWhittakerDifferentialOperatorsHBAR\mmod, \DhbarG\mmod)$-bimodule category, showing (iii). Finally, the coaction map on $\Dpsirhbar$ is given by the formula \[\partial \mapsto \partial \cdot (1 \otimes 1 \otimes 1) \in \DhbarG \otimes_{\Uhbarg} \Dpsirhbar \otimes_{\Zhbarg} \biWhittakerDifferentialOperatorsHBAR\] which gives (iv).  
\end{proof}

\begin{Remark}
    \vi The coalgebra structure on $\biWhittakerDifferentialOperatorsHBAR$ may be viewed as a quantization of the structure exhibiting $J$ as a group scheme over $\LG\sslash G$. Moreover, \cref{Asymptotic Enhancement of All Preliminary Results}(iv) may be viewed as a quantization of the action of $G \times J_{\text{op}}$ on $T^*_\psi(G/N)$.
    
    \vii Essentially identical arguments can be used to equip $\Dpsilhbar$, viewed as a right $\Uhbarg$-module via left multiplication by the subring $\Uhbarg_{\mathrm{op}}$ and as a left $\Zhbarg$-module via left multiplication by the subring $\Zhbarg \subseteq \Dpsilhbar$, with the structure of a $(\biWhittakerDifferentialOperatorsHBAR, \DhbarG)$-bicomodule, as well as to equip the category $\Dpsilhbar\mmod$ with the structure of a $(\biWhittakerDifferentialOperatorsHBAR\mmod,\,\Dh\mmod)$-bimodule category.
   
    \viii There is also a `nonasymptotic' counterpart of \cref{Asymptotic Enhancement of All Preliminary Results} where\ $\hb$ is specialized to $1$. It is proved in a similar way.
\end{Remark}

\newcommand{\epsilonpsir}{\epsilon^{\psi_r}}
\newcommand{\epsilonpsil}{\epsilon^{\psi_{\ell}}}
\newcommand{\counitForbiWhittakerDifferentialOperatorsHBAR}{\overline{\epsilon}}
\section{Bi-Whittaker reduction and $\biWhittakerDifferentialOperatorsHBAR$-comodules}\label{Bi-Whittaker Reduction and Biwhittaker Comodules}
\subsection{} \label{class whitt}
The theorem below,
  which is the main result of this section, informally constructs a functor $\kostantWhittakerASYMPTOTICtoComodForASYMPTOTICbiWhit$ on $\AsymptoticHarishChandraBimodulesForTHISGROUPDefaultsToG$ with many similar properties to the pullback functor $\jmath^*$ for the open embedding $\jmath$ as in \cref{Asymptotic and Classical Variants Subsection}:

\begin{Theorem}\label{Asymptotic Kostant Whittaker Reduction is Monoidal Localization} 
  \vi  The functor $\kostantWhittakerForNONFILTEREDtoZGBIMOD_\hb$ lifts to an exact monoidal functor
  $\kostantWhittakerASYMPTOTICtoComodForASYMPTOTICbiWhit: \HarishChandraBimodulesForTHISGROUPDefaultsToG_\hb \to \biWhittakerDifferentialOperatorsonG_\hb\comod$ to the category of graded $\biWhittakerDifferentialOperatorsHBAR$-comodules which induces a monoidal equivalence of categories \begin{equation}\label{Monoidal Equivalence of Categories Induced by HC}\HarishChandraBimodulesForTHISGROUPDefaultsToG_\hb/\operatorname{ker}(\kappa_\hb)
    \iso \jh\comod.
\end{equation}

  \vii The assignment $C\mto \kostantWhittakerToWComod_\hb^R(C) := \Dpsirhbar \otimes_{\Zhbarg}^{\jh} C$ upgrades to a fully faithful functor
  \[\kappa_\hb^R: \jh\comod\to \HarishChandraBimodulesForTHISGROUPDefaultsToG_\hb\]
  which is right adjoint to $\kostantWhittakerASYMPTOTICtoComodForASYMPTOTICbiWhit$ and commutes with filtered colimits.
  
  \viii There is a natural isomorphism of functors $\kostantWhittakerASYMPTOTICtoComodForASYMPTOTICbiWhit(\td)\cong \Dpsilhbar \otimes_{\Uhbarg}^G (\td)$. 
\end{Theorem}

To prove \cref{Asymptotic Kostant Whittaker Reduction is Monoidal Localization} we first need the following

  \begin{Lemma}
    \label{Asymptotic Kostant-Whittaker Reduction Is Monoidal and Has Right Adjoint}
  The composite functor $\kostantWhittakerHBARandVIEWEDASLEFTZGMODULE: \AsymptoticHarishChandraBimodulesForTHISGROUPDefaultsToG
  \xrightarrow{\, \varkappa_\hb\, } \Zhbarg\bimod\xrightarrow{\,\operatorname{Obl}_r\,} \Zhbarg\mmod$
  admits a right adjoint $\kostantWhittakerHBARandVIEWEDASLEFTZGMODULE^R(\td) := \Dpsirhbar \otimes_{Z\LG} (\td)$. Here, the functor
  $\operatorname{Obl}_r$
  forgets the right $\Zhbarg$-action
  and  $\Dpsirhbar$ in the tensor product is viewed as a $Z_\hb\g$-module via {\em left} multiplication.
\end{Lemma}

\newcommand{\rhopsir}{\rho^{\Dpsir}}
\begin{proof}[Proof of \cref{Asymptotic Kostant-Whittaker Reduction Is Monoidal and Has Right Adjoint}]
  We first construct the counit $c(M)$ of the desired adjunction. In \cref{Monoidal Structure on BiWhit Mod and Monoidality} we equipped $\Zhbarg$ with a left
  $\biWhittakerDifferentialOperatorsHBAR$-module structure. Define the map $\overline{\epsilon}: \biWhittakerDifferentialOperatorsHBAR \to \Zhbarg$ by the formula $\overline{\epsilon}(w) = w\cdot1$. For any $M \in \Zhbarg\modwithdash$, define $c(M):
  \kostantWhittakerHBARandVIEWEDASLEFTZGMODULE(\kostantWhittakerHBARandVIEWEDASLEFTZGMODULE^R(M)) \to M$ to be the composite of the following maps\[\kostantWhittakerHBARandVIEWEDASLEFTZGMODULE(\kostantWhittakerHBARandVIEWEDASLEFTZGMODULE^R(M)) = \kostantWhittakerHBARandVIEWEDASLEFTZGMODULE(\Dpsirhbar \otimes_{\Zhbarg} M) \xleftarrow{\sim} \biWhittakerDifferentialOperatorsHBAR \otimes_{\Zhbarg} M \xrightarrow{\overline{\epsilon} \otimes \text{id}} \Zhbarg \otimes_{\Zhbarg} M \xrightarrow{\sim} M\] where the leftward pointing arrow is the isomorphism induced by the exactness of $\kostantWhittakerHBARandVIEWEDASLEFTZGMODULE$ and the final arrow is the isomorphism induced by the module structure.

  We now construct the unit map of our desired adjunction. Recall that any asymptotic Harish-Chandra bimodule $C$ is equivalently a comodule for $\DhbarG$ by \cref{Differential Operators on G Has Hopf Algebroid Structure and Comodules Has Monoidal Stucture}(iv). Let $\rho^C$ denote the $D_\hbar$-coaction on $C$ and consider
    a chain of maps (read left to right)
    \begin{equation}\label{Unit Map for C Def}C \xrightarrow{\rho^C} \DhbarG \otimes_{\Uhbarg}^G C \xrightarrow{} \DhbarG  \otimes_{\Uhbarg}^{\unipotentRadicalOfOppositeBorel}C/C\LieAlgebraofUnipotentRadicalOfOppositeBorelPSI \xleftarrow{\sim} \Dpsirhbar \otimes_{\Zhbarg} \kostantWhittakerHBARandVIEWEDASLEFTZGMODULE(C)\end{equation} where the unlabeled arrow is induced by the quotient map and the final arrow is the (graded) monoidality isomorphism of \cref{Asymptotic Kostant-Whittaker Reduction Is Monoidal and Exact}.
    We claim that the composite of the maps in \labelcref{Unit Map for C Def}, which we denote $u(C)$, gives the unit of our adjunction.

We now show that $\kostantWhittakerHBARandVIEWEDASLEFTZGMODULE^R(c(M)) \circ u(\kostantWhittakerHBARandVIEWEDASLEFTZGMODULE^R(M))$
is the identity for any $\Zhbarg$-module $M$. We first prove this in the special case when $M = \Zhbarg$. Unwinding the definitions, we see that we wish to show the composite (read left to right) of the maps \begin{equation*}\Dpsirhbar \xrightarrow{\rhopsir} \DhbarG  \otimes_{\Uhbarg}^G \Dpsirhbar  \xrightarrow{}\DhbarG  \otimes_{\Uhbarg}^{\overline{N}} \Dpsirhbar/\Dpsirhbar\LieAlgebraofUnipotentRadicalOfOppositeBorelPSIl \xleftarrow{\sim} \Dpsirhbar \otimes_{\Zhbarg} \biWhittakerDifferentialOperatorsHBAR \xrightarrow{\text{id} \otimes \counitForbiWhittakerDifferentialOperatorsHBAR} \Dpsirhbar \otimes_{\Zhbarg} \Zhbarg \xrightarrow{\overlinepartial \otimes z \mapsto z\overlinepartial} \Dpsirhbar\end{equation*} is the identity, where $\rhopsir$ denotes the coaction on $\Dpsirhbar$ from \cref{Asymptotic Enhancement of All Preliminary Results}(iv). If we choose $\overline{\partial} \in \Dpsirhbar$ and write $\rhopsir(\overline{\partial}) = \partial_{(0)} \otimes \overline{\partial_{(1)}}$ for $\partial_{(0)} \in \DhbarG$ and $\overline{\partial_{(1)}} \in \Dpsirhbar$ and choose some $x_i, y_i \in \DhbarG$ so that $\overline{x_i} \in \Dpsirhbar, \overline{y_i} \in \Dpsirhbar$, and $\overline{\overline{y_i}} \in \biWhittakerDifferentialOperatorsHBAR$ for which $\sum_{i} x_i \otimes \overline{y_i} = \partial_{(0)} \otimes \overline{\partial_{(1)}}$ in $\DhbarG  \otimes_{\Uhbarg}^{\overline{N}} \Dpsirhbar/\Dpsirhbar\LieAlgebraofUnipotentRadicalOfOppositeBorelPSIl$, then this map is given explicitly by the formula \[\overline{\partial} \mapsto \sum_i\counitForbiWhittakerDifferentialOperatorsHBAR(\overline{\overline{y_i}})\overline{x_i} = \sum_i\epsilonpsir(\overline{y_i})\overline{x_i} = \sum_i\overline{\epsilon(y_i)x_i}\] in $\Dpsirhbar$, where $\epsilonpsir$ denotes the $\Dpsirhbar$-module structure on $\Uhbarg/\Uhbarg\cdot\LieAlgebraofUnipotentRadicalOfOppositeBorelPSI$ described explicitly in \cref{Monoidal Structure on BiWhit Mod and Monoidality}. By assumption, then, for such $x_i, y_i$ there exists some $\sum_j p_j \otimes \overline{p_j'} \in \DhbarG  \otimes_{\Uhbarg} \Dpsirhbar$ and some $\xi \in \LieAlgebraofUnipotentRadicalOfOppositeBorelPSIl$ such that $\sum_i x_i \otimes \overline{y_i} = \partial_{(0)} \otimes \overline{\partial_{(1)}} + \sum_j p_j \otimes \overline{p_j'}\xi$ in $\DhbarG  \otimes_{\Uhbarg} \Dpsirhbar$ and so we see that 
    \[\sum_i\overline{\epsilon(y_i)x_i} = \overline{\epsilon(\partial_{(1)})\partial_{(0)}} + \sum_j\overline{\epsilon(p_j'\xi)p_j} = \overline{\epsilon(\partial_{(1)})\partial_{(0)}} = \overlinepartial\]as desired. 
    
    Now assume $M$ is a general $\Zhbarg$-module. Observe that, since $\kostantWhittakerHBARandVIEWEDASLEFTZGMODULE$ and $\kostantWhittakerHBARandVIEWEDASLEFTZGMODULE^R$ are both right exact and commute with arbitrary direct sums--the former by \cref{Asymptotic Kostant-Whittaker Reduction Is Monoidal and Exact}(ii) and the latter as it is given by tensoring with some module. Therefore, the composite $\kostantWhittakerHBARandVIEWEDASLEFTZGMODULE^R\kostantWhittakerHBARandVIEWEDASLEFTZGMODULE\kostantWhittakerHBARandVIEWEDASLEFTZGMODULE^R$ is also right exact and commutes with arbitrary direct sums, and so there is an isomorphism \[\kostantWhittakerHBARandVIEWEDASLEFTZGMODULE^R\kostantWhittakerHBARandVIEWEDASLEFTZGMODULE^R\kostantWhittakerHBARandVIEWEDASLEFTZGMODULE^R(\Zhbarg) \otimes_{\Zhbarg} M \cong \kostantWhittakerHBARandVIEWEDASLEFTZGMODULE^R\kostantWhittakerHBARandVIEWEDASLEFTZGMODULE\kostantWhittakerHBARandVIEWEDASLEFTZGMODULE^R(M)\] natural in $M$ given by say \cite[Lemma 56.2.6]{StacksProject}. This isomorphism intertwines the maps $\kostantWhittakerHBARandVIEWEDASLEFTZGMODULE^R(c(M)) \circ u(\kostantWhittakerHBARandVIEWEDASLEFTZGMODULE^R(M))$ and $\kostantWhittakerHBARandVIEWEDASLEFTZGMODULE^R(c(\Zhbarg)) \circ u(\kostantWhittakerHBARandVIEWEDASLEFTZGMODULE^R(\Zhbarg)) \otimes_{\Zhbarg} M$ in the obvious way, which one can check directly. Therefore our computation that $\kostantWhittakerHBARandVIEWEDASLEFTZGMODULE^R(c(M)) \circ u(\kostantWhittakerHBARandVIEWEDASLEFTZGMODULE^R(M))$ is the identity when $M = \Zhbarg$ implies $\kostantWhittakerHBARandVIEWEDASLEFTZGMODULE^R(c(M)) \circ u(\kostantWhittakerHBARandVIEWEDASLEFTZGMODULE^R(M))$ is the identity for an arbitrary $M \in \Zhbarg\mmod$. The verification that $u(\kostantWhittakerHBARandVIEWEDASLEFTZGMODULE(B))\circ \kostantWhittakerHBARandVIEWEDASLEFTZGMODULE(u(B))$ is the identity for any $B \in \HarishChandraBimodulesForTHISGROUPDefaultsToG$ is similar. 
\end{proof}

\begin{Corollary}\label{Dpsirhbar and BiWhitHbar Are Flat}
  The $\Zhbarg$-modules $\Dpsirhbar$ and $\biWhittakerDifferentialOperatorsHBAR$ are flat. 
\end{Corollary}

\begin{proof}
  Since the functor
  $\Dpsirhbar \otimes_{\Zhbarg} (\td)$ is a right adjoint,  it is  left exact and thus $\Dpsirhbar$ is flat.
  Hence,  the composite  functor
  $\kostantWhittakerHBARandVIEWEDASLEFTZGMODULE(\Dpsirhbar \otimes_{\Zhbarg} (\td))$
  is  exact. Using the isomorphism 
  \[\biWhittakerDifferentialOperatorsHBAR \otimes_{\Zhbarg} (\td)\cong \kostantWhittakerHBARandVIEWEDASLEFTZGMODULE(\Dpsirhbar \otimes_{\Zhbarg} (\td))\] given by the exactness of $\kostantWhittakerHBARandVIEWEDASLEFTZGMODULE$ in \cref{Asymptotic Kostant-Whittaker Reduction Is Monoidal and Exact}(ii), we deduce that $\biWhittakerDifferentialOperatorsHBAR \otimes_{\Zhbarg} (\td)$
  is exact, proving the flatness statement for
  $\biWhittakerDifferentialOperatorsHBAR$.
\end{proof}

Since $\DhbarG \cong \Uhbarg \otimes \mathcal{O}(G)$ is flat as a $\Uhbarg$-module and $\kostantWhittakerHBARandVIEWEDASLEFTZGMODULE$ is exact by \cref{Asymptotic Kostant-Whittaker Reduction Is Monoidal and Exact}, \cref{Asymptotic Kostant Whittaker Reduction is Monoidal Localization}(iii) follows immediately from the following more general result: 

\begin{Lemma} Assume $D$ is some $R$-coring object
    which is flat as an $R$-module 
    and \[\kostantWhittakerSTANDIN: D\comod\to \mathcal{A}\] is some exact functor. 
    Then there is a canonical isomorphism $\kostantWhittakerSTANDIN(C) \to \kostantWhittakerSTANDIN(D) \otimes_{R}^{D} C$ which is natural in $C$.
\end{Lemma}

\begin{proof}
    For any comodule $C$, the coaction map naturally factors through a map $C \xrightarrow{} D \otimes_{R}^{D} C$ of $D$-comodules by definition of a comodule. This map is easily shown to be an isomorphism, see for example \cite[Section 21.1]{BrzezinskiWisbauerCoringsandComodules}. Therefore applying $\kostantWhittakerSTANDIN$ to this isomorphism we obtain an isomorphism \[\kostantWhittakerSTANDIN(C) \xrightarrow{\sim} \kostantWhittakerSTANDIN(D \otimes_{R}^{D} C) \xrightarrow{\sim} \kostantWhittakerSTANDIN(D \otimes_{R} C)^{D} \xleftarrow{\sim} \kostantWhittakerSTANDIN(D) \otimes_{R}^{D} C\] where the middle arrow, respectively rightmost arrow, is an isomorphism by the left exactness, respectively right exactness, of $\kostantWhittakerSTANDIN$. These maps are evidently natural in $C$.
\end{proof}

\begin{proof}[Proof of \cref{Asymptotic Kostant Whittaker Reduction is Monoidal Localization}(i)-(ii)] From the adjunction of \cref{Asymptotic Kostant-Whittaker Reduction Is Monoidal and Has Right Adjoint}, we formally obtain the structure of a \textit{comonad} on the endofunctor $T := \kostantWhittakerHBARandVIEWEDASLEFTZGMODULE\kostantWhittakerHBARandVIEWEDASLEFTZGMODULE^R$, that is, natural transformations $T \to T^2$ and $T \to \mathrm{id}$ satisfying the obvious co-associativity and co-identity conditions.\footnote{These conditions given by taking the conditions morphisms for a \textit{monad} must satisfy, given in \cite[Definition VI.1]{MacLaneCategoriesForTheWorkingMathematician}, and reversing the arrows.} Moreover, since $\kostantWhittakerHBARandVIEWEDASLEFTZGMODULE$ is exact, we see that this comonad is naturally isomorphic to tensoring with the $\Zhbarg$-module $\biWhittakerDifferentialOperatorsHBAR := \kostantWhittakerHBARandVIEWEDASLEFTZGMODULE\kostantWhittakerHBARandVIEWEDASLEFTZGMODULE^R(\Zhbarg)$, and thus $\kostantWhittakerHBARandVIEWEDASLEFTZGMODULE\kostantWhittakerHBARandVIEWEDASLEFTZGMODULE^R(\Zhbarg)$ acquires the structure of a graded $\Zhbarg$-coalgebra. Tracing through the constructions of the coalgebra structure on $\biWhittakerDifferentialOperatorsHBAR$ in \cref{Asymptotic Enhancement of All Preliminary Results}, we see that these two a priori distinct coalgebra structures agree. We therefore may obtain our desired functor $\kostantWhittakerASYMPTOTICtoComodForASYMPTOTICbiWhit$ by following the procedure as in \cite[Chapter VI]{MacLaneCategoriesForTheWorkingMathematician}, with the direction of the arrows reversed. Explicitly, to any $\DhbarG $-comodule $C$, the coaction on $\kostantWhittakerHBARandVIEWEDASLEFTZGMODULE(C)$ is given by the composite, read left to right, of the maps:\[\kostantWhittakerHBARandVIEWEDASLEFTZGMODULE(C) \xrightarrow{\kostantWhittakerHBARandVIEWEDASLEFTZGMODULE(u(C))} \kostantWhittakerHBARandVIEWEDASLEFTZGMODULE\kostantWhittakerHBARandVIEWEDASLEFTZGMODULE^R(\kostantWhittakerHBARandVIEWEDASLEFTZGMODULE(C)) \xleftarrow{\sim} \biWhittakerDifferentialOperatorsHBAR \otimes_{\Zhbarg} \kostantWhittakerHBARandVIEWEDASLEFTZGMODULE(C)\] where the final arrow is an isomorphism by the exactness of Kostant-Whittaker reduction, see \cref{Asymptotic Kostant-Whittaker Reduction Is Monoidal and Exact}.

  This comodule structure can also be described by the following formula: for any $\overline{b} \in (C/C\LieAlgebraofUnipotentRadicalOfOppositeBorelPSI)^{\unipotentRadicalOfOppositeBorel}$, there exists $\overline{\overline{\partial_i}} \in \biWhittakerDifferentialOperatorsHBAR$ and $\overline{y_i} \in (C/C\LieAlgebraofUnipotentRadicalOfOppositeBorelPSI)^{\unipotentRadicalOfOppositeBorel}$ such that $\overline{b_{(0)}} \otimes \overline{b_{(1)}} = \sum_i \overline{\partial_i} \otimes \overline{y_i}$ in $\DhbarG /\DhbarG \LieAlgebraofUnipotentRadicalOfOppositeBorelPSIl \otimes_{\Uhbarg} C/C\LieAlgebraofUnipotentRadicalOfOppositeBorelPSI$, and this coaction map is defined by $\overline{b} \mapsto \sum_i \overline{\overline{\partial_i}} \otimes \overline{y_i}$. We can directly compute that the $\Zhbarg$-bimodule structure on $\kostantWhittakerHBARtoZGBIMOD(C)$ upgrades to that of $\biWhittakerDifferentialOperatorsHBAR$-comodule, where we recall that any $\biWhittakerDifferentialOperatorsHBAR$-comodule acquires a $\Zhbarg$-bimodule structure by \cref{Asymptotic Enhancement of All Preliminary Results}(ii). 
  
  The fact $\kostantWhittakerASYMPTOTICtoComodForASYMPTOTICbiWhit^R: \biWhittakerDifferentialOperatorsHBAR\comod\to \DhbarG \comod$ is a right adjoint follows formally from the fact that the unit and counit map of the adjunction constructed in \cref{Asymptotic Kostant-Whittaker Reduction Is Monoidal and Has Right Adjoint} factors through the $\biWhittakerDifferentialOperatorsHBAR$-cotensor product. The induced counit map for this adjunction is readily checked to be an isomorphism, and so $\kostantWhittakerASYMPTOTICtoComodForASYMPTOTICbiWhit^R$ is fully faithful. 

We now equip $\kostantWhittakerASYMPTOTICtoComodForASYMPTOTICbiWhit$ with a monoidal structure. Observe that we may equivalently describe the induced coaction $\rho_{\kostantWhittakerASYMPTOTICtoComodForASYMPTOTICbiWhit(C)}$ as follows: 
it is the unique map of $\Zhbarg \otimes_{k[\hbar]} \Zhbarg_{\op}$-modules such that the induced diagram of $\Zhbarg \otimes \Zhbarg_{\op} \otimes \biWhittakerDifferentialOperatorsHBAR^{\op}$-modules 
\begin{equation}\label{Definition of Comodule Structure on Kappa as Commuting Diagram}\xymatrix@R+2em@C+2em{ \kostantWhittakerHBARandVIEWEDASLEFTZGMODULE^2(C \otimes_{\Uhbarg} \DhbarG ) \ar[r]^{\kostantWhittakerHBARandVIEWEDASLEFTZGMODULE^2(\widetilde{\rho}_C)} & \kostantWhittakerHBARandVIEWEDASLEFTZGMODULE^2(\DhbarG  \otimes_{\Uhbarg} C)\\ \kostantWhittakerHBARandVIEWEDASLEFTZGMODULE(C) \otimes_{\Zhbarg} \biWhittakerDifferentialOperatorsHBAR \ar[u]^{\sim} \ar[r]^{\tilde{\rho}_{\kostantWhittakerHBARandVIEWEDASLEFTZGMODULE(C)}} & \biWhittakerDifferentialOperatorsHBAR  \otimes_{\Zhbarg} \kostantWhittakerHBARandVIEWEDASLEFTZGMODULE(C)\ar[u]^{\sim}}
\end{equation}commutes, where $\kostantWhittakerForNONFILTEREDtoZGBIMOD^2$ denotes the Kostant-Whittaker reduction for the group $G \times G_{\text{op}}$, $\widetilde{\rho}_C$ is the right $\DhbarG$-linear map induced from the coaction $\rho^C$, and the unlabeled arrows are the isomorphisms induced by exactness and monoidality; for example: the left vertical arrow is given by the formula $\overline{b} \otimes \overline{\overline{\partial}} \mapsto \overline{\overline{b \otimes \partial}}$. Similarly observe that the monoidal structure on the category $\DhbarG\comod$ (which we have shown is equivalent as a monoidal category to $\HarishChandraBimodulesForTHISGROUPDefaultsToG$ in \cref{Differential Operators on G Has Hopf Algebroid Structure and Comodules Has Monoidal Stucture}(iv)) in \cref{Differential Operators on G Has Hopf Algebroid Structure and Comodules Has Monoidal Stucture} has the property and is characterized by the fact that \begin{equation}\label{Equation Characterizing Coaction on Tensor Product as Commuting The Algebra Past and Then Commuting It Past Again}\widetilde{\rho}_{C \otimes_{\Uhbarg} C'} = (\widetilde{\rho}_{C} \otimes \text{id}_{C'}) \circ (\text{id}_C \otimes \widetilde{\rho}_{C'})\end{equation}and, moreover, that this map is isomorphism of $\Uhbarg \otimes \Uhbarg_{\op} \otimes \DhbarG ^{\op}$-modules.

Now, using \eqref{Definition of Comodule Structure on Kappa as Commuting Diagram} we obtain commutative diagrams \begin{equation*}\xymatrix@R+2em@C+2em{\kostantWhittakerASYMPTOTICtoComodForASYMPTOTICbiWhit^2(C \otimes_{\Uhbarg} C' \otimes_{\Uhbarg} \DhbarG) \ar[r]^{\kostantWhittakerASYMPTOTICtoComodForASYMPTOTICbiWhit^2(\text{id}_C \otimes \widetilde{\rho}_{C'})} &  \kostantWhittakerASYMPTOTICtoComodForASYMPTOTICbiWhit^2(C \otimes_{\Uhbarg} \DhbarG \otimes_{\Uhbarg} C') \ar[r]^{\kostantWhittakerASYMPTOTICtoComodForASYMPTOTICbiWhit^2(\widetilde{\rho}_C \otimes \text{id}_{C'})}  & \kostantWhittakerASYMPTOTICtoComodForASYMPTOTICbiWhit^2(\DhbarG \otimes_{\Uhbarg} C \otimes_{\Uhbarg} C') \\ \kostantWhittakerASYMPTOTICtoComodForASYMPTOTICbiWhit(C) \otimes_{\Zhbarg} \kostantWhittakerASYMPTOTICtoComodForASYMPTOTICbiWhit(C') \otimes_{\Zhbarg} \biWhittakerDifferentialOperatorsHBAR \ar[u]^{\sim} \ar[r]^{\text{id}_{\kappa(C)} \otimes \widetilde{\rho}_{\kostantWhittakerASYMPTOTICtoComodForASYMPTOTICbiWhit(C')}}& \kostantWhittakerASYMPTOTICtoComodForASYMPTOTICbiWhit(C) \otimes_{\Zhbarg} \biWhittakerDifferentialOperatorsHBAR \otimes_{\Zhbarg}  \kostantWhittakerASYMPTOTICtoComodForASYMPTOTICbiWhit(C') \ar[u]^{\sim} \ar[r]^{\tilde{\rho}_{\kostantWhittakerASYMPTOTICtoComodForASYMPTOTICbiWhit(C)} \otimes \text{id}_{C'}} & \biWhittakerDifferentialOperatorsHBAR \otimes_{\Zhbarg} \kostantWhittakerASYMPTOTICtoComodForASYMPTOTICbiWhit(C) \otimes_{\Zhbarg} \kostantWhittakerASYMPTOTICtoComodForASYMPTOTICbiWhit(C') \ar[u]^{\sim}}
\end{equation*} and \begin{equation*}\xymatrix@R+2em@C+2em{  \kostantWhittakerASYMPTOTICtoComodForASYMPTOTICbiWhit(C \otimes_{\Uhbarg} C') \otimes_{\Zhbarg} \biWhittakerDifferentialOperatorsHBAR \ar[d]^{\sim} \ar[r]^{\tilde{\rho}_{\kostantWhittakerASYMPTOTICtoComodForASYMPTOTICbiWhit(C \otimes_{\Uhbarg} C')}} & \biWhittakerDifferentialOperatorsHBAR \otimes_{\Zhbarg} \kostantWhittakerASYMPTOTICtoComodForASYMPTOTICbiWhit(C \otimes_{\Uhbarg} C')\ar[d]^{\sim} \\
\kostantWhittakerASYMPTOTICtoComodForASYMPTOTICbiWhit^2(C \otimes_{\Uhbarg} C' \otimes_{\Uhbarg} \DhbarG) \ar[r]^{\kostantWhittakerASYMPTOTICtoComodForASYMPTOTICbiWhit^2(\widetilde{\rho}_{C \otimes_{\Uhbarg} C'})} & \kostantWhittakerASYMPTOTICtoComodForASYMPTOTICbiWhit^2(\DhbarG \otimes_{\Uhbarg} C \otimes_{\Uhbarg} C')} 
\end{equation*} where as above all unlabeled arrows are induced by exactness and monoidality.  Applying $\kostantWhittakerASYMPTOTICtoComodForASYMPTOTICbiWhit^2$ to the equation \eqref{Equation Characterizing Coaction on Tensor Product as Commuting The Algebra Past and Then Commuting It Past Again} 
for the $\DhbarG$-comodule $C \otimes_{\Uhbarg} C'$, we obtain the equality \begin{equation}\label{Equality of Morphisms for Tensoring}\kostantWhittakerASYMPTOTICtoComodForASYMPTOTICbiWhit^2(\widetilde{\rho}_{C \otimes_{\Uhbarg} C'}) = \kostantWhittakerASYMPTOTICtoComodForASYMPTOTICbiWhit^2(\widetilde{\rho}_{C} \otimes \text{id}_{C'}) \circ \kostantWhittakerASYMPTOTICtoComodForASYMPTOTICbiWhit^2(\text{id}_C \otimes \widetilde{\rho}_{C'})\end{equation} of morphisms.
It follows from \labelcref{Equality of Morphisms for Tensoring} and the two commutative diagrams above that the monoidality isomorphism of \cref{Asymptotic Kostant-Whittaker Reduction Is Monoidal and Exact} respects the $\biWhittakerDifferentialOperatorsHBAR$-comodule structure, and completes our proof that $\kostantWhittakerASYMPTOTICtoComodForASYMPTOTICbiWhit$ is monoidal. 

We now verify that $\kostantWhittakerASYMPTOTICtoComodForASYMPTOTICbiWhit^R$ preserves filtered colimits. The compact objects in $\HarishChandraBimodulesForTHISGROUPDefaultsToG$ are precisely those Harish-Chandra bimodules which are finitely generated as left, equivalently right, $\Uhbarg$-modules. The functor $\kappa_{\hbar}$ sends finitely generated $\Uhbarg$-modules to finitely generated $\Zhbarg$-modules. In particular, since an object $C \in\biWhittakerDifferentialOperatorsHBAR\text{-comod}$ is compact only if the underlying $\Zhbarg$-bimodule is compact,
we see that $\kappa_{\hbar}$ preserves compact objects. Since any object of $\HarishChandraBimodulesForTHISGROUPDefaultsToG$ is a filtered colimit of finitely generated ones, a Yoneda lemma argument gives that $\kostantWhittakerASYMPTOTICtoComodForASYMPTOTICbiWhit^R$ preserves filtered colimits. 

Finally, since $\kostantWhittakerASYMPTOTICtoComodForASYMPTOTICbiWhit$ is monoidal, the universal property of quotient categories induces the desired monoidal functor in \labelcref{Monoidal Equivalence of Categories Induced by HC}, see for example \cite{DayNoteonMonoidalLocalization}. Moreover, \cite[Proposition III.5]{GabrielDesCategoriesAbeliannes} implies that this induced functor is an equivalence.
\end{proof}

  One has a monoidal functor
  $\freeFunctorASYMPTOTICSETTING: \text{Rep}(G \times \mathbb{G}_m) \to \AsymptoticHarishChandraBimodulesForTHISGROUPDefaultsToG$ given by tensor product
  $V \mapsto \Uhbarg \otimes V$ of graded vector spaces.
  
  \begin{Lemma}\label{Fr lem} 
    For any, not necessarily finite dimensional, $G$-representation $V$
    the unit
    $\freeFunctorASYMPTOTICSETTING(V) \to \kostantWhittakerASYMPTOTICtoComodForASYMPTOTICbiWhit^R\kostantWhittakerASYMPTOTICtoComodForASYMPTOTICbiWhit(\freeFunctorASYMPTOTICSETTING(V))$
    for the adjunction $(\kostantWhittakerASYMPTOTICtoComodForASYMPTOTICbiWhit, \kostantWhittakerASYMPTOTICtoComodForASYMPTOTICbiWhit^R)$  is an isomorphism.
    \end{Lemma}
    \begin{proof}
          We have isomorphisms \[\DhbarG \otimes_{\Uhbarg} \Uhbarg/\Uhbarg \cdot \mathfrak{n}^{\psi} \cong \mathcal{O}(G) \otimes \Uhbarg/\Uhbarg \cdot \mathfrak{n}^{\psi} \cong
        \oplus_{\lambda} \, (V_{\lambda}^{\text{dim}(V_{\lambda})} \otimes \Uhbarg/\Uhbarg \cdot \mathfrak{n}^{\psi}).\]
Since the quotient $\Uhbarg/\Uhbarg \cdot \mathfrak{n}^{\psi}$ has nonnegative grading, it follows that the grading on $\Dpsilhbar$, respectively $\Dpsirhbar$ induced by the actions of $G_{\text{op}}$, respectively $G$, have isotypic components bounded below.

The functor $\kostantWhittakerASYMPTOTICtoComodForASYMPTOTICbiWhit^R$ commutes with filtered colimits and $\kostantWhittakerASYMPTOTICtoComodForASYMPTOTICbiWhit$ commutes with all colimits, and in particular filtered colimits, as $\kostantWhittakerASYMPTOTICtoComodForASYMPTOTICbiWhit$ is a left adjoint; both of these claims are justified by \cref{Asymptotic Kostant Whittaker Reduction is Monoidal Localization}(ii). Therefore it suffices to prove the claim when $V$ is a finite dimensional representation. In this case,
following \cite{BezrukavnikovFinkelbergEquivariantSatakeCategoryandKostantWhittakerReduction}, observe that the unit map $u(\freeFunctorASYMPTOTICSETTING(V))$ from the adjunction of \cref{Asymptotic Kostant Whittaker Reduction is Monoidal Localization} is a map in $\AsymptoticHarishChandraBimodulesForTHISGROUPDefaultsToG$ from a graded
  vector space with   $\freeFunctorASYMPTOTICSETTING(V)$
  separated filtration whose isotypic components are bounded below, since $\kostantWhittakerASYMPTOTICtoComodForASYMPTOTICbiWhit^R\kostantWhittakerASYMPTOTICtoComodForASYMPTOTICbiWhit(\freeFunctorASYMPTOTICSETTING(V)) \subseteq \Dpsirhbar \otimes_{\Zhbarg} \kostantWhittakerASYMPTOTICtoComodForASYMPTOTICbiWhit(\freeFunctorASYMPTOTICSETTING(V))$ and the latter is the tensor product of a vector space whose grading is bounded from below with $\Dpsirhbar$, whose isotypic components are bounded below by the above analysis.
By  graded Nakayama
  and separatedness, then, to show $u(\freeFunctorASYMPTOTICSETTING(V))$ is an isomorphism it therefore suffices to check that $k \otimes_{k[\hbar]} u(\freeFunctorASYMPTOTICSETTING(V))$ is an isomorphism. Moreover, since $\kostantWhittakerASYMPTOTICtoComodForASYMPTOTICbiWhit(\freeFunctorASYMPTOTICSETTING(V))$ is torsion free as a $k[\hbar]$-module, the natural map \[k \otimes_{k[\hbar]} \kostantWhittakerASYMPTOTICtoComodForASYMPTOTICbiWhit^R\kostantWhittakerASYMPTOTICtoComodForASYMPTOTICbiWhit(\freeFunctorASYMPTOTICSETTING(V)) \xrightarrow{} \kostantWhittakerInCLASSICALSETTINGwithGRADEDLIFTtoREPOFWZERO^R\kostantWhittakerInCLASSICALSETTINGwithGRADEDLIFTtoREPOFWZERO(k \otimes_{k[\hbar]} \freeFunctorASYMPTOTICSETTING(V)) \cong \Dpsirzero \otimes_{Z\LG}^{\biWhittakerDifferentialOperatorsZERO} \kostantWhittakerInCLASSICALSETTINGwithGRADEDLIFTtoREPOFWZERO(\Sym(\LG) \otimes V)\] is injective, 
  and so it suffices to show that the composite of $k \otimes_{k[\hbar]} u(\freeFunctorASYMPTOTICSETTING(V))$ with this map is an isomorphism. However, this composite is the unit map for $\kostantWhittakerInCLASSICALSETTINGwithGRADEDLIFTtoREPOFWZERO$. Therefore by \cref{KostantWhittaker In Classical Setting Induces Monoidal Equivalence with G equivariant QCoh on greg} this unit map is an isomorphism
  if and only if the unit map $\Sym(\LG) \otimes V \to \jmath_*\jmath^*(\Sym(\LG) \otimes V)$ for $(\jmath^*, \jmath_*)$ is an isomorphism.
  But this map is an isomorphism since the complement of $\LGd_{\text{reg}}$ in $\LGd$ has codimension at least two and $\Sym(\LG) \otimes V$ is a vector bundle on $\LGd$.
     \end{proof}

Various versions of the following corollary
     have appeared earlier (at least in the classical setting $\hb=0$) 
         in the context of 
\textit{Sicilian theories} and \textit{Coulomb branches}, see 
\cite{MooreTachikawaOn2dTQFTsWhoseValuesareHolomorphicSymplecticVarieties}, \cite{BravermanFinkelbergNakajimaRingObjectsIntheEquivariantDerivedSatakeCategoryArisingFromCoulombBranches}, \cite{GinzburgKazhdanDifferentialOperatorsOnBasicAffineSpaceandtheGelfandGraevAction}, \cite{BielawskiOnTheMooreTachikawaVarieties}, \cite{GannonCotangentBundleofParabolicBaseAffineSpaceandKostantWhittakerDescent}, \cite{CrooksMayrandMooreTachikawaViaShiftedSymplecticGeometry}.

\begin{Corollary}\label{MT Iso for Differential Operators}
  The unit, respectively counit, for the adjoint pair $(\kostantWhittakerASYMPTOTICtoComodForASYMPTOTICbiWhit, \kostantWhittakerASYMPTOTICtoComodForASYMPTOTICbiWhit^R)$
   constructed above gives a algebra isomorphisms $\DhbarG \xrightarrow{\sim} \Dpsirhbar \otimes_{\Zhbarg}^{\biWhittakerDifferentialOperatorsHBAR} \Dpsilhbar$ and $\biWhittakerDifferentialOperatorsHBAR \xrightarrow{\sim} \Dpsilhbar \otimes_{\Uhbarg}^G \Dpsirhbar$. Similar algebra isomorphisms exist in the non-filtered and associated graded settings.
  \end{Corollary}

  \begin{proof}
    The first claim follows from \cref{Fr lem}
      by taking $V = \mathcal{O}(G)$. The second claim follows from the fact that $\kostantWhittakerASYMPTOTICtoComodForASYMPTOTICbiWhit^R$ is fully faithful by \cref{Asymptotic Kostant Whittaker Reduction is Monoidal Localization}(iii) as well as the fact that a right adjoint is fully faithful if and only if the counit map is an isomorphism.
\end{proof}


\newcommand{\os}{{\operatorname{S}}}
\newcommand{\oq}{{\operatorname{Q}}}
\renewcommand{\BX}{{\mathbb{X}}}
\newcommand{\ft}{\mathfrak{t}}
\newcommand{\loc}{\operatorname{loc}}
\newcommand{\bN}{\quantizationOfCentralizersHBAREQUALSONE}
\newcommand{\bM}{\BM}
\section{Kostant-Whittaker reduction and the Miura bimodule}
\label{Kostant-Whittaker reduction and the Miura bimodule section}
In the rest of the paper we work in the `quantum setting', that is,
we specialize $\hb=1$ unless explicitly stated otherwise. 
The constructions and results of the previous section apply in the  $\hb=1$ setting,
with essentially identical proofs.
In particular, the functor $\kostantWhittakerForNONFILTEREDtoZGBIMOD$ lifts to a functor $\kostantWhittakerToWComod$ giving  the following adjoint pair
\[
  \xymatrix{
       \HarishChandraBimodulesForTHISGROUPDefaultsToG\
    \ar@<0.5ex>[rrr]^<>(0.5){\ \kostantWhittakerToWComod\,\cong\, \Dpsil \otimes_{U\LG}^G (\td) \ }
    &&& \
    \biWhittakerDifferentialOperatorsonG\comod\
    \ar@<0.5ex>[lll]^<>(0.5){\ \kostantWhittakerToWComod^R \,=\, \Dpsir \otimes_{Z\LG}^{\biWhittakerDifferentialOperatorsonG} (\td) \ }
  }
 \]
 such that $\kostantWhittakerToWComod^R$ is fully faithful and commutes with filtered colimits. 

 \subsection{The Miura bimodule and the Gelfand-Graev action}\label{The Miura bimodule and the Gelfand-Graev action}

 \label{The Miura bimodule and the Gelfand-Graev action}
 We now introduce the {\em Miura bimodule}, the object
 that provides a link between $\J$-modules and $D(T)$-modules, and recall some properties of it we will use that are discussed in more detail in \cite[Section 6.2]{Gin} and \cite[Section 6]{GinzburgKazhdanDifferentialOperatorsOnBasicAffineSpaceandtheGelfandGraevAction}.
 To this end, we consider the action on $D(G)$ of the subgroup
 $\bar N \times N_{op}\subset G \times G_{op}$,
 where $\bar N$, resp. $N_{op}$, acts by left, resp. right, translations.
 Let
 \[D(\overline{N}\backslash G/_{\psi}N) := 
   \big(\differentialOperatorsOnG/(\differentialOperatorsOnG\LieAlgebraOfUnipotentRadicalOfBorel + \LieAlgebraofUnipotentRadicalOfOppositeBorelPSIr\differentialOperatorsOnG)\big)^{\unipotentRadicalOfBorel \times N_{\mathrm{op}}}
 \] be the corresponding
 Hamiltonian reduction of $D(G)$  with respect to the
 Lie algebra $\bar\fn\o1+1\o\fn^\psi_r\subset U\g\o U\g_{op}$. 
 We may also perform this Hamiltonian reduction in stages and obtain the same algebra: more precisely, letting $\vkap_r(A) := (A/\fn^\psi A)^{N^{\mathrm{op}}}$ for any Hamiltonian $N^{\mathrm{op}}$-algebra $A$, one obtains the following
 algebra isomorphisms, \cite[Lemma 5.4.1]{Gin}:
 \newcommand{\brr}{_{\mathbf{r}}}
 \begin{equation}\label{Isomorphisms Like If We Had Exactness for QHR for
     N}
  \vkap_r(D(\unipotentRadicalOfBorel\backslash G))=
   (D(\unipotentRadicalOfBorel\backslash G)/(\fn^\psi_{r}
   D(\unipotentRadicalOfBorel\backslash G)))^{N_{op}}
\   \xrightarrow{\sim}\ D(\overline{N}\backslash G/_{\psi}N)
\    \xleftarrow{\sim}\  (\Dpsir/\Dpsir\LieAlgebraOfUnipotentRadicalOfBorel)^{\unipotentRadicalOfBorel},\end{equation}
 analogous to \eqref{Algebra Isos For Different Descriptions of biWhittaker}.

By completely analogous arguments as those in \cref{psi reduction}, we see that there is a $(\differentialOperatorsOnT, \Dpsir)$-bimodule structure on $\LieAlgebraOfUnipotentRadicalOfBorel\Dpsir\backslash \Dpsir$, resp.
 $(\Dpsir,\differentialOperatorsOnT)$-bimodule structure
 on $\Dpsir/\Dpsir\LieAlgebraOfUnipotentRadicalOfBorel$.
 This gives the space 
 \[\MiuraBimodule := \Dpsir/(\LieAlgebraOfUnipotentRadicalOfBorel\Dpsir +\Dpsir\LieAlgebraofUnipotentRadicalOfOppositeBorelPSIl),\enspace\text{resp.}\enspace
   \check{\mathbb{M}} := \Dpsir/(\LieAlgebraofUnipotentRadicalOfOppositeBorelPSIl\Dpsir+
  \Dpsir\LieAlgebraOfUnipotentRadicalOfBorel),
 \]
 the structure of a $(\differentialOperatorsOnT, \biWhittakerDifferentialOperatorsonG)$-bimodule, resp. 
 $(\biWhittakerDifferentialOperatorsonG, \differentialOperatorsOnT)$-bimodule,
 which we will refer to as the \textit{Miura bimodule}, resp. 
  \lq transposed\rq{} Miura bimodule. Observe that, if $U := U\LG$, we have an isomorphism \begin{equation}\label{Easy Iso For Miura from Skryabin}\Symt \otimes_{Z\LG} \biWhittakerDifferentialOperatorsonG \xrightarrow{\sim} \overline{\mathfrak{n}}U\backslash U/U\mathfrak{n}^{\psi} \otimes_{Z\LG} \biWhittakerDifferentialOperatorsonG \xrightarrow{\sim} \MiuraBimodule\end{equation} given by the inclusion map and, respectively, Skryabin's equivalence, see \cref{Asymptotic Kostant-Whittaker Reduction Is Monoidal and Exact}(i).
  
Recall next  that the algebra $D(\unipotentRadicalOfBorel\backslash G)$ has an action of the Weyl group known as the \textit{Gelfand-Graev action},
see \cite[Section 3.2]{GinzburgRicheDifferentialOperatorsOnBasicAffineSpaceandtheAffineGrassmannian}, \cite{GinzburgKazhdanDifferentialOperatorsOnBasicAffineSpaceandtheGelfandGraevAction}. The subalgebra $U\g_\op\sset D(\unipotentRadicalOfBorel\backslash G)$
is fixed by this action. Furthermore, the actions of the groups $G_\op, T$,
and $W$, combined together give an action of the group $(T \rtimes W) \times G_{\text{op}}$
on $D(\unipotentRadicalOfBorel\backslash G)$ by algebra automorphisms. The action
of $T \rtimes W$ survives the Whittaker reduction, so the $(\differentialOperatorsOnT, \biWhittakerDifferentialOperatorsonG)$-bimodule, resp. 
$(\biWhittakerDifferentialOperatorsonG, \differentialOperatorsOnT)$-bimodule,
structure on
$\Mi$, resp. $\check\Mi$, can be upgraded to a
$(\differentialOperatorsOnT\rtimes W,\, \biWhittakerDifferentialOperatorsonG)$-bimodule, resp. 
$(\biWhittakerDifferentialOperatorsonG,\, D(T)\rtimes W)$-bimodule, structure.
Here, and in what follows, we use the \textit{dot}-action of $W$ on $D(T)$: in other words, taking the isomorphism $\Symt \otimes \mathcal{O}(T) \xrightarrow{\sim} D(T)$ induced by multiplication, we take the action of $W$ which acts on the $\Symt$-factor by the $\rho$-shifted $W$-action and on $\mathcal{O}(T)$ by its usual action; see also \cite[Section 1.1]{Gin} for an alternate description of this action.

The bimodule structure on $\MiuraBimodule$ gives rise to a functor 
\[F := \MiuraBimodule \otimes_{\biWhittakerDifferentialOperatorsonG} (-):  \biWhittakerDifferentialOperatorsonG\mmod \to \differentialOperatorsOnT \rtimes W\mmod =
    \differentialOperatorsOnT\mmod^W,
\] which we will now equip with a monoidal structure. To do this, we will need the following proposition, whose proof is
entirely analogous to the proof of \cref{Asymptotic Enhancement of All Preliminary Results}.

\begin{Proposition}\label{Bimodule Category with Gelfand-Graev Action}
  \vi The category $\differentialOperatorsOnT\mmod^W$ 
  acquires the structure of a $(\differentialOperatorsOnT\mmod^W, \biWhittakerDifferentialOperatorsonG\mmod)$-bimodule category. \vskip 3pt
\vii The category ${D}(\unipotentRadicalOfBorel\backslash G)\mmod^W \,(=
  \unipotentRadicalOfBorel\backslash G\rtimes W\mmod)$
  is a $((\differentialOperatorsOnT \rtimes W)\mmod, \differentialOperatorsOnG\mmod)$-bimodule category.
\end{Proposition}

Next, we observe that the fact that we have a $W$-equivariant isomorphism \begin{equation}\label{Iso by Rest}D(\overline{N}\backslash G/_{\psi}N) \xrightarrow{\sim} \D(T)\end{equation} induced by restriction to the big cell, 
and so we obtain that the category $D(\overline{N}\backslash G/_{\psi}N)\mmod^W$ is free rank one module category over the monoidal category $D(T)\mmod^W$. More precisely, if we let $\delta \in D(\overline{N}\backslash G/_{\psi}N)\mmod^W$ denote the skyscraper module supported at the identity coset then we have an equivalence of categories
\[\operatorname{act}_{\delta}: D(T) \rtimes W\mmod \to D(\overline{N}\backslash G/_{\psi}N)\mmod^W\] given by $E \mapsto E\star \delta$ and, since this functor admits a right adjoint $\operatorname{act}_{\delta}^R$, an inverse is given by $\operatorname{act}_{\delta}^R$. Using the ring isomorphism \labelcref{Iso by Rest}, we can describe $\operatorname{act}_{\delta}$ explicitly as the functor which sends a $D(T)$-module $M$ to its restriction of scalars along the above ring map. Using this fact and the isomorphism \labelcref{Easy Iso For Miura from Skryabin} we see that there is a natural isomorphism \begin{equation}\label{Natural Iso}\operatorname{act}_{\delta}^R(\delta \otimes_{Z\LG}(-)) \cong F\end{equation} and so $F$ acquires a monoidal structure from the monoidal structure on $\operatorname{act}_{\delta}^R(\delta \otimes_{Z\LG}(-))$ given by \cite[Proposition 2.23]{KalSaf}. We use this monoidal structure on $F$ in what follows.

There is also an action of the Weyl group $W$ on
  $D(\unipotentRadicalOfBorel\backslash G)$ by algebra automorphisms
    known as the \textit{Gelfand-Graev action},
see \cite[Section 3.2]{GinzburgRicheDifferentialOperatorsOnBasicAffineSpaceandtheAffineGrassmannian}, \cite{GinzburgKazhdanDifferentialOperatorsOnBasicAffineSpaceandtheGelfandGraevAction}. The subalgebra $U\g_\op\sset D(\unipotentRadicalOfBorel\backslash G)$
is fixed by this action. Furthermore, the actions of the groups $G_\op, T$,
and $W$, combined together give an action of the group $(T \rtimes W) \times G_{\text{op}}$
on $D(\unipotentRadicalOfBorel\backslash G)$ by algebra automorphisms. 

We let $W$ act on $D(T)$ via the \textit{dot}-action defined 
using the isomorphism
$\Symt \otimes \mathcal{O}(T) \xrightarrow{\cong} D(T)$ induced by multiplication
by the formula $w: a\o f \mto (w\cdot a) \o w^*(f)$, for all
$w\in W,\, a\in\Symt,\, f\in \mathcal{O}(T)$.
Here, $f\mto w^*(f)$ is the usual action on the functions and
$a\mto w\cdot a$ is the $\rho$-shifted $W$-action on $\Symt$, the `dot-action';
see also \cite[Section 1.1]{Gin} for an alternate interpretation of the $W$-action on $D(T)$
defined in this way.
The action 
of $T \rtimes W$ on ${D}(\unipotentRadicalOfBorel\backslash G)$
survives the Whittaker reduction, so the $(\differentialOperatorsOnT, \biWhittakerDifferentialOperatorsonG)$-bimodule, resp. 
$(\biWhittakerDifferentialOperatorsonG, \differentialOperatorsOnT)$-bimodule,
structure on
$\Mi$, resp. $\check\Mi$, can be upgraded to a
$(\differentialOperatorsOnT\rtimes W,\, \biWhittakerDifferentialOperatorsonG)$-bimodule, resp. 
$(\biWhittakerDifferentialOperatorsonG,\, D(T)\rtimes W)$-bimodule, structure. 
\subsection{Parabolic induction and the Miura bimodule}\label{New parabolic induction and Miura}
The functor of parabolic induction
is defined as a pull-push functor $\int_p q^*(\td):  
{\scr D}^W(T)\to {\scr D}^{\Ad G}(G)$ 
with respect to the diagram
\begin{equation}\label{spr diag}
  G\, \xleftarrow{\ gbg\inv \leftarrowtail (g,b) \, :\,  p\ } \,\tilde{G}= G\times^B  B\,
  \xrightarrow{\ q:\, (g,b) \rightarrowtail bN/N \ }\, B/N=T,
  \end{equation}
where the Borel subgroup $B$ acts on itself by conjugation
and the map 
$p$ is the Grothendieck-Springer morphism.
It is known that parabolic induction is an exact functor \cite{BezrukavnikovYomDinOnParabolicRestrictionofPerverseSheaves} and we let $\IndTG: \differentialOperatorsOnT \rtimes W\modwithdash \xrightarrow{} \differentialOperatorsOnG\modwithdash^{{\Ad G}}$ denote the induced functor between the corresponding abelian
hearts.

It was shown in \cite{GinzburgParabolicInductionandtheHarishChandraDModule} that
the functor $\IndTG$ has an algebraic description in terms of  the
$D(G)$-module
$\bN=D(G)/D(G)\ad\g$, cf. \cref{N coalg}.
In more detail,  Harish-Chandra proved that the algebra
$(D(G)/D(G)\ad\g)^{\Ad G}$ obtained from $D(G)$ by Hamiltonian reduction
with respect to the adjoint action of $G$ is
isomorphic to $\differentialOperatorsOnT^W$ via the
radial parts isomorphism,  see \cref{main1}.
Therefore, the left action of $D(G)$ and the natural right action of $(D(G)/D(G)\ad\g)^{\Ad G}$ on
$D(G)/D(G)\ad\g$ give $\bN$ the structure  of a
$(D(G), \differentialOperatorsOnT^W)$-bimodule.

Let $\bM:=\bN\o_{\differentialOperatorsOnT^W} D(T)$.
This tensor product has the structure of
a $(D(G),\, D(T)\rtimes W)$-bimodule where the action of $W$ comes from the
natural $W$-action
on the second tensor factor. Let $\mathcal{M}$ denote the $D_G$-module whose global sections are $\bM$. There are isomorphisms of functors, cf.
\cite[Theorem 3.9]{GinzburgParabolicInductionandtheHarishChandraDModule}:
\begin{equation}\label{Tensoring with quantizationOfCentralizers Is Parabolic Induction No W Invariants}
\operatorname{Ind}_T^G(\td)\cong
\mathbf{M} \otimes_{\differentialOperatorsOnT} (\td) \cong \quantizationOfCentralizersHBAREQUALSONE \otimes_{\differentialOperatorsOnT^W} \differentialOperatorsOnT \otimes_{\differentialOperatorsOnT} (\td).
\end{equation}
Taking $W$-invariants one deduces an isomorphism
\begin{equation}\label{Tensoring with quantizationOfCentralizers Is Parabolic Induction}
    \mathbf{N} \otimes_{\differentialOperatorsOnT^W} (\td) \cong \IndTG(\td)^W.
  \end{equation}

  The right $D(T)^W$-action on the $D(G)$-module $\mathbf{N}$
induces, by functoriality,  an algebra map $D(T)^W\to (\End_{\J}\kappa(\mathbf{N}))_{op}$.
Similarly,  one can apply $\kap(\td)$ to
$\mathbf{M}$ viewed as a $D(G)$-module. Then, the action of the algebra
$D(T)\rtimes W$ on
$\mathbf{M}$ induces  an algebra map $D(T)\rtimes W\to (\End_{\J}\kappa(\mathbf{M}))_{op}$.
This makes $\kappa(\mathbf{N})$ a
$(\J,\differentialOperatorsOnT^W)$-bimodule
and $\kappa(\mathbf{M})$ a $(\J,\,D(T)\rtimes W)$-bimodule, respectively. The following result, when combined with \labelcref{Tensoring with quantizationOfCentralizers Is Parabolic Induction}, gives a natural isomorphism $\kappa(\IndTG(\td)^W) \cong \check{\mathbb{M}}^W \otimes_{\D(T)} (\td)$:
    
\begin{Proposition}\label{Equivalence of Bimodules for KostantWhittaker for N and Miura Fixed Points Summer 2025 Version}
  There is an isomorphism $\kappa(\mathbf{M})\cong \check{\mathbb{M}}$,
  resp. $\kappa(\quantizationOfCentralizersHBAREQUALSONE) \cong \check{\mathbb{M}}^W$, 
  of $(\J,\,D(T)\rtimes W)$-bimodules, 
  resp. $(\J,\differentialOperatorsOnT^W)$-bimodules; moreover, $\kappa(\quantizationOfCentralizersHBAREQUALSONE)$ is a free rank one $\biWhittakerDifferentialOperatorsonG$-module.
\end{Proposition}

To prove \cref{Equivalence of Bimodules for KostantWhittaker for N and Miura Fixed Points Summer 2025 Version}, we first show the following lemma. In it, we write $a_{\ell}: \unipotentRadicalOfOppositeBorel \times G \to G$ for the map $a_{\ell}(m, g) := mg$ and, for the ease of exposition we write $(-)^{\unipotentRadicalOfOppositeBorel}$ for the \textit{derived} invariants functor.
\newcommand{\genericCharacterOfUnipotentRadicalOfOppositeBorel}{\mathring{\psi}}
\begin{Lemma}\label{Lem: Averaging is Coinvariants}
    For any $\F \in D_G$-mod and any character $\genericCharacterOfUnipotentRadicalOfOppositeBorel: \unipotentRadicalOfOppositeBorel \to \mathbb{G}_a$, there is an isomorphism \[R\Gamma(\int_{a_{\ell}}\O_{\unipotentRadicalOfOppositeBorel}^{\genericCharacterOfUnipotentRadicalOfOppositeBorel} \otimes_k \F)^{\unipotentRadicalOfOppositeBorel} \cong \LieAlgebraofUnipotentRadicalOfOppositeBorel^{\genericCharacterOfUnipotentRadicalOfOppositeBorel}U\LG\backslash U\LG \otimes_{U\LG}^L \Gamma(\F)\]  in the derived category of $D(\unipotentRadicalOfOppositeBorel_{\genericCharacterOfUnipotentRadicalOfOppositeBorel}\backslash G)$-modules which is natural in $\F$. 
\end{Lemma}

\begin{Remark}
    An analogous argument gives an isomorphism \[R\Gamma(\int_{\mathrm{act}: N \times X \to X}\O_{\unipotentRadicalOfOppositeBorel}^{\genericCharacterOfUnipotentRadicalOfOppositeBorel} \otimes_k \F)^{\unipotentRadicalOfOppositeBorel} \cong k_{\genericCharacterOfUnipotentRadicalOfOppositeBorel} \otimes^L_{U\LieAlgebraofUnipotentRadicalOfOppositeBorel} R\Gamma(\F)\] of the derived $\LieAlgebraofUnipotentRadicalOfOppositeBorel^{\genericCharacterOfUnipotentRadicalOfOppositeBorel}$-coinvariants of any complex $\F$ of $D_X$-modules for $X$ any variety with an action of $\unipotentRadicalOfOppositeBorel$. 
\end{Remark}
\begin{proof} For such $\F$, we have isomorphisms \[R\Gamma(\int_{a_{\ell}}\O_{\unipotentRadicalOfOppositeBorel}^{\genericCharacterOfUnipotentRadicalOfOppositeBorel} \otimes_k \F)^{\unipotentRadicalOfOppositeBorel} \cong R\Gamma((\int_{a_{\ell}}\O_{\unipotentRadicalOfOppositeBorel}^{\genericCharacterOfUnipotentRadicalOfOppositeBorel} \otimes_k \F)^{\unipotentRadicalOfOppositeBorel}) \cong (\O(\unipotentRadicalOfOppositeBorel) \otimes^L_{U\LieAlgebraofUnipotentRadicalOfOppositeBorel} \Gamma(\F))^{\unipotentRadicalOfOppositeBorel}\] since the derived invariants and the derived global sections commute and from the isomorphism \labelcref{bfo}, respectively. We moreover have \[(\O(\unipotentRadicalOfOppositeBorel) \otimes^L_{U\LieAlgebraofUnipotentRadicalOfOppositeBorel} \Gamma(\F))^{\unipotentRadicalOfOppositeBorel} \cong \O(\unipotentRadicalOfOppositeBorel)^{\unipotentRadicalOfOppositeBorel} \otimes^L_{U\LieAlgebraofUnipotentRadicalOfOppositeBorel} \Gamma(\F) \cong k \otimes^L_{U\LieAlgebraofUnipotentRadicalOfOppositeBorel} \Gamma(\F) \cong \LieAlgebraofUnipotentRadicalOfOppositeBorel^{\genericCharacterOfUnipotentRadicalOfOppositeBorel} U\LG\backslash U\LG \otimes_{U\LG}^L \Gamma(\F)\] and so combining this chain of isomorphisms we obtain our desired claim.
\end{proof}

\begin{proof}[Proof of \cref{Equivalence of Bimodules for KostantWhittaker for N and Miura Fixed Points Summer 2025 Version}]Let $\nu: \BorelSubgroup \to \BorelSubgroup/\unipotentRadicalOfBorel = T$
  be the quotient map,
  $\iota: \BorelSubgroup \xhookrightarrow{} G$  the closed embedding. Let $a_r: G \times \unipotentRadicalOfOppositeBorel \to G$ denote the action map $(g, m) \mapsto gm$, and let $j: \unipotentRadicalOfOppositeBorel \times \BorelSubgroup \xhookrightarrow{} \tilde{G}$ denote the open embedding induced by the Bruhat decomposition. Let $\varphi := (a_r \times \mathrm{id}) \circ (p \times \mathrm{id}_B \times q)$ and let \[\phi: \unipotentRadicalOfOppositeBorel \times \BorelSubgroup \times \unipotentRadicalOfOppositeBorel \to \unipotentRadicalOfOppositeBorel \times \BorelSubgroup \times \unipotentRadicalOfOppositeBorel\text{, resp. } a_{\ell r}: \unipotentRadicalOfOppositeBorel \times \BorelSubgroup \times \unipotentRadicalOfOppositeBorel \to G\] denote the map given by $\phi(n,g,m) := (n,g,n^{-1}m)$, resp. $a_{\ell r}(n,g,m) := ngm$. 
Observe that, with this notation, the diagram 
  \begin{equation}\label{Replacement for Big Commutative Diagram}\xymatrix@R+2em@C+2em{\unipotentRadicalOfOppositeBorel \times \BorelSubgroup \times \unipotentRadicalOfOppositeBorel \ar[r]^{j \times \mathrm{id}} \ar[d]^{\phi} & \tilde{G} \times \unipotentRadicalOfOppositeBorel \ar@/^1.5pc/[rr]^{\varphi} \ar[r]^{p \times \mathrm{id} \times q} & G \times \unipotentRadicalOfOppositeBorel \times T \ar[r]^{a_r \times \mathrm{id}} & G \times T\\ \unipotentRadicalOfOppositeBorel \times \BorelSubgroup \times \unipotentRadicalOfOppositeBorel \ar[urrr]_{\textcolor{white}{white}a_{\ell r}, \nu \circ \mathrm{pr}_{\BorelSubgroup}} &
 }
  \end{equation}
commutes. Moreover, the argument of \cite[Section 4.6, Step 3]{RaskinAffineBBLocalization} shows that the morphism \begin{equation}\int_{\varphi}\O_{\tilde{G}} \otimes_k \mathcal{O}_{\unipotentRadicalOfOppositeBorel}^{\psi} \to \int_{\varphi}j_*(\O_{\unipotentRadicalOfOppositeBorel} \otimes_k \O_{\BorelSubgroup}) \otimes_k \O_{\unipotentRadicalOfOppositeBorel}^{\psi}\end{equation} induced by restriction of sections is an isomorphism of bi-Whittaker $D$-modules on $G \times T$. In particular, we obtain an induced isomorphism \begin{equation}\label{Map Induced by the Unit of the Adjunction for j on Global Sections}(R\Gamma\int_{\varphi}\O_{\tilde{G}} \otimes_k \mathcal{O}_{\unipotentRadicalOfOppositeBorel}^{\psi})^{\unipotentRadicalOfOppositeBorel \times \unipotentRadicalOfOppositeBorel_{\mathrm{op}}} \xrightarrow{\sim} (R\Gamma\int_{\varphi}j_*(\O_{\unipotentRadicalOfOppositeBorel} \otimes_k \O_{\BorelSubgroup}) \otimes_k \O_{\unipotentRadicalOfOppositeBorel}^{\psi})^{\unipotentRadicalOfOppositeBorel \times \unipotentRadicalOfOppositeBorel_{\mathrm{op}}}\end{equation} on invariants of the global sections. 

We observe that there are isomorphisms \[(R\Gamma\int_{\varphi}\O_{\tilde{G}} \otimes_k \mathcal{O}_{\unipotentRadicalOfOppositeBorel}^{\psi})^{\unipotentRadicalOfOppositeBorel \times \unipotentRadicalOfOppositeBorel_{\mathrm{op}}}  \cong (R\Gamma\int_{a_r \times \mathrm{id}}\int_{p \times \mathrm{id} \times q}\O_{\tilde{G}} \otimes_k \mathcal{O}_{\unipotentRadicalOfOppositeBorel}^{\psi})^{\unipotentRadicalOfOppositeBorel \times \unipotentRadicalOfOppositeBorel_{\mathrm{op}}} \cong R\Gamma(\int_{a_r \times \mathrm{id}}\mathcal{M} \otimes_k \mathcal{O}_{\unipotentRadicalOfOppositeBorel}^{\psi})^{\unipotentRadicalOfOppositeBorel \times \unipotentRadicalOfOppositeBorel_{\mathrm{op}}}\] \[\cong (\mathbf{M}/\mathbf{M}\LieAlgebraofUnipotentRadicalOfOppositeBorelPSIr)^{N} =: \kappa(\mathbf{M})\] by applying, respectively, the functoriality of pushforward, \cite[Theorem 2.2]{GinzburgParabolicInductionandtheHarishChandraDModule}, and \cref{Lem: Averaging is Coinvariants}. We have thus identified the domain of the isomorphism \labelcref{Map Induced by the Unit of the Adjunction for j on Global Sections} with $\kappa(\mathbf{M})$.

We now identify the codomain of \labelcref{Map Induced by the Unit of the Adjunction for j on Global Sections} with the Miura bimodule. Observe that the commutativity of the diagram \labelcref{Replacement for Big Commutative Diagram} and the functoriality of pushforward implies that \[\int_{\varphi}j_*(\O_{\unipotentRadicalOfOppositeBorel} \otimes_k \O_{\BorelSubgroup}) \otimes_k \O_{\unipotentRadicalOfOppositeBorel}^{\psi} \cong \int_{a_{\ell r}  \times \nu \circ \mathrm{pr}_{\BorelSubgroup}}\int_{\phi}\O_{\unipotentRadicalOfOppositeBorel} \otimes_k \O_{\BorelSubgroup} \otimes_k \O_{\unipotentRadicalOfOppositeBorel}^{\psi} \cong \int_{a_{\ell r}  \times \nu \circ \mathrm{pr}_{\BorelSubgroup}}\O_{\unipotentRadicalOfOppositeBorel}^{\psi} \otimes_k \O_{\BorelSubgroup} \otimes_k \O_{\unipotentRadicalOfOppositeBorel}^{\psi}\] and so if we take invariants of global sections we see that \[(\Gamma\int_{a_{\ell r}  \times \nu \circ \mathrm{pr}_{\BorelSubgroup}}\O_{\unipotentRadicalOfOppositeBorel}^{\psi} \otimes_k \O_{\BorelSubgroup} \otimes_k \O_{\unipotentRadicalOfOppositeBorel}^{\psi})^{\unipotentRadicalOfOppositeBorel \times \unipotentRadicalOfOppositeBorel_{\mathrm{op}}} \cong \Gamma (\int_{\unipotentRadicalOfOppositeBorel \times G \times T \to G \times T}(\O_{\unipotentRadicalOfOppositeBorel}^{\psi} \otimes_k \Dpsir/\Dpsir\LieAlgebraOfUnipotentRadicalOfBorel))^{\unipotentRadicalOfBorel} \cong \LieAlgebraofUnipotentRadicalOfOppositeBorelPSIl\backslash \Dpsir/\Dpsir\LieAlgebraOfUnipotentRadicalOfBorel =: \MiuraBimodule\] by direct computation and \cref{Lem: Averaging is Coinvariants} for the group $G \times T$ respectively, as desired.

This constructs a $(\biWhittakerDifferentialOperatorsonG, D(T))$-bimodule isomorphism $\iota: \kappa(\mathbf{M}) \xrightarrow{\sim} \MiuraBimodule$. We temporarily use the notation $R := \biWhittakerDifferentialOperatorsonG \otimes D(T)_{\mathrm{op}}$, so that $\iota$ is an isomorphism of $R$-modules. It remains to show that $\iota$ is $W$-equivariant. However, observe that both $\kappa(\mathbf{M})$ and $\MiuraBimodule$ are generated as an $R$-module by a single element, $1 \in \kappa(\mathbf{M})$ and $\iota(1) \in \MiuraBimodule$, both of which are fixed by $W$. Since any element in $\kappa(\mathbf{M})$ can thus be written as $r \cdot 1$, the $W$-equivariance of the isomorphism immediately follows from the fact that both $\kappa(\mathbf{M})$ and $\MiuraBimodule$ have $R \rtimes W$-module structures: explicitly, \[\iota(wr\cdot 1) = \iota(w(r)w\cdot 1) = \iota(w(r)\cdot 1) = w(r)\iota(1) = wrw^{-1} \cdot\iota(1) = wr \cdot \iota(1)\] for any $r \in R$. Since any element of $\bM$ can be written in the form $r \cdot 1$, we immediately obtain that $\iota$ is $W$-equivariant.
\end{proof}

\subsubsection{Kostant-Whittaker reduction of $\bfN$ is free of rank one}\label{Kostant Whittaker reduction of bfN is free of rank one}
\newcommand{\ENGtoBiwhit}{E_{\LNbarpsir\bN\backslash\bN}}
\newcommand{\LNbarpsil}{\LieAlgebraofUnipotentRadicalOfOppositeBorelPSIl}
\newcommand{\LNbarpsir}{\LieAlgebraofUnipotentRadicalOfOppositeBorelPSIr}
\newcommand{\LNbarpsi}{\LieAlgebraofUnipotentRadicalOfOppositeBorelPSI}

\cref{Equivalence of Bimodules for KostantWhittaker for N and Miura Fixed Points Summer 2025 Version} in particular shows that $\kostantWhittakerToWComod(\bfN)$ is free of rank one as a $\biWhittakerDifferentialOperatorsonG$-module. This can also be derived in a more elementary way. We prove this here, and prove similar facts which can be derived by analogous methods. The results of \cref{Kostant Whittaker reduction of bfN is free of rank one} will not be used in what follows.

\begin{Proposition}
    The map $\overline{\overline{\partial}} \mapsto \overline{\partial} + \LieAlgebraofUnipotentRadicalOfOppositeBorelPSIr\bN$, resp. $\overline{\overline{\partial}} \mapsto \overline{\partial} + (\LieAlgebraOfUnipotentRadicalOfBorel + \LieAlgebraofUnipotentRadicalOfOppositeBorelPSIr)\bN$ induce isomorphisms \[\biWhittakerDifferentialOperatorsonG \xleftarrow{\sim} \kostantWhittakerForNONFILTEREDtoZGBIMOD(\bN)\text{, resp. } (\LieAlgebraOfUnipotentRadicalOfBorel + \LieAlgebraofUnipotentRadicalOfOppositeBorelPSIr)\bN\backslash \bN \xleftarrow{\sim} \MiuraBimodule\] of $\biWhittakerDifferentialOperatorsonG$-modules, resp. $D(T)$-modules.
\end{Proposition}

This result immediately follows by taking the invariants, respectively left coinvariants, of the following isomorphism for the $\LieAlgebraofUnipotentRadicalOfOppositeBorelPSIl$, respectively $\LieAlgebraOfUnipotentRadicalOfBorel$:

\begin{Lemma}
    The map $\Dpsir \to \LieAlgebraofUnipotentRadicalOfOppositeBorelPSIr\bN\backslash\bN$ given by the formula $\partial \mapsto \overline{\partial} + \LieAlgebraofUnipotentRadicalOfOppositeBorelPSIr\bN$ induces an isomorphism $\Dpsir/\Dpsir\LieAlgebraofUnipotentRadicalOfOppositeBorelPSIl \xrightarrow{\sim} \LieAlgebraofUnipotentRadicalOfOppositeBorelPSIr\bN\backslash\bN$ of left $\Dpsir$-modules, where the $\Dpsir$-module structure on $\LieAlgebraofUnipotentRadicalOfOppositeBorelPSIr\bN\backslash\bN$ is induced by the $D$-module structure on $\bN$.
\end{Lemma}

\begin{proof}
Observe that the fact \labelcref{N as Fiber of D} is an isomorphism of $D$-modules implies the induced isomorphism \begin{equation}\label{Unraveled Definition of N}\mathfrak{n}_r^{\psi}\mathbf{N}\backslash\mathbf{N} \cong \mathfrak{n}_r^{\psi}D\backslash D \otimes_{U\otimes_{Z} U_{\mathrm{op}}}U\end{equation} is an isomorphism of $\Dpsir$-modules.  We have an isomorphism \[\mathfrak{n}_r^{\psi}D\backslash D \xleftarrow{\sim} D^{\psi_r} \otimes_{U \otimes_Z Z} (U \otimes_Z \mathfrak{n}^{\psi}U_{\mathrm{op}}\backslash U_{\mathrm{op}})\] of $(\Dpsir, U \otimes_Z U_{\mathrm{op}})$-bimodules by Skryabin's equivalence; therefore we obtain an isomorphism \begin{equation}\label{Induced Skryabin Isomorphism for N}\mathfrak{n}_r^{\psi}D\backslash D \otimes_{U\otimes_{Z} U_{\mathrm{op}}}U \xleftarrow{\sim} D^{\psi_r} \otimes_{U \otimes_Z Z} (U \otimes_Z \mathfrak{n}^{\psi}U_{\mathrm{op}}\backslash U_{\mathrm{op}}) \otimes_{U \otimes_Z U_{\mathrm{op}}} U\end{equation} of $\Dpsir$-modules by applying $(-) \otimes_{U \otimes_Z U_{\mathrm{op}}} U$ to this isomorphism. Finally, observe that the map $U \to (U \otimes_Z \mathfrak{n}^{\psi}U_{\mathrm{op}}\backslash U_{\mathrm{op}}) \otimes_{U \otimes_Z U_{\mathrm{op}}} U$ of left $U$-modules determined by $1 \mapsto 1 \otimes 1 \otimes 1$ induces an isomorphism of left $U$-modules \[U/U\mathfrak{n}^{\psi} \xrightarrow{\sim} (U \otimes_Z \mathfrak{n}^{\psi}U_{\mathrm{op}}\backslash U_{\mathrm{op}}) \otimes_{U \otimes_Z U_{\mathrm{op}}} U.\] We therefore obtain an isomorphism \begin{equation}\label{Iso Folding Quotient for N}D^{\psi_r} \otimes_{U \otimes_Z Z} (U \otimes_Z \mathfrak{n}^{\psi}U_{\mathrm{op}}\backslash U_{\mathrm{op}}) \otimes_{U \otimes_Z U_{\mathrm{op}}} U \xleftarrow{\sim} D^{\psi_r} \otimes_{U \otimes_Z Z} U/U\mathfrak{n}^{\psi} \cong D^{\psi_r}/D^{\psi_r}\mathfrak{n}^{\psi}_{\ell}\end{equation} by applying $D^{\psi_r} \otimes_U (-)$ to this isomorphism. The composite of \labelcref{Unraveled Definition of N}, \labelcref{Induced Skryabin Isomorphism for N}, and \labelcref{Iso Folding Quotient for N} is given by the map $\partial \mapsto \partial + \LieAlgebraofUnipotentRadicalOfOppositeBorelPSIr\bN$ and so we see that our map is an isomorphism of $\Dpsir$-modules, as desired.
\end{proof}

\subsection{The Miura bimodule and comodules}\label{Miura Bimodule and Comodules Subsection}
Let $\fieldOfFractionsForSymt$ denote the field of fractions of $\Symt$.
Given a $\Symt$-module $M$, we write $M_{\loc} := \fieldOfFractionsForSymt \otimes_{\Symt} M$. When clear from context, we also use the same notation $M_{\loc} := \fieldOfFractionsForSymt \otimes_{Z\LG} M$ if $M$ is a $Z\LG$-module, and use the notation $f_{\loc}$ for the map induced by the tensor product of a map $f$ of modules for either $\Symt$ or $Z\LG$. 
Taking $W$-invariants of the isomorphism \labelcref{Easy Iso For Miura from Skryabin}, we see that $\MiuraBimodule^W$ is free of rank one as a right $\biWhittakerDifferentialOperatorsonG$-module. Thus, acting by $\differentialOperatorsOnT^W$ induces a ring morphism $\differentialOperatorsOnT^W \to \biWhittakerDifferentialOperatorsonG$. Since $1 \in \MiuraBimodule$ is fixed by $W$, we also obtain an induced $W$-equivariant map
\[\varphi: \differentialOperatorsOnT 
  \to \MiuraBimodule\]
where $W$ acts on $\differentialOperatorsOnT$ by the dot action as above.

\begin{Proposition}\label{Miura Bimodule is Cocommutative Coalgebra and Isomorphic to DT Upon Tensoring With Field of Fractions}
The $\Symt$-module structure on $\MiuraBimodule$ upgrades to the structure of a cocommutative coalgebra object in $(\differentialOperatorsOnT \rtimes W)\mmod$ such that $\varphi$ is a map of coalgebras. Moreover, the induced map $\fieldOfFractionsForSymt \otimes_{\Symt} \varphi$ is an isomorphism of coalgebras.
\end{Proposition}

\begin{proof} We have constructed a monoidal structure on $F$ in \cref{The Miura bimodule and the Gelfand-Graev action}, and so $F$ preserves coalgebra objects. Since $\biWhittakerDifferentialOperatorsonG$ is a coalgebra object in $\biWhittakerDifferentialOperatorsonG\mmod$ and $\MiuraBimodule \cong F(\biWhittakerDifferentialOperatorsonG)$ we see that $\MiuraBimodule$ is a coalgebra object of the category of modules for $\differentialOperatorsOnT \rtimes W$. We temporarily let $\Delta_{\differentialOperatorsOnT}$ denote the comultiplication in $\differentialOperatorsOnT$ and $\Delta_{\MiuraBimodule}$ denote the comultiplication for the Miura bimodule. To show $\varphi$ is a coalgebra map, we must show \begin{equation}\label{Equality to Prove A Is Coalgebra Map}\Delta_{\MiuraBimodule} \circ \varphi = (\varphi \otimes \varphi) \circ \Delta_{\differentialOperatorsOnT}\end{equation} and that $\varphi$ is compatible with counits. However, by construction, $\Delta_{\differentialOperatorsOnT}$ and $\Delta_{\MiuraBimodule}$ are $\differentialOperatorsOnT \rtimes W$-module maps, and $\varphi$ is a $(\differentialOperatorsOnT \rtimes W)$-module map since it is given by acting on some element. Moreover, for $\nu, \nu' \in \differentialOperatorsOnT$ and $\partial \in \differentialOperatorsOnT \rtimes W$ we have \begin{align*}
    (\varphi \otimes \varphi)(\partial \cdot (\nu \otimes \nu')) 
    &:= (\varphi \otimes \varphi)(\partial_{(0)}\nu \otimes \partial_{(1)}\nu')
    = \varphi(\partial_{(0)}\nu) \otimes \varphi(\partial_{(1)}\nu') = \partial_{(0)}\varphi(\nu) \otimes \partial_{(1)}\varphi(\nu') 
\end{align*} and by the definition of the $(\differentialOperatorsOnT \rtimes W)$-module structure on the tensor product we have \[\partial_{(0)}\varphi(\nu) \otimes \partial_{(1)}\varphi(\nu') = \partial \cdot (\varphi(\nu) \otimes \partial_{(1)}\varphi(\nu'))
    = \partial \cdot (\varphi \otimes \varphi)(\nu \otimes \nu')\] and so combining these we see that $\varphi \otimes \varphi$ is also a $(\differentialOperatorsOnT \rtimes W)$-module map. Thus both terms in \labelcref{Equality to Prove A Is Coalgebra Map} are $(\differentialOperatorsOnT \rtimes W)$-module maps which send $1$ to $1 \otimes 1$ and therefore are the same map. Compatibility with counits is similar.
    
  Using \cite[Proposition 8.1.1]{Gin}, one can show that the left $(\differentialOperatorsOnT \rtimes W)$-module structure on $\MiuraBimodule$
  upgrades to a module for the \textit{nil DAHA} $\nilDAHAAtHBARISONE$, a certain $\Symt$-ring object for which there is an an algebra embedding
    $\differentialOperatorsOnT \rtimes W\into \nilDAHAAtHBARISONE$ for which the morphism
   \begin{equation}\label{nilDaha Is Iso On Field of Fractions} (\differentialOperatorsOnT \rtimes W)_{\loc} \to  \nilDAHAAtHBARISONE_{\loc}
  \end{equation} induced by inclusion is an isomorphism, by construction. Moreover, one can use the arguments of the proof of \cite[Theorem 8.1.2(i)]{Gin}
  to show that the left $\nilDAHAAtHBARISONE$-module $\MiuraBimodule$ is isomorphic to $\nilDAHAAtHBARISONE e$ for a certain idempotent $e \in \nilDAHAAtHBARISONE$. From this, using the isomorphism \labelcref{nilDaha Is Iso On Field of Fractions} we obtain that 
  $\varphi_{\loc}$ is an isomorphism. 

  It remains to prove that the coalgebra structure on $\mathbb{M}$  is cocommutative. First observe that $\MiuraBimodule$ is flat as a $\Symt$-module: indeed, $\Symt\o_{Z\g}\biWhittakerDifferentialOperatorsonG$ is flat as a $\Symt$-module since $\biWhittakerDifferentialOperatorsonG$ is flat as a $Z\LG$-module by \cref{Dpsirhbar and BiWhitHbar Are Flat}, so the isomorphism \labelcref{Easy Iso For Miura from Skryabin} gives our desired flatness. Thus to prove cocommutativity of $\MiuraBimodule$ it suffices to prove the cocommutativity of $\MiuraBimodule_{\loc}$. 
  However, the isomorphism $\varphi_{\loc}$ identifies $\MiuraBimodule_{\loc}$
 with the cocommutative coalgebra $\D(T)_{\mathrm{loc}}$, and so we obtain our desired cocommutativity. 
\end{proof}

\subsection{Trivial comodules}
Let $\biWhittakerDifferentialOperatorsonG\comod_{\text{triv}}$ denote the full subcategory of graded $\biWhittakerDifferentialOperatorsonG$-comodules which are isomorphic to comodules ${C}$ such that $\coact(c)=1\o c$ for all $c\in C$, where
$\coact: C \to \biWhittakerDifferentialOperatorsonG\o_{Z\g}C$ is the coaction.
This subcategory is a monoidal subcategory, and so we may consider the lax monoidal functor \begin{equation}\label{KappaR restricted to triv}\biWhittakerDifferentialOperatorsonG\comod_{\text{\em triv}} \subseteq \biWhittakerDifferentialOperatorsonG\comod \xrightarrow{\kappa^R} \HarishChandraBimodulesForTHISGROUPDefaultsToG\end{equation} obtained by restricting the lax monoidal functor $\kappa^R$. We also have a functor obtained by the composite \begin{equation}\label{Forget Tensor Up}\biWhittakerDifferentialOperatorsonG\comod_{\text{\em triv}} \xrightarrow{\mathrm{Obl}^{\biWhittakerDifferentialOperatorsonG}_{Z\LG}} Z\LG\mmod \xrightarrow{U\LG \otimes_{Z\LG}(-)} \HarishChandraBimodulesForTHISGROUPDefaultsToG\end{equation} of the forgetful functor $\mathrm{Obl}^{\biWhittakerDifferentialOperatorsonG}_{Z\LG}$ (which factors through $Z\LG$-mod and is monoidal by \cref{Asymptotic Enhancement of All Preliminary Results}(ii)) and the monoidal functor $U\LG \otimes_{Z\LG}(-)$. We claim these functors are naturally isomorphic:

\begin{Lemma}\label{Lax Monoidal Structure on Kappahbar Is Monoidal on Trivial Subcategory}
  There is a natural isomorphism of the composite of the functors in \labelcref{KappaR restricted to triv} and the composite of the functors in \labelcref{Forget Tensor Up}. This natural isomorphism respects the monoidal structure in the obvious way. In particular, the lax monoidality morphism \[\kappa^R(M_1) \otimes_{U\LG} \kappa^R(M_2) \xrightarrow{} \kappa^R(M_1 \otimes_{Z\LG} M_2)\] obtained from \cref{Adjoint to Monoidal Functor is Appropriately Lax Monoidal} is an isomorphism if $M_1, M_2$ are trivial $\biWhittakerDifferentialOperatorsonG$-comodules.
\end{Lemma}

\begin{proof}
  For ${C} \in \biWhittakerDifferentialOperatorsonG\comod_{\text{triv}}$ we have the following maps
  \[
    \kap^R({C}) \ = \ \Dpsir \otimes_{Z\g}^{\biWhittakerDifferentialOperatorsonG} {C} \ =\ 
    (D^{\psi_r})^{\biWhittakerDifferentialOperatorsonG}\otimes_{Z\g} {C}
\     \xleftarrow{}\ U\g\otimes_{Z\g} {C}.\]
 Here, the second equality  holds by the triviality assumption on $C$ and the map on the right is a isomorphism by \cref{Fr lem} with $V := k$ the trivial representation. Thus, we obtain the desired isomorphism and, tracing through the constructions, one readily checks that this isomorphism is natural in $C$ and respects the monoidal structure.
\end{proof}

\begin{Remark} There is an asymptotic analogue of the above lemma with identical proof.
\end{Remark}

\begin{Corollary}\label{If Youre A Flat Biwhittaker Comodule Whose Induced Bimodule Structure Comes from a Zg Module Your Comodule Structure is Trivial}
  Let $C$ be a left $\biWhittakerDifferentialOperatorsonG$-comodule which is flat as a $Z\LG$-module and such that $Z\LG$-bimodule structure of
  \cref{Asymptotic Enhancement of All Preliminary Results}
  is symmetric in the sense of \cref{Definition of Ring Object Coring Object Symmetric and Left Comodule}, then the $\biWhittakerDifferentialOperatorsonG$-comodule structure is trivial. 
\end{Corollary}

\begin{proof} Let $K$ denote the kernel of the counit map $\eps: \J\onto Z\g$. We have a $Z\g$-module direct sum decomposition $\biWhittakerDifferentialOperatorsonG= Z\g\oplus K$ with splitting by the counit, so the coaction may be
  viewed as a map
  \[\coact: C\to (Z\g\o_{Z\g}C) \oplus (K\o_{Z\g}C).\]
  We claim that the image of this map is contained in the first direct summand.
  The statement of the corollary follows from the claim since an equation
  $\coact(c)=z\o c'$  for some $c,c'\in C$ and $z\in Z\g$ forces $z\o c' = 1\o c$,
by the identity axiom of comodules.

To prove the claim, observe that there are isomorphisms \[ (\J\o_{Z\g} C)_{\loc}\cong \J_{\loc} \o_{\fieldOfFractionsForSymt} C_{\loc}\text{ and }(K\o_{Z\g} C)_{\loc}\cong K_{\loc} \o_{\fieldOfFractionsForSymt} C_{\loc}\] which are consequences of the fact that $\fieldOfFractionsForSymt \otimes_{Z\LG} (-)$ is symmetric monoidal.  Thus one has the $\fieldOfFractionsForSymt$-module
  direct sum decomposition $\J_{\loc}=\fieldOfFractionsForSymt\oplus  K_{\loc}$
  and the localized map 
   \[\coact_{\loc}: C_{\loc}\to \J_{\loc}
  \o_{\fieldOfFractionsForSymt}  C_{\loc} = (\fieldOfFractionsForSymt\o_{\fieldOfFractionsForSymt}
  C_{\loc}) \oplus  (K_{\loc}\o_{\fieldOfFractionsForSymt} C_{\loc}).\]
    The assumption that $C$ be flat over $Z\g$ ensures that the natural map $\J\o_{Z\g} C\to \J_{\loc}\o_{\fieldOfFractionsForSymt} C_{\loc}$ is injective, so it suffices to prove that 
   the image of the map $\coact_{\loc}$ is contained in the  direct summand
  $\fieldOfFractionsForSymt\o_{\fieldOfFractionsForSymt}
   C_{\loc} $.

  We may view $C_{\loc}$ as a $\differentialOperatorsOnT_{\loc}$-comodule using the isomorphism of \cref{Miura Bimodule is Cocommutative Coalgebra and Isomorphic to DT Upon Tensoring With Field of Fractions}. By \cref{Asymptotic Enhancement of All Preliminary Results},
  this comodule has a canonical right $U\t$-action \eqref{comod1}
  which agrees with
  the right $Z\LG$-action on $C$ in the sense that, for any
  $q\in \fieldOfFractionsForSymt,\, c\in C$, and $z\in Z\g\subseteq U\t$, in
  $C_{\loc}=\fieldOfFractionsForSymt\o_{Z\g} C$ one has
  $(q\o c)z=q\o (cz)$.
  We see that the assumption that the $Z\LG$-bimodule structure on $C$
  be symmetric implies that  for any $z\in Z\g$ and $f\in C_{\loc}$ one
  has $zf=fz$.

  Now,  the $\differentialOperatorsOnT_{\loc}$-comodule $C_{\loc}$ has a weight decomposition
  $C_{\loc}=\oplus_{\lambda\in {\mathbb{X}}^*(T)} \, C_{\loc}^{(\la)}$ with respect to the $T$-action.
  It follows as in \cref{torus ex}  that we may write $C_{\loc} \cong C_{\loc}^{(\la)}$ where $C_{\loc}^{(\la)}$ is the subspace of elements $c$ with $\mathrm{coact}(c) = \lambda \otimes c$. 
  For any $\la\neq0$, 
  one can find $z\in Z\g$ such that $z\neq\tau_\la(z)$.
  For such a $z$ the equation $(z-\tau_\la(z))C_{\loc}^{(\la)}=0$
  forces $C_{\loc}^{(\la)}=0$, since $\fieldOfFractionsForSymt$ is a field  the equation $(z-\tau_\la(z))C_{\loc}^{(\la)}=0$
  forces $C_{\loc}^{(\la)}=0$. We conclude that $C_{\loc}=C_{\loc}^{(0)}$, so the 
  $\differentialOperatorsOnT_{\loc}$-comodule structure on $C_{\loc}$ is trivial; in particular, the image of $\coact_{\loc}$ is contained in the  direct summand $\fieldOfFractionsForSymt\o_{\fieldOfFractionsForSymt}
   C_{\loc}$ as desired. 
 \end{proof}

We recall that the Kostant-Whittaker reduction of any $M \in \D(G)\mmod^{\mathrm{Ad}(G)}$ acquires \textit{both} a $\biWhittakerDifferentialOperatorsonG$-module structure (as discussed in \cref{Monoidal Structure on BiWhit Mod and Monoidality}) in addition to its $\biWhittakerDifferentialOperatorsonG$\textit{-comodule} structure constructed in \cref{Bi-Whittaker Reduction and Biwhittaker Comodules}. However, the $Z\LG$-bimodule structure on $\kappa(M)$ is symmetric for any such $M$. Therefore, from \cref{If Youre A Flat Biwhittaker Comodule Whose Induced Bimodule Structure Comes from a Zg Module Your Comodule Structure is Trivial}, we obtain the following corollary:

  \begin{Corollary}\label{biWhit Comodule Structure on KappaN Is Trivial} If  $M\in D(G)\mmod^{\Ad G}$ has the property that $\kappa(M)$ is flat as a $Z\g$-module
 then the $\biWhittakerDifferentialOperatorsonG$-comodule structure on  $\kappa(M)$ is trivial. In particular, $\kostantWhittakerToWComod(\quantizationOfCentralizersHBAREQUALSONE)$ is trivial as a $\biWhittakerDifferentialOperatorsonG$-comodule.
\end{Corollary}

\begin{proof}
The first claim follows immediately from the above discussion. To see that $\kostantWhittakerToWComod(\quantizationOfCentralizersHBAREQUALSONE)$ is flat  is flat as a $Z\LG$-module, since we have an isomorphism $\kappa(\bfN) \cong \check{\MiuraBimodule}^W$ given by \cref{Equivalence of Bimodules for KostantWhittaker for N and Miura Fixed Points Summer 2025 Version}, we may show $\check{\MiuraBimodule}^W$ is flat as a $Z\LG$-module. We observe that by applying the transpose of the isomorphism \labelcref{Easy Iso For Miura from Skryabin} (or running the same arguments) and taking $W$-invariants, we see that $\check{\MiuraBimodule}^W$ is free of rank one as a $\biWhittakerDifferentialOperatorsonG$-module. However, $\biWhittakerDifferentialOperatorsonG$ is flat as a $Z\LG$-module by taking the specialization at $\hbar = 1$ of \cref{Dpsirhbar and BiWhitHbar Are Flat}.\end{proof}

\section{Construction of Knop-Ng\^o functor}\label{Radj for Drinfeld Center via Relative Setting}

\newcommand{\repGStandIn}{\mathcal{R}}
\newcommand{\HarishChandraBimoduleStandin}{\mathcal{H}}
\newcommand{\biWhitComodStandin}{\mathcal{C}}
\newcommand{\freeFunctorStandin}{F}

\subsection{Relative Drinfeld center}\label{Generalities for Monoidal Functor Subsection}
Given a monoidal functor
$\freeFunctorStandin: \repGStandIn \to \HarishChandraBimoduleStandin$ we define its \textit{relative Drinfeld center} as follows. It is the category $\mathcal{Z}_{\freeFunctorStandin}(\HarishChandraBimoduleStandin)$ whose objects are pairs $(B, \{Z_{F(R)}\})$ where $Z_{F(R)}$ is an isomorphism $B \otimes \freeFunctorStandin(R) \xrightarrow{\sim} \freeFunctorStandin(R) \otimes B$ for every object $R$ of $\repGStandIn$ which is natural in $R$ and such that the following diagram commutes for all $R_1, R_2 \in \repGStandIn$:
 \begin{equation}\label{Central Structure is Compatible with Tensor Structure}\xymatrix@R+2em@C+2em{B \otimes \freeFunctorStandin(R_1) \otimes \freeFunctorStandin(R_2)\ar[r]^{Z_{\freeFunctorStandin(R_1)} \otimes \text{id}} \ar[d]^{\text{id} \otimes m_{R_1, R_2}}  & \freeFunctorStandin(R_1) \otimes B \otimes \freeFunctorStandin(R_2) \ar[r]^{\text{id} \otimes Z_{\freeFunctorStandin(R_2)}}  & \freeFunctorStandin(R_1) \otimes \freeFunctorStandin(R_2) \otimes B \ar[dl]^{m_{R_1, R_2} \otimes \text{id}}\\
B \otimes \freeFunctorStandin(R_1 \otimes R_2)  \ar[r]^{Z_{\freeFunctorStandin(R_1 \otimes R_2)}} & \freeFunctorStandin(R_1 \otimes R_2) \otimes B}
\end{equation} where $m$ is the monoidality isomorphism for $\freeFunctorStandin$, and morphisms \[(B^1, \{Z^1_{\freeFunctorStandin(R)}\}) \xrightarrow{} (B^2, \{Z^2_{\freeFunctorStandin(R)}\})\] are given by maps $f: B^1 \to B^2$ in $\HarishChandraBimoduleStandin$ which intertwine the maps $Z^i_{F(R)}$ in the sense that \[Z^2_{\freeFunctorStandin(R)}(f \otimes \text{id}_{\freeFunctorStandin(R)}) = (\text{id}_{\freeFunctorStandin(R)} \otimes f)Z^1_{\freeFunctorStandin(R)}\] for all $R \in \repGStandIn$. 

It is not difficult to check that $\Z_{\repGStandIn}(\HarishChandraBimoduleStandin)$ admits a monoidal structure, where monoidality is given by the formula \[(B^1, \{Z^1_{\freeFunctorStandin(R)}\}) \otimes (B^2, \{Z^2_{\freeFunctorStandin(R)}\}) := (B^1 \otimes B^2, \{Z^{1,2}_{\freeFunctorStandin(R)}\})\] with $Z^{1, 2}_{\freeFunctorStandin(R)} := (Z^1_{\freeFunctorStandin(R)} \otimes \text{id}) \circ (\text{id}_{B^1} \otimes Z^2_{\freeFunctorStandin(R)})$, and unit ($\mathbf{1}_{\HarishChandraBimoduleStandin}, \{Z^{\mathbf{1}}_R\})$ defined from the unit isomorphisms as in \cite[Chapter 7.13]{EtingofGelakiNikshychOstrikTensorCategories}.

In the special case where $\freeFunctorStandin: {\mathcal R}\to {\mathcal R}$ is the identity functor on $\HarishChandraBimoduleStandin$ this recovers the usual Drinfeld center
construction, so ${\mathcal Z}_{\Id}({\mathcal H})=:{\mathcal Z} ({\mathcal H})$
  is the Drinfeld center of ${\mathcal H}$ with braiding given by the formula \[Z^1_{B_2}: (B^1, \{Z^1_{\freeFunctorStandin(R)}\}) \otimes (B^2, \{Z^2_{\freeFunctorStandin(R)}\})\xrightarrow{\sim}(B^2, \{Z^2_{\freeFunctorStandin(R)}\}) \otimes (B^1, \{Z^1_{\freeFunctorStandin(R)}\}).\] Unlike the Drinfeld center, 
the relative Drinfeld center is not braided in general.

We record the following observation, which is a routine check:

\begin{Proposition}\label{With a Monoidal Functor Theres (1)An Induced Functor on Relative Drinfeld Centers} Assume $\kostantWhittakerSTANDIN: \HarishChandraBimoduleStandin \to \biWhitComodStandin$ is a monoidal functor. There is an induced monoidal functor \[\underline{\kostantWhittakerSTANDIN}: \mathcal{Z}_{\freeFunctorStandin}(\HarishChandraBimoduleStandin) \to \mathcal{Z}_{\kostantWhittakerSTANDIN\freeFunctorStandin}(\biWhitComodStandin)\] by the formula $(B, \{Z_{\freeFunctorStandin(R)}\}) \mapsto (\kostantWhittakerSTANDIN(B), \{Z_{\kostantWhittakerSTANDIN\freeFunctorStandin(R)}\})$ where by definition $Z_{\kostantWhittakerSTANDIN\freeFunctorStandin(R)}$ is the isomorphism $\kostantWhittakerSTANDIN(B) \otimes \kostantWhittakerSTANDIN\freeFunctorStandin(R) \xrightarrow{\sim} \kostantWhittakerSTANDIN\freeFunctorStandin(R) \otimes \kostantWhittakerSTANDIN(B)$ making the diagram \begin{equation}\label{Diagram Defining Central Structure After Post Composing}\xymatrix@R+2em@C+2em{\kostantWhittakerSTANDIN(B \otimes \freeFunctorStandin(R)) \ar[d]^{\sim} \ar[r]^{\kostantWhittakerSTANDIN(Z_{R})} & \kostantWhittakerSTANDIN(\freeFunctorStandin(R) \otimes B) \ar[d]^{\sim} \\
\kostantWhittakerSTANDIN(B) \otimes \kostantWhittakerSTANDIN(\freeFunctorStandin(R)) \ar[r]^{Z_{\kostantWhittakerSTANDIN(R)}} & \kostantWhittakerSTANDIN(\freeFunctorStandin(R))\otimes \kostantWhittakerSTANDIN(B)} 
\end{equation} commute, where the vertical arrows are the monoidality isomorphisms.
\end{Proposition} 

\begin{Remark}\label{Etingof Remark} One can also construct an equivalence of categories $\Z_{\freeFunctorStandin}(\HarishChandraBimoduleStandin) \cong \mathrm{Hom}_{\repGStandIn \boxtimes \repGStandIn}^{\otimes}(\repGStandIn, \HarishChandraBimoduleStandin)$ in settings where $\repGStandIn \boxtimes \repGStandIn$ is defined. This gives a natural interpretation of the functor in \cref{With a Monoidal Functor Theres (1)An Induced Functor on Relative Drinfeld Centers}: namely, the functor $\underline{K}$ is given by the formula $\vartheta \mapsto \vartheta \circ K$ for every $\vartheta \in \mathrm{Hom}_{\repGStandIn \boxtimes \repGStandIn}^{\otimes}(\repGStandIn, \HarishChandraBimoduleStandin)$. This fact will not be used in this paper what but provides motivation and context for our constructions.\footnote{We thank Pavel Etingof for explaining these facts to us.}

\end{Remark}

\subsection{Construction of Knop-Ng\^o functor}\label{Definition of Ngo Functor Subsection} 

Our construction is based on the following two theorems.

\begin{Theorem}\label{theorem asserting abstract central functor in biWhittaker case}
Any $M \in \biWhittakerDifferentialOperatorsonG$-mod, viewed as a trivial $\biWhittakerDifferentialOperatorsonG$-comodule, gives an element of the Drinfeld center of $\biWhittakerDifferentialOperatorsonG$-comod whose central structure $M \otimes C \xrightarrow{\sim} C \otimes M $ has inverse given by the formula $c \otimes m \mapsto c_{(0)}m \otimes c_{(1)}$. This assignment upgrades to a fully faithful, braided monoidal functor $\biWhittakerDifferentialOperatorsonG\modwithdash \to \Z(\biWhittakerDifferentialOperatorsonG\comod)$.
\end{Theorem}

Theorem \ref{theorem asserting abstract central functor in biWhittaker case} will be proved in \cref{Hopf Algebroids and Central Functors Section}.
The second theorem is the following general result on functoriality of Drinfeld center:

\begin{Theorem}\label{Functoriality of Drinfeld Center Under Strong Generation Assumption}
Let $\repGStandIn, \HarishChandraBimoduleStandin$ and  $\biWhitComodStandin$
  be cocomplete abelian categories equipped with monoidal structures which commute with colimits separately in each entry and let
\begin{equation}\label{Functors for Standin of Free and KostantWhittaker}\repGStandIn \xrightarrow{\freeFunctorStandin} \HarishChandraBimoduleStandin \xrightarrow{\kostantWhittakerSTANDIN} \biWhitComodStandin\end{equation}
a pair of monoidal functors between them such that the following holds:

\vi  Every object in $\repGStandIn$ is a filtered colimit of projective dualizable objects.

\vii The functor $\freeFunctorStandin$ commutes with colimits and its essential image forms a projective generating set.

\viii The functor $\kostantWhittakerSTANDIN$ admits a fully faithful right adjoint $\kostantWhittakerSTANDIN^R$ which commutes with filtered colimits.
\vskip 3pt 
 Then, there is a braided monoidal functor $\underline{K}: \Z(\HarishChandraBimoduleStandin) \to \Z(\biWhitComodStandin)$ which lifts the monoidal functor $\kostantWhittakerSTANDIN$, and this braided monoidal functor admits a fully faithful, braided lax monoidal right adjoint $\underline{K}^R$ which lifts the lax monoidal functor $\kostantWhittakerSTANDIN^R$. Moreover, the unit and counit map for the adjoint pair $(\underline{K}, \underline{K}^R)$ are induced by the unit and counit map for $\kostantWhittakerSTANDIN$.
\end{Theorem}

Using \cref{theorem asserting abstract central functor in biWhittaker case}, we define the \textit{Knop-Ng\^o functor} $\ups$ to be the composite of the braided lax monoidal functors \[\biWhittakerDifferentialOperatorsonG\modwithdash \to \Z(\biWhittakerDifferentialOperatorsonG \comod) \xrightarrow{\underline{\kappa}^R}\Z(\HarishChandraBimodulesForTHISGROUPDefaultsToG)\] where the first arrow is given by \cref{theorem asserting abstract central functor in biWhittaker case} and the lax braided monoidal functor $\underline{\kappa}^R$ is defined in \cref{Functoriality of Drinfeld Center Under Strong Generation Assumption}, which satisfies the hypotheses of \cref{Functoriality of Drinfeld Center Under Strong Generation Assumption} by \cref{Asymptotic Kostant Whittaker Reduction is Monoidal Localization}. Equipping a $\biWhittakerDifferentialOperatorsonG$-module $M$ with a trivial $\biWhittakerDifferentialOperatorsonG$-comodule structure via the counit map for $\biWhittakerDifferentialOperatorsonG$, we see that the underlying object of $\ups(M)$ is \[\kostantWhittakerToWComod^R(M) \cong \Dpsir \otimes_{Z\LG}^{\biWhittakerDifferentialOperatorsonG} B = D^{\psi_r, \biWhittakerDifferentialOperatorsonG }\otimes_{Z\LG} B \xleftarrow{\sim} U\LG \otimes_{Z\LG} B\] where the final isomorphism is given by \cref{Fr lem}. Thus this functor is given on objects by tensor product $U\LG \otimes_{Z\LG} (\td)$ and so the braided lax monoidal structure on $\ups$ in fact has the property of being braided monoidal, and is fully faithful as it is the composite of functors which are fully faithful by \cref{theorem asserting abstract central functor in biWhittaker case} and \cref{Functoriality of Drinfeld Center Under Strong Generation Assumption}. 
By light abuse of notation, we identify the Knop-Ng\^o functor with its preimage under the global sections isomorphism $\Gamma: \differentialOperatorsOnG\mmod^{{\Ad G}} \xrightarrow{\sim} \Z(\HarishChandraBimodulesForTHISGROUPDefaultsToG)$ of \cite[Lemma 3.8]{BezrukavnikovFinkelbergOstrikCharacterDModulesViaDrinfeldCenterofHarishChandraBimodules}. This completes the proof of \cref{Theorem Asserting Existence of Ngo Functor as Braided Monoidal Fully Faithful Functor to Center}(i). \qed
\begin{proof}[Proof of \cref{Morphism Which Quantizes the Ngo Morphism as Map of Coalgebra Objects in Central HCG}]\label{QuantizationOfNgoMorphismAsCoringSubsection}
In \cref{Equivalence of Bimodules for KostantWhittaker for N and Miura Fixed Points Summer 2025 Version}, we identified $\kostantWhittakerToWComod(\bfN) 
\cong \biWhittakerDifferentialOperatorsonG$ as $\biWhittakerDifferentialOperatorsonG$-modules and, by \cref{biWhit Comodule Structure on KappaN Is Trivial}, this identification is also an isomorphism of trivial $\biWhittakerDifferentialOperatorsonG$-comodules. Applying $\kostantWhittakerToWComod^R$ to this isomorphism, we define our map as the composite \begin{equation}\label{Composite Giving Quantization of Ngo}\quantizationOfCentralizersHBAREQUALSONE \xrightarrow{u(\quantizationOfCentralizersHBAREQUALSONE)} \kostantWhittakerToWComod^R\kostantWhittakerToWComod(\quantizationOfCentralizersHBAREQUALSONE ) \cong \kostantWhittakerToWComod^R(\biWhittakerDifferentialOperatorsonG ) \cong U\LG \otimes_{Z\LG} \biWhittakerDifferentialOperatorsonG \end{equation}
where the first map of coalgebra objects is the unit of the adjunction given by \cref{Asymptotic Kostant Whittaker Reduction is Monoidal Localization}, 
and the final isomorphism is given by the fact that $\kostantWhittakerToWComod^R$ applied to a trivial comodule is given by $U\LG \otimes_{Z\LG} (-)$, see \cref{Lax Monoidal Structure on Kappahbar Is Monoidal on Trivial Subcategory}. Since each usage of a natural transformation is a monoidal transformation of strongly monoidal functors, this map is a map of coalgebra objects. 
Now, setting $\freeFunctorStandin := \freeFunctorForNONFILTEREDHarishChandraBimodules$ and $\kostantWhittakerSTANDIN := \kostantWhittakerToWComod$, we see that by \cref{Functoriality of Drinfeld Center Under Strong Generation Assumption} this unit map can be upgraded to a map in $\Z(\HarishChandraBimodulesForTHISGROUPDefaultsToG)$, as required. This morphism is a map of a quotient of a free $\differentialOperatorsOnG$-module which sends $1$ to $1 \otimes 1$, and therefore is given by acting on $1 \otimes 1$. 
\end{proof}

\begin{Remark}The composite of morphisms \labelcref{Composite Giving Quantization of Ngo} giving rise to the quantization of the Knop-Ng\^o morphism as a map of \textit{Harish-Chandra bimodules} does not require any of our above results on the relative Drinfeld center. In particular, using our arguments, one can construct the morphism \labelcref{quant ngo alg}, as a map of Harish-Chandra bimodules, without defining the functor $\ups$.
\end{Remark}


\subsection{Adjoint from Projection Formulas}\label{Projection Formula Subsection} In this subsection we prove some results required for the proof of \cref{Functoriality of Drinfeld Center Under Strong Generation Assumption}. The proof of the theorem will be completed in \cref{With Fully Faithful Right Adjoint Subsection}.

Observe first that, for any $C \in \biWhitComodStandin$ and $B \in \HarishChandraBimoduleStandin$, we have maps \begin{equation}\label{One Side Projection Formula Map in Generality Changed Definition}\kostantWhittakerSTANDIN^R(C) \otimes_{\HarishChandraBimoduleStandin} B \xrightarrow{\text{id} \otimes u(B)} \kostantWhittakerSTANDIN^R(C) \otimes_{\HarishChandraBimoduleStandin} \kostantWhittakerSTANDIN^R\kostantWhittakerSTANDIN(B) \xrightarrow{} \kostantWhittakerSTANDIN^R(C\otimes_{\biWhitComodStandin} \kostantWhittakerSTANDIN(B))\end{equation} where the second map is given by the lax monoidality morphism for $\kostantWhittakerSTANDIN^R$. Similarly, we have maps \begin{equation}\label{Other Side Projection Formula Map in Generality Changed Definition}B \otimes_{\HarishChandraBimoduleStandin} \kostantWhittakerSTANDIN^R(C) \xrightarrow{u(B) \otimes \text{id}} \kostantWhittakerSTANDIN^R\kostantWhittakerSTANDIN(B) \otimes_{\HarishChandraBimoduleStandin} \kostantWhittakerSTANDIN^R(C) \xrightarrow{} \kostantWhittakerSTANDIN^R(\kostantWhittakerSTANDIN(B) \otimes C).\end{equation} 
We now show  that the arguments of \cite{FlakeLaugwitzPosurProjectionFormulasandInducedFunctorsonCentersofMonoidalCategories} imply a relative verison of \cite[Theorem B]{FlakeLaugwitzPosurProjectionFormulasandInducedFunctorsonCentersofMonoidalCategories}:

\begin{Proposition}\label{Relative Version of Existence of Right Adjoint for Drinfeld Centers Via Projection Formula}
    The functor \[\underline{\kostantWhittakerSTANDIN}: \Z_{\freeFunctorStandin}(\HarishChandraBimoduleStandin) \to \Z_{\kostantWhittakerSTANDIN\freeFunctorStandin}(\biWhitComodStandin)\] admits a lax monoidal right adjoint $\underline{\kostantWhittakerSTANDIN}^R$ which can be defined on objects by the formula \[\underline{\kostantWhittakerSTANDIN}^R(C, \{Z_{\kostantWhittakerSTANDIN\freeFunctorStandin(R)}\}) := (\kostantWhittakerSTANDIN^R(C), \{\tilde{Z}_{\freeFunctorStandin(R)}\})\] where $\tilde{Z}_{\freeFunctorStandin(R)}$ is the unique natural transformation such that for all $R \in \repGStandIn$, the following diagram commutes: \begin{equation}\label{Diagram Defining Central Structure On Right Adjoint}\xymatrix@R+2em@C+2em{\kostantWhittakerSTANDIN^R(C) \otimes \freeFunctorStandin(R) \ar[r]^{\tilde{Z}_{\freeFunctorStandin(R)}} \ar[d]^{\sim} & \freeFunctorStandin(R) \otimes \kostantWhittakerSTANDIN^R(C) \ar[d]^{\sim} \\
\kostantWhittakerSTANDIN^R(C \otimes \kostantWhittakerSTANDIN\freeFunctorStandin(R)) \ar[r]^{\kostantWhittakerSTANDIN^R(Z_{\kostantWhittakerSTANDIN\freeFunctorStandin(R)})} & \kostantWhittakerSTANDIN^R(\kostantWhittakerSTANDIN\freeFunctorStandin(R) \otimes C)
}
\end{equation} where the left vertical, resp. right vertical, arrow is given by the composite of maps in \eqref{One Side Projection Formula Map in Generality Changed Definition}, resp. \eqref{Other Side Projection Formula Map in Generality Changed Definition}, and is an isomorphism. The functor $\underline{\kostantWhittakerSTANDIN}^R$ is fully faithful if $\kostantWhittakerSTANDIN^R$ is fully faithful.
\end{Proposition}

\begin{proof}
The composites of the maps in \eqref{One Side Projection Formula Map in Generality Changed Definition} and \eqref{Other Side Projection Formula Map in Generality Changed Definition} are isomorphisms for dualizable objects by \cite[Corllary 3.19]{FlakeLaugwitzPosurProjectionFormulasandInducedFunctorsonCentersofMonoidalCategories}. Since by assumption $\kostantWhittakerSTANDIN^R$ commutes with filtered colimits and $\kostantWhittakerSTANDIN$ commutes with colimits as it is a left adjoint, the projection formula holds for filtered colimits of dualizable objects. Therefore, $\tilde{Z}_{F(R)}$ is defined.
    
We claim that, for any $(B, \{Z_{\freeFunctorStandin(R)}\}) \in \Z_{\freeFunctorStandin}(\HarishChandraBimoduleStandin)$, respectively $(C, \{Z'_{\kostantWhittakerSTANDIN\freeFunctorStandin(R)}\}) \in \Z_{\kostantWhittakerSTANDIN\freeFunctorStandin}(\biWhitComodStandin)$, the unit map $u(B)$, respectively the counit map $c(C)$, for the adjunction $(\kostantWhittakerSTANDIN, \kostantWhittakerSTANDIN^R)$ is a map in $\Z_{\freeFunctorStandin}(\HarishChandraBimoduleStandin)$, respectively $\Z_{\kostantWhittakerSTANDIN\freeFunctorStandin}(\biWhitComodStandin)$. We show the claim for the unit map; the counit map is similar. We claim that the following diagram commutes: \vspace{-0.1in}
\begin{equation*}
\xymatrix@R+2em@C+2em{ 
B \otimes \freeFunctorStandin(R) \ar[r]_{u(B \otimes \freeFunctorStandin(R))\textcolor{white}{iddon}} \ar[d]_{Z_{F(R)}} \ar@/^1.5pc/[rrr]^{u(B) \otimes \mathrm{id}}& 
\kostantWhittakerSTANDIN^R\kostantWhittakerSTANDIN(B \otimes \freeFunctorStandin(R)) \ar[d]_{\kostantWhittakerSTANDIN^R\kostantWhittakerSTANDIN(Z_{F(R)})} \ar[r]^{\sim} & 
\kostantWhittakerSTANDIN^R(\kostantWhittakerSTANDIN(B) \otimes \kostantWhittakerSTANDIN\freeFunctorStandin(R)) \ar[d]^{\kostantWhittakerSTANDIN^R(Z_{\kostantWhittakerSTANDIN\freeFunctorStandin(R)})} & 
\kostantWhittakerSTANDIN^R\kostantWhittakerSTANDIN(B) \otimes \freeFunctorStandin(R)\ar[l]^{\eqref{One Side Projection Formula Map in Generality Changed Definition}}\ar[d]^{\tilde{Z}_{F(R)}}\\
\freeFunctorStandin(R) \otimes B \ar[r]^{u(\freeFunctorStandin(R) \otimes B)\textcolor{white}{iddon}} \ar@/_2pc/[rrr]^{\mathrm{id} \otimes u(B)} & 
\kostantWhittakerSTANDIN^R\kostantWhittakerSTANDIN(\freeFunctorStandin(R) \otimes B) \ar[r]^{\sim} & 
\kostantWhittakerSTANDIN^R(\kostantWhittakerSTANDIN\freeFunctorStandin(R) \otimes \kostantWhittakerSTANDIN(B)) & 
\freeFunctorStandin(R) \otimes \kostantWhittakerSTANDIN^R\kostantWhittakerSTANDIN(B) \ar[l]_{\eqref{Other Side Projection Formula Map in Generality Changed Definition}} 
}
\end{equation*}
where the unlabeled maps are the monoidality isomorphisms. Indeed, the leftmost box commutes by naturality, the middle two boxes commute by the respective definitions in \eqref{Diagram Defining Central Structure After Post Composing} and \eqref{Diagram Defining Central Structure On Right Adjoint}, and the top and bottom squares commute because the unit and counit for the adjoint pair $(\kostantWhittakerSTANDIN, \kostantWhittakerSTANDIN^R)$ are monoidal natural transfomations. Reading the exterior of this diagram shows that $u(B)$ is a map in $\Z_{\freeFunctorStandin}(\HarishChandraBimoduleStandin)$ as desired.

Applying \cref{Adjoint to Monoidal Functor is Appropriately Lax Monoidal} to our functor $K$ yields a lax monoidal structure on $\underline{\kostantWhittakerSTANDIN}^R$. Finally, the fully faithfulness claim follows since a right adjoint is fully faithful if and only if the counit is an isomorphism. 
\end{proof}


\subsection{Proof of \cref{Functoriality of Drinfeld Center Under Strong Generation Assumption}}
  \label{With Fully Faithful Right Adjoint Subsection} 
As $\kostantWhittakerSTANDIN$ is monoidal,  there is a monoidal functor $\mathcal{P}_{\kostantWhittakerSTANDIN}: \Z(\biWhitComodStandin) \to \Z_{\kostantWhittakerSTANDIN}(\biWhitComodStandin)$ given on objects by the formula $(C, \{Z_{\freeFunctorStandin(R)}\}) \mapsto (C, \{Z_{\kostantWhittakerSTANDIN\freeFunctorStandin(R)}\})$.

\begin{Proposition}\label{Pullback By KostantWhittakerStandIn Gives Equivalence of Categories}
The functor $\mathcal{P}_{\kostantWhittakerSTANDIN}$ is a monoidal equivalence of categories. 
\end{Proposition}

\begin{proof}
This is given by standard arguments: we give a proof for the convenience of the reader. This functor is evidently faithful; we now argue it is full: let $f: C_1 \to C_2$ be a map in $\biWhitComodStandin$ inducing a map in $\Z_{\kostantWhittakerSTANDIN}(\biWhitComodStandin)$. We wish to show that \begin{equation}\label{Fully Faithful Diagram Central}
\xymatrix@R+0.01em@C+0.01em{
C_1 \otimes C \ar[d]^{f \otimes \text{id}} \ar[r]^{Z^1_C} &  C \otimes C_1 \ar[d]^{\text{id} \otimes f} \\
C_2 \otimes C \ar[r]^{Z^2_{C}} &  C \otimes C_2
}
\end{equation} commutes. However, observe that this diagram naturally fits as the front face of a cube whose \lq back face\rq{} commutes by assumption that $f$ induces a map in $\Z_{\kostantWhittakerSTANDIN}(\biWhitComodStandin)$ and whose four other non-front faces commute by naturality of $Z^1, Z^2$, and the counit of our adjunction. Therefore \eqref{Fully Faithful Diagram Central} commutes and so $\mathcal{P}_{\kostantWhittakerSTANDIN}$ is fully faithful.

Finally, we argue $\mathcal{P}_{\kostantWhittakerSTANDIN}$ is essentially surjective: if $(C_1, Z^1_{\kostantWhittakerSTANDIN(B)}) \in \Z_{\kostantWhittakerSTANDIN}(\biWhitComodStandin)$, we define an extension $Z^1_C$ as the unique map which makes the following diagram commute:
\begin{equation*}
  \xymatrix{
    C_1 \otimes \kostantWhittakerSTANDIN\kostantWhittakerSTANDIN^R(C) \  \ar[d]^<>(0.5){Z^1_{\kostantWhittakerSTANDIN(\kostantWhittakerSTANDIN^R(C))}} \ar[rr]^<>(0.5){\sim} &&\    C_1 \otimes C\   \ar[d]^<>(0.5){Z^1_C}\\
\kostantWhittakerSTANDIN\kostantWhittakerSTANDIN^R(C) \otimes C_1 \  \ar[rr]^<>(0.5){\sim} &&\   C \otimes C_1
}
\end{equation*}
where the horizontal arrows induced by the counit, which are isomorphisms as $\kostantWhittakerSTANDIN$ is fully faithful by assumption. Using this definition along with the fact that the unit and counit of this adjunction are monoidal natural transformations, one can indeed check that $(C, Z_C^1) \in \Z(\biWhitComodStandin)$. Now observe that if $C = \kostantWhittakerSTANDIN(B)$ the diagram 
\begin{equation*}
\xymatrix{
  C_1 \otimes C \  \ar[d]^<>(0.5){Z^1_{\kostantWhittakerSTANDIN(B)}} \ar[rr]^<>(0.5){\mathrm{id} \otimes u(B)\textcolor{white}{whe}} &&
\    C_1 \otimes \kostantWhittakerSTANDIN\kostantWhittakerSTANDIN^R(C) \  \ar[d]^<>(0.5){Z^1_{\kostantWhittakerSTANDIN(\kostantWhittakerSTANDIN^R(C))}} \ar[rr]^<>(0.5){\sim} && \   C_1 \otimes C \  \ar[d]^<>(0.5){Z^1_C}\\
  C \otimes C_1 \  \ar[rr]^<>(0.5){u(B) \otimes \mathrm{id}\textcolor{white}{whe}} &&\  
  \kostantWhittakerSTANDIN\kostantWhittakerSTANDIN^R(C) \otimes C_1\   \ar[rr]^<>(0.5){\sim} && \  C \otimes C_1
}
\end{equation*} where $u(B)$ is the unit map, by construction of $Z^1_C$ and naturality of $Z^1_{\kostantWhittakerSTANDIN(\td)}$. The composite of the horizontal arrows in both rows of the diagram give the identity functor by the unit counit criterion of the adjunction, and so we see that $\mathcal{P}_{\kostantWhittakerSTANDIN}(C, Z_C^1) = (C, Z_{\kostantWhittakerSTANDIN(B)}^1)$ as required.
\end{proof}

Using that $\mathcal{P}_{\kostantWhittakerSTANDIN}$ is a monoidal equivalence of categories by \cref{Pullback By KostantWhittakerStandIn Gives Equivalence of Categories}, we may use $\mathcal{P}_{\kostantWhittakerSTANDIN}$ to equip $\Z_{\kostantWhittakerSTANDIN}(\biWhitComodStandin)$ with a braiding. 

\begin{Corollary}\label{underline KostantWhitstandin Is Braided Monoidal}
The functor $\underline{\kostantWhittakerSTANDIN}: \Z(\HarishChandraBimoduleStandin) \to \Z_{\kostantWhittakerSTANDIN}(\biWhitComodStandin)$ defined in \cref{With a Monoidal Functor Theres (1)An Induced Functor on Relative Drinfeld Centers} is braided monoidal.
\end{Corollary}

\begin{proof}
    Let $(B_1, Z^1), (B_2, Z^2) \in \Z(\HarishChandraBimoduleStandin)$. Applying the substitutions $B := B_1, R := B_2$, and $F = \text{id}$ in the diagram \eqref{Diagram Defining Central Structure After Post Composing}, we see that the upper horizontal arrow is $\kostantWhittakerSTANDIN(Z^1_{B_2})$ and the lower horizontal arrow is the braiding morphism in $\Z_{\kostantWhittakerSTANDIN}(\biWhitComodStandin)$ by construction of the braiding. The fact that this diagram commutes thus proves that $\underline{\kostantWhittakerSTANDIN}$ is braided. 
\end{proof}

Finally, we prove the following result which informally states that, in the above notation, the categories $\Z(\HarishChandraBimoduleStandin)$ and $\Z_{\kostantWhittakerSTANDIN}(\biWhitComodStandin)$ are determined by the restriction of the central structure to free objects: 

\begin{Proposition}\label{Forgetful Functor to Relative Drinfeld Center of Projective Generating Set is Equivalence}
    The functors \[\Z(\HarishChandraBimoduleStandin) \xrightarrow{\mathcal{P}_{\freeFunctorStandin}} \Z_{\freeFunctorStandin}(\HarishChandraBimoduleStandin) \text{, and } \Z_{\kostantWhittakerSTANDIN}(\biWhitComodStandin) \xrightarrow{\mathcal{P}_{\freeFunctorStandin}} \Z_{\kostantWhittakerSTANDIN\freeFunctorStandin}(\biWhitComodStandin)\] are monoidal equivalences of abelian categories.
\end{Proposition}

We first recall the following standard categorical lemma:

\begin{Lemma}\label{In a Cocomplete Abelian Category with Compact Projective Generating Set Natural Transformations of Continuous Functors on Generating Subcategory Uniquely Lift to Natural Transformation}
Assume $\mathcal{A}$ is a cocomplete abelian category and $\mathcal{P}$ is a full subcategory of projective generators, and assume $F, G: \mathcal{A} \to \mathcal{B}$ are functors to some abelian category $\mathcal{B}$ which commute with colimits. Then any natural transformation of restricted functors $\eta: F|_{\mathcal{P}} \xRightarrow{} G|_{\mathcal{P}}$ extends uniquely to a natural transformation $\widetilde{\eta}: F \xRightarrow{} G$ which is a natural isomorphism if $\eta$ is a natural isomorphism.
\end{Lemma} 
\begin{proof}[Proof of \cref{Forgetful Functor to Relative Drinfeld Center of Projective Generating Set is Equivalence}] We show the latter functor is a monoidal equivalence, the former functor is an essentially identical argument. The monoidality morphisms are given by the identity maps. The functor $\mathcal{P}_{\freeFunctorStandin}$ is faithful since the forgetful functor to $\biWhitComodStandin$ is faithful; we now show it is full. Assume $(C_1, \{Z^1_{\kostantWhittakerSTANDIN(R)\}}), (C_2, \{Z^2_{\kostantWhittakerSTANDIN(R)}\}) \in \Z_{\kostantWhittakerSTANDIN}(\biWhitComodStandin)$ and $f: C^1 \to C^2$ is a map for which we have the equality \[(\text{id} \otimes f)\circ Z^1_{\kostantWhittakerSTANDIN\freeFunctorStandin(R)} = Z^2_{\kostantWhittakerSTANDIN\freeFunctorStandin(R)}\circ (f \otimes \text{id})\] for all $R \in \repGStandIn$. Then the two natural transformations $(\text{id} \otimes f)\circ Z^1_{\kostantWhittakerSTANDIN(R)}, Z^2_{\kostantWhittakerSTANDIN(R)}\circ (f \otimes \text{id})$ have the property that they agree upon restriction to a projective generating set, so they agree by the uniqueness statement of \cref{In a Cocomplete Abelian Category with Compact Projective Generating Set Natural Transformations of Continuous Functors on Generating Subcategory Uniquely Lift to Natural Transformation}. (Note that our assumptions that $\freeFunctorStandin, \kostantWhittakerSTANDIN$, and the monoidal structure on $\biWhitComodStandin$ in the first argument and the second argument all commute with colimits, and so these functors commute with colimits.) Similarly, to show that this functor is essentially surjective, observe that if we are given some $(C_1, \{Z^1_{\kostantWhittakerSTANDIN\freeFunctorStandin(R)}\}) \in \Z_{\kostantWhittakerSTANDIN\freeFunctorStandin}(\biWhitComodStandin)$ we may use the existence statement in \cref{In a Cocomplete Abelian Category with Compact Projective Generating Set Natural Transformations of Continuous Functors on Generating Subcategory Uniquely Lift to Natural Transformation} to extend $Z^1_{\kostantWhittakerSTANDIN\freeFunctorStandin(\td)}$ to a natural transformation $Z^1_{\kostantWhittakerSTANDIN(\td)}$ of functors $C_1 \otimes \kostantWhittakerSTANDIN(\td) \xRightarrow{} \kostantWhittakerSTANDIN(\td) \otimes C_1$.
\end{proof}

As above, we use the braidings on the domain categories in \cref{Forgetful Functor to Relative Drinfeld Center of Projective Generating Set is Equivalence} to equip the respective codomains with braidings for which the functors appearing in the statement of \cref{Forgetful Functor to Relative Drinfeld Center of Projective Generating Set is Equivalence} are braided monoidal.

\begin{proof}[Proof of \cref{Functoriality of Drinfeld Center Under Strong Generation Assumption}] We have shown in \cref{Forgetful Functor to Relative Drinfeld Center of Projective Generating Set is Equivalence} and \cref{Pullback By KostantWhittakerStandIn Gives Equivalence of Categories} that the horizontal arrows in the commutative diagram 
 \begin{equation*}
\xymatrix@R+0.01em@C+0.01em{
& \Z(\HarishChandraBimoduleStandin) \ar[r]^{\mathcal{P}_{\freeFunctorStandin}} \ar[d]^{\underline{\kostantWhittakerSTANDIN}} & \Z_{\freeFunctorStandin}(\HarishChandraBimoduleStandin) \ar[d]^{\underline{\kostantWhittakerSTANDIN}} \\
\Z(\biWhitComodStandin) \ar[r]^{\mathcal{P}_{\kostantWhittakerSTANDIN}} &  \Z_{\kostantWhittakerSTANDIN}(\biWhitComodStandin) \ar[r]^{\mathcal{P}_{\freeFunctorStandin}} & \Z_{\kostantWhittakerSTANDIN\freeFunctorStandin}(\biWhitComodStandin)
}
\end{equation*} give monoidal equivalences of categories which, by construction, are braided. Moreover, the left vertical arrow is braided by \cref{underline KostantWhitstandin Is Braided Monoidal}. Therefore by \cref{Relative Version of Existence of Right Adjoint for Drinfeld Centers Via Projection Formula} we obtain our desired fully faithful right adjoint; the fact that it is braided lax monoidal immediately follows from \cref{Adjoint to Monoidal Functor is Appropriately Lax Monoidal}.
\end{proof}

\newcommand{\upsFunctorFromGradedASYMPTOTICBiWhitModToCenterOfAsymptoticHCG}{\mathcal{N}g\text{\^o}}

\newcommand{\upsFunctorFromCLASSICALBiWhitModGRADEDToCenterOfClassicalGRADEDHarishChandra}{\ups^0}
\newcommand{\upsFunctorFromCLASSICALBiWhitModToCenterOfClassicalHarishChandra}{\overline{\ups}^0}

\section{Knop-Ng\^o functor and parabolic restriction}
\label{Essential Image of Ngo Functor}
\label{Comparison to Parabolic Induction Section}

In this section, we first prove \cref{Theorem Asserting Existence of Ngo Functor as Braided Monoidal Fully Faithful Functor to Center}(ii) and then use it to prove \cref{Braverman-Kazhdan Exactness Conjecture}. After that we complete the proof of \cref{Main Theorem 2}.
Recall that we set $\hb=1$ in \cref{Kostant-Whittaker reduction and the Miura bimodule section};  we continue to work in this setting in \cref{Comparison to Parabolic Induction Section}.


\subsection{Proof of \cref{Theorem Asserting Existence of Ngo Functor as Braided Monoidal Fully Faithful Functor to Center}(ii)}\label{Proof of First Theorem Part Two}
The functor $\Res^G_T=\int_q p^*: \differentialOperatorsOnG\modwithdash^{{\Ad G}}\to\differentialOperatorsOnT \rtimes W\modwithdash$
of parabolic restriction  is a right adjoint of $\Ind_T^G$, see \labelcref{spr diag}.
Therefore, it follows from the isomorphism \labelcref{Tensoring with quantizationOfCentralizers Is Parabolic Induction No W Invariants} and standard tensor-$\Hom$ adjunctions, that one has isomorphisms of functors, cf.
\cite[Theorem 4.4]{GinzburgParabolicInductionandtheHarishChandraDModule},
  \begin{equation}\label{res-iso}
    \Res^G_T(\td)^W \cong \Hom_{D(G)\mmod^G}(\mathbf{N},\td) =
    \Hom_{D(G\mmod^G)}(D(G)/D(G)\ad\g,\td)\cong (\td)^G.
  \end{equation} 
   Thus, for any $M \in \differentialOperatorsOnG\modwithdash^{{\Ad G}}$,
  the counit of the adjunction gives a morphism  of $D(G)$-modules
  $\alpha: \bN \otimes_{D(T)^W} M^G\cong(\Ind^G_T(M^G))^W \to M$.
  Explicitly, for $u\in D(G)$ and $m\in M^G$ the map $\al$ 
  sends
  $(u\text{mod}D(G)\ad\g)\o m\in D(G)/D(G)\ad\g\otimes_{D(T)^W} M^G$
  to $um$.

\begin{Lemma}\label{secondad}
  The map $\alpha$ is injective.
\end{Lemma}
\begin{proof} If two functors are naturally isomorphic, then a standard Yoneda lemma argument proves that their adjoints are naturally isomorphic in such a way that the units and counits naturally intertwine. In particular, in view of the natural isomorphism \labelcref{res-iso}, it suffices to show that the counit \[c(M):  \Ind_T^G(\Res^G_T(M))^W \to M\] is injective for any $M \in D(G)\mmod^G$. Since the functor $\Ind_T^G(-)^W$ is fully faithful, the unit of the adjunction is a natural isomorphism. By definition of adjunction, $\Res_T^G(c(M)) \circ u(\Ind_T^G(M)^W) = \mathrm{id}_M$, so $\Res_T^G(c(M))$ is an isomorphism. Let $K$ denote the kernel of $c(M)$. Since $\Res_T^G$ is exact and $\Res_T^G(c(M))$ is an isomorphism, $\Res_T^G(K) = 0$. However, by the second adjointness theorem (see for example \cite[\S 0.2.1]{DrinfeldGaitsgoryOnATheoremofBraden} or \cite[Theorem 1.2]{BezrukavnikovYomDinOnParabolicRestrictionofPerverseSheaves}) we see that \[\Hom_{D(G)\mmod^G}(K, \Ind_T^G(\Res^G_T(M)) = \Hom_{D(T)\mmod}(\Res^G_T(K), \Res^G_T(M)) = \Hom_{D(T)\mmod}(0, \Res^G_T(M))\] must be zero, so the inclusion map $K \subseteq \Ind_T^G(\Res^G_T(M))^W$ must be the zero map and thus $K = 0$.
  \end{proof}

  \begin{proof}[Proof of \cref{Theorem Asserting Existence of Ngo Functor as Braided Monoidal Fully Faithful Functor to Center}(ii)]
    Using the isomorphisms
\[U\g\o_{Z\g}(\td)\,\cong\,(U\g\o_{Z\g}\J)\o_{\J}(\td),\enspace\text{resp.}
\enspace
\mathbf{N}\o_{D(T)^W}(\td)\,\cong\, (\mathbf{N}\o_{D(T)^W}\J)\o_{\J}(\td),
\]
we see that proving \cref{Theorem Asserting Existence of Ngo Functor as Braided Monoidal Fully Faithful Functor to Center}(ii) reduces to proving that the map
\begin{equation}\label{Counit for Parabolic Induction as Tensoring as Left Adjoint}
  \quantizationOfCentralizersHBAREQUALSONE\o_{D(T)^W} \J \ \xrightarrow{} \
    {U\LG}\otimes_{Z\LG} {\mathfrak J},\enspace u1_{\quantizationOfCentralizersHBAREQUALSONE}\o j\,
    \mapsto\,u(1\otimes j),
    \end{equation}
    is an isomorphism of $(D(G),\J)$-bimodules.
    The surjectivity of the map is clear since
    the element $1 \otimes 1$  generates $U\LG\otimes_{Z\LG} \biWhittakerDifferentialOperatorsonG$
    as a $(U\g\o \J_\op)$-module. To prove injectivity, we observe that 
    \[
      (\mathbf{N}\o_{D(T)^W} \J)^G\,=\,1\o \J\,=\,
      ({U\LG}\otimes_{Z\LG} {\mathfrak J})^G.
    \]
    Therefore, the restriction of the
    map \labelcref{Counit for Parabolic Induction as Tensoring as Left Adjoint} to $G$-invariants
    is the identity map. Hence, this map is an isomorphism and
    the desired injectivity now follows from \cref{secondad}.
    \end{proof}

    \subsection{Braverman-Kazhdan sheaves}
        \label{Proof of BK Exactness Conjecture Subsection} We now prove \cref{Braverman-Kazhdan Exactness Conjecture}. From the explicit formula \labelcref{bfo} for convolution of $D(G)$-modules, we see that the functor $\PhiTHISGROUPRho{G}\star(\td)$ is exact if and only if $\PhiTHISGROUPRho{G}$ is flat as a right $U\LG$-module. For a $G^{\vee}$-representation $V$, we may restrict $V$ to the torus $T^\vee$: Let $\PhiTHISGROUPRho{T}$ denote the corresponding $W$-equivariant
                 $D(T)$-module and write $\PhiTHISGROUPRho{T}^W:=(\PhiTHISGROUPRho{T})^W$.
                 By \cite[Theorem 1.4, Section 1.6]{ChenAVanishingConjecturetheGLnCase} and the non-holonomic variant of \cite[Theorem 1.5.1]{Gin} (which holds with the same proof)
                 that, using the notation of
                 \cref{Main Theorem 2}, we have
$\PhiTHISGROUPRho{T}^W\in D(T)\mmod^W_\circ$. Thus, it follows from the theorem that the
                 $D(T)^W$-action on  $\PhiTHISGROUPRho{T}^W$ 
                 has a canonical    extension to a $\biWhittakerDifferentialOperatorsonG$-action
                 and the $\Symt$-action map \begin{equation}\label{Descent Isomorphism for RhoBesselSheaves}\Symt \otimes_{\Symt^W} \PhiTHISGROUPRho{T}^W \xrightarrow{\sim} \PhiTHISGROUPRho{T}\end{equation}  is an isomorphism.
                 Now, by definition, one has
                 $\PhiTHISGROUPRho{G}:=\operatorname{Ind}_T^G(\PhiTHISGROUPRho{T})^W$.
                                  We deduce that $\PhiTHISGROUPRho{G}\cong \quantizationOfCentralizersHBAREQUALSONE \otimes_{\differentialOperatorsOnT^W} \PhiTHISGROUPRho{T}^W$. Hence,         for any $\differentialOperatorsOnG$-module $M$ we obtain isomorphisms \begin{align*}
                   \PhiTHISGROUPRho{G} \star M\cong\PhiTHISGROUPRho{G} \otimes_{U\LG} M
\cong (\quantizationOfCentralizersHBAREQUALSONE \otimes_{\differentialOperatorsOnT^W} 
\PhiTHISGROUPRho{T}^W) \otimes_{U\LG} M
 \cong \ups(\PhiTHISGROUPRho{T}^W) \otimes_{U\LG} M \cong \PhiTHISGROUPRho{T}^W \otimes_{Z\LG} M
                                  \end{align*} where the first isomorphism holds by \labelcref{bfo}, the second follows
                                  from  \cref{Theorem Asserting Existence of Ngo Functor as Braided Monoidal Fully Faithful Functor to Center}(ii), and the final follows from the definition of the functor $\ups$.
                                  We conclude that it suffices to prove that $\PhiTHISGROUPRho{T}^W$ is flat as a $Z\LG$-module. Since
                                  $\Symt$ is a faithfully flat $Z\LG$-module,  this is equivalent
                                   to proving that the $\Symt$-module $\Symt \otimes_{\Symt^W} \PhiTHISGROUPRho{T}^W$ is flat. 
                                   We see                from isomorphism
                                   \eqref{Descent Isomorphism for RhoBesselSheaves}
                                   that we are reduced to proving that the functor
                                   that $\PhiTHISGROUPRho{T} \star (\td)$ of convolution
                                   on $T$ is  exact, i.e. the statement of the theorem in the case of the torus.
                                   The representation $\rho|_{T^{\vee}}$ is a direct
                                   sum of 1-dimensional representations, some
                                   characters $\lambda_i,\,1 \leq i \leq n$ of $T^{\vee}$.
                                   Our assumption that $V^{T^\vee}=\{0\}$ implies that $\lambda_i \neq 0$ for all $i$. Following Braverman and Kazhdan, we
                                   view the $\lambda_i$'s as
                                   cocharacters $\mathbb{G}_m \to T$
                                   and let $\mathcal{E}=D(\mathbb{G}_m)e^x$ be the restriction of the exponential $D$-module on $\mathbb{A}^1$ to the open subset $\mathbb{G}_m$.
                                   Mimicking the arguments in \cite[(4.6)]{BravermanKazhdanGammaSheavesonReductiveGroups} in the $D$-module setting, we obtain
                                   an isomorphism \[\PhiTHISGROUPRho{T} \cong
                                     \Big(\int_{\lambda_{1}}\mathcal{E}\Big) \star \ldots
                                     \star \Big(\int_{\lambda_{n}}\mathcal{E}\Big) .\]
                                                                      Therefore it suffices to prove that
                                   for any nonzero character $\lambda$
                                   the functor of convolution with $\int_{\lambda}\mathcal{E}$ is exact.
    
                                   Let               $\partial=t\frac{d}{dt}$ be
                                   the $\mathbb{G}_m$-invariant differential operator on
                                   $\mathbb{G}_m = \text{Spec}(k[t^{\pm 1}])$ so that
                                                                      $D(\mathbb{G}_m)^{\mathbb{G}_m}=k[\partial]$.
                                                                      It is immediate to check that $\mathcal{E}$ is  torsion free as a $k[\partial]$-module.

                                                                      Thus, to complete the proof it suffices to prove that 
                                   for any nonzero character $\lambda$
                                   and                                    
                                   any $D(\mathbb{G}_m)$-module $\F$ which is torsion free as a
                                   $k[\partial] $-module the
functor  of convolution with $\int_{\lambda}\mathcal{F}$ is exact.
To prove this,
we factor $\lambda$ as a composition $\mathbb{G}_m\xrightarrow{x\mto x^n}
\mathbb{G}_m\into T$, where the the second homomorphism
is a closed embedding and  $n\neq 0$ since $\lambda$ is nonzero.
One can directly compute that $D$-module pushforward
along the map $x\mto x^n$ preserves the property of being $k[\partial]$-torsion free.
So we are reduced to the case where $\lambda$ is a closed embedding.
In that case, the map ${\mathbb Z}={\mathbb X}_*({\mathbb G}_m)\into {\mathbb X}_*(T^\vee)$
  is a split embedding, hence  the short exact sequence
  $1\to{\mathbb G}_m\to T^\vee\to T^\vee/{\mathbb G}_m\to 1$ splits. Thus, $\lambda$ imbeds ${\mathbb G}_m$ as a direct factor of $T^\vee$,
  so $\Symt \cong k[\partial] \otimes \Sym(\mathfrak{t}/\text{Lie}(\mathbb{G}_m))$
  Then, it follows from the definition of the pushforward
  $\int_{\lambda}$ that $\int_{\lambda}\mathcal{F}$ is
  flat over $\Sym(\mathfrak{t})$ whenever $\mathcal{F}$ is
  flat over $k[\partial]$.
  \qed

\begin{proof}[Sketch of proof of Proposition \ref{transfer cor} ] Part (i)
  follows from Theorem \ref{Main Theorem 2} using the isomorphism
  $U\h\otimes_{Z\h}(Z\h\otimes_{Z\g}\frj_G)\o_{\frj_G}(\td)\cong U\h\otimes_{Z\g}(\td)$.
  The second statement in (ii) holds by construction and the first statement
  follows from the monoidality of
  the functor  $Z\LH \otimes_{Z\LG} (-)$ and
  the fact that the equivalences of \cref{Main Theorem 2} are monoidal.
  
To prove (iii) observe that the assumption  ${V}^{T_H^{\vee}}=0$
ensures that $\boldsymbol{\Phi}_{{V}}$, a priori a complex of $\D$-modules, lies in the heart of the usual $t$-structure: see \cite[Proposition 6.4]{BravermanKazhdanGammaSheavesonReductiveGroups} or the proof of \cref{Braverman-Kazhdan Exactness Conjecture} above. Moreover, as we have recalled above in \cref{BK rem}(iii), one can show that $\mathrm{tr}^*(\psi) = \Upsilon(\mathrm{tr}|_{\mathrm{Diag}}^*(\psi))$ where $\mathrm{Diag} \subseteq \GL_n$ is the subset of diagonal matrices (see eg. \cite[Theorem 5.1]{BravermanKazhdanGammaSheavesonReductiveGroups}) and that $\boldsymbol{\Phi}_{{V}} := \Upsilon_G(\boldsymbol{\Phi}_{{V}|_{T^{\vee}}})$. Therefore, this claim follows from the fact that, if $G$ and $H$ are tori, then ${\mathscr T}_f$ is the pushforward by $f^{\vee}$, which can be checked explicitly. 
  \end{proof}

  \subsection{Vanishing for central $D$-modules induced from their invariants} We now prove a key vanishing result in the general parabolic setting. To this end, let $P$ be a parabolic subgroup which contains $\BorelSubgroup$. Let $U_P$ denote the unipotent radical of $P$, and let $\pi^P: G \to G/U_P$ denote the quotient map. Let $L$ denote the group $P/U_P$. Let $\mathfrak{l}$ denote its Lie algebra and $\mathfrak{u}_L$ denote the Lie algebra of the the unipotent radical $U_L$ of $L$.
  Let $\ell: P/U_P \xhookrightarrow{} G/U_P$ denote the closed embedding.
  We let  $P$ act on $G$ by conjugation $p\cdot g = pgp^{-1}$.
  This induces a $P$-action on $G/U_P$ for which $\pi^P$ is $P$-equivariant. In what follows we consider $G/U_P$ as a $P$-variety for this \lq adjoint\rq{} action of $P$.

  Let $W_L := N_L(T)/T$ denote the Weyl group for $L$, and let $\ups_L$ denote the Knop-Ng\^o functor for the group $L$. 
\begin{Theorem}\label{Vanishing Conjecture For Parabolics for Central D Modules Induced from Invariants}
If $M$ is a $\differentialOperatorsOnG$-module which is induced from its invariants, then:

\vi The canonical morphism $c: \int_{\ell}\ell^{\dagger}(\int_{\pi^P}M) \xrightarrow{} \int_{\pi^P}M$, a priori a morphism in the derived category of $\D_{G/U_P}$-modules, is an isomorphism of sheaves concentrated in degree zero. In particular, $\int_{\pi^P}M$ is supported on the closed subscheme $P/U_P$. 

\vii We have isomorphisms of left $D(P/U_P)$-modules:
\[\operatorname{Res}_L^G(M) = \ell^{\dagger}\Big(\int_{\pi^P}M\Big) \cong   U\mathfrak{l} \otimes_{Z\LG} M^G = \ups_L(M^G)\] 

\viii For any $\differentialOperatorsOnG$-module $E$ and $F\in D(T)\mmod^W_\circ$, viewed as a
$W$-equivariant $\differentialOperatorsOnT$-module,
there is an isomorphism \[\int_{\pi^P}(E \star_G \IndTG(F)^W) \cong (\int_{\pi^P}E) \star_L \operatorname{Ind}_{T}^L(F)^{W_L}.\]
\end{Theorem}
To prove \cref{Vanishing Conjecture For Parabolics for Central D Modules Induced from Invariants} we will use the following

\begin{Proposition}\label{Parabolic Universal Verma is Flat Zg Module Proposition}
    The $Z\LG$-module $\Delta_P := U\LG/U\LG\cdot \mathfrak{u}_P$ is flat.
\end{Proposition}
  
\begin{proof}   
  By the PBW theorem and the Harish-Chandra isomorphism
  the statement reduces to the claim that
  $\Sym\g/(\Sym\g) \mathfrak{u}_P$ is flat over $(\Sym\g)^G$.
 In the special case $P=G$ and $\mathfrak{u}_P=0$ this
is a classic result of Kostant. To prove the general case
one can mimic the arguments
used in the proof of \cite[Theorem 6.7.4(i)]{ChrissGinzburgRepresentationTheoryandComplexGeometry}.
  \end{proof}


  Let $Z_P$ denote the inverse image of the center $Z(L)$ under the map $\pi^P$.
  
\begin{Lemma}\label{QCoh on Basic Affine With UP Invariant Sections Having Trivial L Representation Are Supported on Levi}
  Let $\F $ be a $Z_P^{\circ}$-equivariant  quasi-coherent sheaf on $G/U_P$  such that the induced $Z(L)^{\circ}$-action on the vector space $\Gamma(\F)^{U_P}=\Gamma(G/U_P, \F)^{U_P}$ is trivial.
   Then any element $m \in \Gamma(\F)^{U_P}$ is supported on $P/U_P$.
\end{Lemma}

\begin{proof} It is known that there exists a finite dimensional representation $V\in \Rep(G)$ and a vector $v\in V$
  fixed by $U_P$ such that the map $g\mto g.v$ induces a 
  $G$-equivariant locally closed embedding $\eta: G/U_P \to V$ such that, if we write $V \cong \oplus_i V_{\lambda_i}$ as a direct sum decomposition
  into irreducible $G$-representations $V_{\lambda_i}$ with highest weight $\lambda_i \in\mathbb{X}^*(T)$ with respect to $\BorelSubgroup$ and write $v = \sum_{i} v_{\lambda_i}$ where $v_{\lambda_i}\in V_{\lambda_i}$, then each $v_{\lambda_i}$ is in the irreducible $P$-subrepresentation $S_1^i$ of $V_{\lambda}^i$ containing the highest weight line of $V_{\lambda}^i$ with respect to $\BorelSubgroup$.

Fix some irreducible summand $V_{\lambda_i}$ of $V$. We now argue that the support of $\Gamma(\F)^{U_P}$, as a $\Sym(V_{\lambda_i}^*)$-module, is contained in the $P$-subrepresentation $S_1^i$. We first observe that, by the Lie-Kolchin theorem, we may write $V_{\lambda_i}$ as an increasing union of $\BorelSubgroup$-representations \[\{0\} \subsetneq S_1^i \subsetneq S_2^i \subsetneq ... \subsetneq S^i_{n_i} = V_{\lambda_i}\] such that for all $j \geq 2$, $S^i_j/S^i_{j - 1}$ is a one dimensional $\BorelSubgroup$-representation. Choose a basis for $S_1^i$, and complete it via some vectors $\{u_2, ..., u_{n_i}\}$ to a basis which is compatible with this filtration, and let $u_j^*$ be the corresponding elements of the dual basis. To prove that the support of $\Gamma(\F)^{U_P}$ is contained in $S_1^i$, it obviously suffices to recursively prove the following claim for $j \in \{n_i, n_i-1, ..., 2\}$:
\begin{quote}\label{X}
    The vector $u_j^*$ acts on the $\Sym(S_j^*)$-module $\Gamma(\mathcal{F})^{U_P}$ by zero, and so $\Gamma(\mathcal{F})^{U_P}$ is naturally a $\Sym(S_{j-1}^*)$-module, and, moreover, $V_{\lambda_i}^{Z_P^{\circ}} \subseteq S_{j-1}^i$, where we take invariants for the \lq adjoint\rq{} $Z_P^{\circ}$-action.
\end{quote}

To prove this claim, first observe that, since $S_j^i/S_j^{i-1}$ is $\BorelSubgroup$-invariant, $u_j^* \in \Sym(S_j^{i, *})^{\unipotentRadicalOfBorel}$, so that multiplication by $u_j^*$ on $\Gamma(\F)$ preserves $\Gamma(\F)^{U_P}$, and $u_j^*$ is homogeneous with respect to the action of $T$. Since $u_j \in V_{\lambda_i}$, $T$ acts on the line spanned by $u_j^*$ by some character of the form $\lambda_i - \sum_r p_r\alpha_r - \sum_j q_j\beta_j$ where $\{\beta_j\}$ is the complement of $\{\alpha_r\}$ in the set of simple roots and $p_r, q_j$ are some nonnegative integers. Moreover, since $u_j$ is homogeneous for the $T$-action and is not in the $P$-subrepresentation $S_1^i$, $\sum_j q_j\beta_j \neq 0$.  Thus, since $Z(L)^{\circ} = (\cap_r \text{ker}(\alpha_r))^{\circ}$, $u_j^*$ is homogeneous with respect to the left $Z(L)^{\circ}$-action on $S_j^i$ for the character $\lambda_i - \sum_jq_j\beta_j$ of $Z(L)^{\circ}$, and this character is not $\lambda_i$. Therefore, $u_j^*$ is homogeneous with respect to the \lq adjoint\rq{} $Z(L)^{\circ}$-action on $S_j^i$ of degree $-\sum_jq_j\beta_j \neq 0$. Therefore this function must vanish on the fixed point subscheme $V_{\lambda_i}^{Z_P^{\circ}}$ and so $V_{\lambda_i}^{Z_P^{\circ}}$ lies in the closed subscheme $S_{j-1}^i$ of $S_j^i$ cut out by the vanishing of $u_j^*$. Moreover, the fact that $u_j^*$ is homogeneous of nonzero degree implies that if $m \in \Gamma(\F)^{U_P} = \Gamma(\F)^{Z_P^{\circ}}$, the ($\leq 1$ dimensional) vector subspace spanned by $u_j^*m \in \Gamma(\F)^{U_P}$ is homogeneous for the $Z(L)^{\circ}$ with respect to a nontrivial character of $Z(L)^{\circ}$. Thus, since the action of $Z(L)^{\circ}$ is trivial on $\Gamma(\F)^{U_P}$ by assumption, $u_j^*m = 0$, as desired, which proves our claim.

Our claim implies that the support of any $m \in \Gamma(\F)^{U_P}$, viewing $\F$ as a quasicoherent sheaf on $V$, has support contained in $\oplus_i (V_{\lambda_i} )^{Z_P^{\circ}} \cong V^{Z_P^{\circ}}$. Thus the support of any $U_P$-invariant global section of $\F$ is contained in $(G/U_P)^{Z_P^{\circ}}$. By the parabolic Bruhat decomposition, the points of $(G/U_P)^{Z_P^{\circ}}$ are precisely the points of $P/U_P$, and so \cref{QCoh on Basic Affine With UP Invariant Sections Having Trivial L Representation Are Supported on Levi} is proved. \end{proof}

\begin{proof}[Proof of \cref{Vanishing Conjecture For Parabolics for Central D Modules Induced from Invariants}]
The functor $R\Gamma$ is conservative since, as we have already noted above, $G/U_P$ is quasi-affine.
Therefore, it suffices to prove that $R\Gamma(c)$ is an isomorphism. Applying the same argument as in \cite[Example 6.1]{ChenOnTheConjecturesofBravermanKazhdan} gives an isomorphism \[R\Gamma\int_{\pi^P}M \cong M \mathop{\otimes}\limits^{L}_{U\LG} \Delta_P\]
of $D(G/U_P)$-modules for any $M \in \D(G)\mmod^G$ (which, as discussed in \cref{Monoidality and Tensor Product Subsection}, we view $M$ as a $U\LG$-bimodule via the ring map \labelcref{Comoment Map for G times Gop with Identifications}) where the $\mathcal{O}(G/U_P)$-action on the $\mathfrak{u}_P$-coinvariants $M \otimes_{U\LG} \Delta_P$ comes from the inclusion of the ring $\mathcal{O}(G/U_P) \cong \mathcal{O}(G)^{U_P}$ into 
  $D(G/U_P) = (\differentialOperatorsOnG/\mathfrak{u}_P\differentialOperatorsOnG)^{U_P}$.
Therefore we see that, since $M$ is induced from its invariants and $U\LG$ is flat as a $Z\LG$-module, we have isomorphisms \begin{equation}\label{For Very Centrals Pushforward is Universal Verma Tensor over Zg With G Invariants}R\Gamma\int_{\pi^P}M \cong M \mathop{\otimes}\limits^{L}_{U\LG} \Delta_P \cong (U\LG \mathop{\otimes}\limits_{Z\LG}  M^G) \mathop{\otimes}\limits^{L}_{U\LG} \Delta_P \cong (U\LG \mathop{\otimes}\limits_{Z\LG}^L  M^G)
  \mathop{\otimes}\limits^{L}_{U\LG} \Delta_P \cong  M^G \mathop{\otimes}\limits^{L}_{Z\LG} \Delta_P\end{equation}
and so $R\Gamma\big(\int_{\pi^P} M\big)$
is concentrated in degree zero by \cref{Parabolic Universal Verma is Flat Zg Module Proposition}. A base change argument identifies $\ell^{\dagger}\int_{\pi^P}$ with the parabolic restriction functor (see \cite[Remark 4.1]{BezrukavnikovYomDinOnParabolicRestrictionofPerverseSheaves}) which is $t$-exact by \cite[Theorem 5.4]{BezrukavnikovYomDinOnParabolicRestrictionofPerverseSheaves}. Since $D$-module pushforward by a closed embedding is $t$-exact, we see that 
the morphism $c$ in part (i) of  \cref{Vanishing Conjecture For Parabolics for Central D Modules Induced from Invariants}
is a map of objects concentrated in degree zero.

Since any section of $\F$ supported on $P/U_P$ is contained in the image of the map $c$, cf.
  \cite[Proposition 1.7.1(iii)]{HTT}, it suffices to prove that any section of $\int_{\pi^P}M$ is supported on $P/U_P$. Since $G/U_P$ is quasi-affine, hence $\F$  is globally generated \cite[Proposition 5.1.2]{EGAII},  it suffices to show the support of any element of $\Gamma(\int_{\pi^P}M)$ is contained in $P/U_P$.


  It of course further suffices to prove that the global sections
 $(2\rho_G - 2\rho_L)\o \Gamma\big(\int_{\pi^P}M\big)$
  is supported on $P/U_P$. However, by the isomorphism \eqref{For Very Centrals Pushforward is Universal Verma Tensor over Zg With G Invariants}, we see that all global sections of $\Gamma(M)$ are generated as a $U\LG$-module by the $U_P$-invariant sections, and moreover the $Z(L)$-representation on
$((2\rho_G - 2\rho_L)\o \Gamma\big(\int_{\pi^P}M\big))^{U_P}$
  is trivial.
Therefore by \cref{QCoh on Basic Affine With UP Invariant Sections Having Trivial L Representation Are Supported on Levi} all sections are supported on $P/U_P$, which establishes (i). 

Observe that, taking $U_P$-invariants of \eqref{For Very Centrals Pushforward is Universal Verma Tensor over Zg With G Invariants}, we obtain (ii). We now show (iii). Let $F$ be a $W$-equivariant $\differentialOperatorsOnT$-module such that $\Symt \otimes_{\Symt^W} F^W \xrightarrow{} F$ is an isomorphism. Then by \cref{Theorem Asserting Existence of Ngo Functor as Braided Monoidal Fully Faithful Functor to Center}(ii) we have $\ups(F^W) = \IndTG(F)^W$. Observe that \[\int_{\pi^P}(E \star_G \IndTG(F)^W) = \int_{\pi^P}(E \star_G \ups(F^W)) \simeq E \star_G \ups(F^W) \star_G \delta_{1U_P} \simeq E \star_G \int_{\pi^P}\ups(F^W)\] which by \cref{Vanishing Conjecture For Parabolics for Central D Modules Induced from Invariants}(i) identifies with \[E \star_G \int_{\ell}(\ell^{\dagger}\int_{\pi^P}\ups(F^W)) \simeq E \star_G \int_{\ell}(\delta_{1U_P}) \star_L \ell^{\dagger}\int_{\pi^P}\ups(F^W) \simeq (\int_{\pi^P}E) \star_L \ell^{\dagger}\int_{\pi^P}\ups(F^W)\] using the counit map of the adjunction and the $L_{\text{op}}$-equivariance of $\ell$. By \cref{Vanishing Conjecture For Parabolics for Central D Modules Induced from Invariants}(ii), we may continue this chain of isomorphisms to obtain \[\simeq (\int_{\pi^P}E) \otimes_{U\mathfrak{l}} U\mathfrak{l} \otimes_{Z\LG} F^W \cong (\int_{\pi^P}(E)\otimes_{Z\mathfrak{l}} (Z\mathfrak{l} \otimes_{Z\LG} F^W)  \cong (\int_{\pi^P}E)\otimes_{Z\mathfrak{l}} F^{W_L} \cong (\int_{\pi^P}E)\otimes_{U\mathfrak{l}} U\mathfrak{l} \otimes_{Z\mathfrak{l}} F^{W_L}\] 
\[(\int_{\pi^P}E)\otimes_{U\mathfrak{l}} \ups_L(F^{W_L}) = (\int_{\pi^P}E)\otimes_{U\mathfrak{l}} \operatorname{Ind}_T^L(F)^{W_L} = (\int_{\pi^P}E) \star_L \operatorname{Ind}_T^L(F)^{W_L}\] again using \cref{Theorem Asserting Existence of Ngo Functor as Braided Monoidal Fully Faithful Functor to Center}(ii).
\end{proof}

\subsection{Proof of \cref{Main Theorem 2}}\label{Proof of Main Theorem 2 Subsection}
We have constructed the Knop-Ng\^o functor, as a braided monoidal functor, and verified that it is fully faithful in \cref{Definition of Ngo Functor Subsection}. It remains to determine its essential image. Observe that, for any $E \in \biWhittakerDifferentialOperatorsonG\mmod$, the underlying Harish-Chandra bimodule of $\ups(E)$ is $U\LG \otimes_{Z\LG} E$, which is clearly induced from its invariants. Conversely, assume we are given a $\differentialOperatorsOnG$-module $M$ induced from its invariants. By \cref{Vanishing Conjecture For Parabolics for Central D Modules Induced from Invariants}(ii), the canonical map $\Symt \otimes_{Z\LG} M^G
\xrightarrow{} \operatorname{Res}_T^G(M)$ is an isomorphism. Since there are isomorphisms
\[\operatorname{Res}_T^G(M) \xleftarrow{\sim} \differentialOperatorsOnT \otimes_{\differentialOperatorsOnT^W} \operatorname{Res}_T^G(M)^W \cong \differentialOperatorsOnT \otimes_{\differentialOperatorsOnT^W} M^G\] by applying \cite[Proposition 4.1]{GinzburgParabolicInductionandtheHarishChandraDModule} and \cite[Theorem 4.4(i)]{GinzburgParabolicInductionandtheHarishChandraDModule}, respectively, we may trace through the constructions of these isomorphisms and obtain that the canonical map $\Symt \otimes_{Z\LG} M^G \xrightarrow{} \differentialOperatorsOnT \otimes_{\differentialOperatorsOnT^W}M^G$ is an isomorphism. Therefore, by \cite[Theorem 1.5.1]{Gin} (which also holds, with the same proof, after removing the condition of being holonomic) the $\differentialOperatorsOnT^W$-module structure on $\operatorname{Res}_T^G(M)^W = M^G$ upgrades to a module for $\biWhittakerDifferentialOperatorsonG$. We therefore obtain a map \[\ups(M^G) = \quantizationOfCentralizersHBAREQUALSONE \otimes_{\differentialOperatorsOnT^W} M^G \xrightarrow{} M\]
of central
Harish-Chandra bimodules (objects of $\D(G)\mmod^G$)
given by the module structure, i.e. given by the formula $\partial, m \mapsto \partial \cdot m$. Since it is clear this map is an isomorphism, we obtain an isomorphism $\ups(M^G) \cong M$ and so $M$ is in the essential image of the Knop-Ngo functor. 

It remains to show that a central
Harish-Chandra bimodule is induced from its invariants if and only if it is very central. By \cref{Vanishing Conjecture For Parabolics for Central D Modules Induced from Invariants}(i), any $G$-equivariant $\differentialOperatorsOnG$-module induced from its invariants is very central. Conversely, assume $M \in \differentialOperatorsOnG\modwithdash^{{\Ad G}}$ is very central. First, we claim that the $\differentialOperatorsOnT^W$-module structure on $M^G$ upgrades to a $\biWhittakerDifferentialOperatorsonG$-module structure. To see this, observe that, since $M$ is very central, the canonical map
\[\int_i\Res^G_T(M) = \int_ii^{\dagger}\int_{\pi^{\BorelSubgroup}}M \xrightarrow{} \int_{\pi^{\BorelSubgroup}}M\] is an isomorphism where $i$ is the inclusion of $\BorelSubgroup/\unipotentRadicalOfBorel$ into $G/\unipotentRadicalOfBorel$. In particular, using the exactness of parabolic restriction \cite{BezrukavnikovYomDinOnParabolicRestrictionofPerverseSheaves}, we see that $\int_{\pi^{\BorelSubgroup}}M$ is a complex concentrated in degree zero and moreover supported on an affine open subset. Thus applying the global sections functor and using \labelcref{bfo} (see also \cite[Example 6.1]{ChenOnTheConjecturesofBravermanKazhdan}) we obtain isomorphisms of complexes

\begin{align*}
U\LG/U\LG\cdot\LieAlgebraOfUnipotentRadicalOfBorel \otimes_{\Symt} \differentialOperatorsOnT \otimes_{\differentialOperatorsOnT^W} M^G 
& \cong U\LG/U\LG\cdot\LieAlgebraOfUnipotentRadicalOfBorel \otimes_{\Symt} \text{res}(M) 
\\ & \xrightarrow{\sim}  \Gamma\Big( \int_{\pi^{\BorelSubgroup}}M\Big) = R\Gamma\Big(\int_{\pi^{\BorelSubgroup}}M\Big)
= M \otimes_{U\LG}^L U\LG/U\LG\cdot\LieAlgebraOfUnipotentRadicalOfBorel
\end{align*} where the first isomorphism uses \cite{GinzburgParabolicInductionandtheHarishChandraDModule}. Applying the functor $\LieAlgebraofUnipotentRadicalOfOppositeBorelPSI\cdot U\LG\backslash U\LG \otimes_{U\LG} (\td)$ of left $\LieAlgebraofUnipotentRadicalOfOppositeBorelPSI$-coinvariants, we obtain an isomorphism \[\differentialOperatorsOnT \otimes_{\differentialOperatorsOnT^W} M^G \cong \LieAlgebraofUnipotentRadicalOfOppositeBorelPSI\cdot U\LG\backslash U\LG \otimes_{U\LG} M \otimes_{U\LG} U\LG/U\LG\cdot\LieAlgebraOfUnipotentRadicalOfBorel\] and so applying Skryabin's equivalence to this isomorphism we obtain an isomorphism \[\differentialOperatorsOnT \otimes_{\differentialOperatorsOnT^W} M^G \cong \Dpsilhbar/\Dpsilhbar\LieAlgebraofUnipotentRadicalOfOppositeBorelPSIr \otimes_{\biWhittakerDifferentialOperatorsonG} \kostantWhittakerForNONFILTEREDtoZGBIMOD(M) \otimes_{U\LG}U\LG/U\LG\cdot\LieAlgebraOfUnipotentRadicalOfBorel \cong \LieAlgebraOfUnipotentRadicalOfBorel^r\Dpsilhbar\backslash \Dpsilhbar/\Dpsilhbar\LieAlgebraofUnipotentRadicalOfOppositeBorelPSIr \otimes_{\biWhittakerDifferentialOperatorsonG} \kostantWhittakerForNONFILTEREDtoZGBIMOD(M) = \MiuraBimodule \otimes_{\biWhittakerDifferentialOperatorsonG} \kostantWhittakerForNONFILTEREDtoZGBIMOD(M)\] where $\MiuraBimodule$ is the Miura bimodule defined in \cref{Miura Bimodule and Comodules Subsection}. Tracing through the constructions, we see that this isomorphism is $W$-equivariant, and so taking invariants we obtain that $M^G \cong \MiuraBimodule^W \otimes_{\biWhittakerDifferentialOperatorsonG} \kostantWhittakerHBARtoZGBIMOD(M)$ lies in the essential image of the fully faithful functor $\MiuraBimodule^W \otimes_{\biWhittakerDifferentialOperatorsonG} (\td)$ and so we see that the $\differentialOperatorsOnT^W$-module structure on $M^G$ upgrades to a $\biWhittakerDifferentialOperatorsonG$-module, as desired.

We now claim that the counit map \begin{equation}\label{Counit Map for Very Centrals}\quantizationOfCentralizersHBAREQUALSONE \otimes_{\differentialOperatorsOnT^W} M^G \xrightarrow{} M\end{equation} for parabolic restriction is an isomorphism, which we will prove using arguments similar to those appearing in the proofs of \cite[Proposition 2.9]{BravermanKazhdanGammaSheavesonReductiveGroups} and  \cite[Proposition 3.3]{ChenOnTheConjecturesofBravermanKazhdan}. To this end, first observe that, since $M$ is very central, one can identify $\operatorname{Ind}_T^G(\operatorname{Res}_T^G(M))$ with $\text{hc}^L(\text{hc}(M))$ where $\text{hc}$ is the Harish-Chandra functor $\differentialOperatorsOnG\modwithdash^G \to \D_{G/N}\modwithdash^{B_{\text{ad}}}$ induced by pushforward, and $\text{hc}^L$ is its left adjoint. Direct computation proves the (well known) fact that $\text{hc}^L\text{hc}$ is naturally isomorphic to convolution with the Springer sheaf, and so we see that $\IndTG(\Res^G_T(M))$ is also given by convolution with the Springer sheaf. Therefore we see that the counit map \[\operatorname{Ind}_T^G(\operatorname{Res}_T^G(M))^W \to M\] is an isomorphism, and thus using \labelcref{Tensoring with quantizationOfCentralizers Is Parabolic Induction} we see that \eqref{Counit Map for Very Centrals} is an isomorphism as well, since it is the counit for a naturally isomorphic adjoint pair. Now, since $M^G$ has a $\biWhittakerDifferentialOperatorsonG$-module structure, we therefore obtain an isomorphism
$\ups(M^G) \xrightarrow{\sim} M$ and so $M$ lies in the essential image of the Knop-Ng\^o functor, as required.

\begin{rem}\label{Can Rederive Vanishing Conjecture Remark}
Since we have shown that a $G$-equivariant $\differentialOperatorsOnG$-module is very central if and only if it is induced from its $G$-invariants, \cref{Vanishing Conjecture For Parabolics for Central D Modules Induced from Invariants}(i) can be used to re-derive a Levi upgrade of the vanishing conjecture--a devissage argument immediately shows the claim at the level of derived categories follows from the vanishing conjecture for abelian categories.  This Levi upgrade was originally proved by Chen \cite[Theorem 6.2]{ChenTowardTheDepthZeroStableBernsteinCenterConjecture}.
\end{rem}

\subsection{Perverse sheaf counterpart of the $D$-module
  ${\mathbf N}\o_{D(T)^W}\mathfrak{J}$}\label{Perverse Sheaf Counterpart Remark} 
   Using \cref{Jcor}, it is possible to give a geometric construction of
     the object $\ups(\mathfrak{J})$ that appears in the quantized Knop-Ng\^o morphism
     that makes sense in the
      constructible setting (eg. of $\ell$-adic sheaves) as well.
      To this end, we define \[\wt \ups(\mathfrak{J}):= \ups(\biWhittakerDifferentialOperatorsonG) \otimes_{D(T)^W} D(T) \cong \mathbf{N}\o_{D(T)^W}\biWhittakerDifferentialOperatorsonG \otimes_{D(T)^W} D(T).\]
      Thus, $\wt \ups(\mathfrak{J})$ is a  $G\times W$-equivariant  $(D(G),\,D(T))$-bimodule which we may (and will) view as a left $D$-module on
      $G\times T$ using the isomorphism
      $D(T)\cong D(T)_{op}$ provided by the choice of a translation invariant trivialization of the canonical bundle on $T$.

One obtains a chain of functors
    \[
        \xymatrix{
          \mathsf{D}(G)\ \ar[rrr]^<>(0.5){\ \Av_*^{\bar N}\ccirc \Av_*^{N_{op},\psi} \ }
          &&& \ \mathsf{D}(G)^{(\bar N \times N,        \,\text{triv}\times\psi), W}\
          \ar[r]^<>(0.5){\ i^* \ }_<>(0.5){\cong} &
          \ \mathsf{D}(T)^W\  \ar[rr]^<>(0.5){\ \Ind_T^G(-)^W\ } &&\
          \mathsf{D}(G)^{\Ad G},
        }
      \]
      where the functor $i^*$ is known to be an equivalence by standard arguments, see for example \cite[Proposition 1.3]{GannonUniversalCategoryOandGelfandGraevAction}.
        Let $\Psi$ denote the composite functor and write
        $\Psi^{(1)}=\Psi\boxtimes \Id_{\mathsf{D}(T)}:\,
        \mathsf{D}(G\times T)\to \mathsf{D}(G\times T)^{\Ad G}$
        for a similarly defined functor that acts along the first factor of $G\times T$ only.
        
        Now, consider the natural diagram $T \xleftarrow{\ q  \ } \wt G \xrightarrow{\ \pi\ } G$
        where $\pi$ is the Grothendieck-Springer morphism.
        The map $\pi\times q: \wt G\to G\times T$ is a small projective morphism
        with image $G\times_{G/\!/\Ad G} T$.
         Hence,     
         $\int_{\pi\times q}\oo_{\wt G}$ is a simple holonomic
         $D$-module on $G\times T$ that comes equipped with a $G\times W$-equivariant structure
        with respect to the adjoint action of $G$ on $G$, resp. $W$ on $T$.
We remark that the $D(G)$-modules  $\mathbf N$ and $\ups(\mathfrak{J})$ are not
 holonomic.
 
        \begin{claim}\label{motivic claim} There is a canonical isomorphism
          $\wt\ups(\mathfrak{J})\cong \Psi^{(1)}(\int_{\pi\times q}\oo_{\wt G})$ of 
          $G\times W$-equivariant $\dd_{G\times T}$-modules, in particular,
          we have $\ups(\mathfrak{J})\cong\Psi^{(1)}(\int_{\pi\times q}\oo_{\wt G})^W$.
        \end{claim}

        Thus, the counterpart of  $\wt\ups(\mathfrak{J})$ in the constructible setting
        is provided by the  $G\times W$-equivariant perverse sheaf $\Psi^{(1)}((\pi\times q)_*C_{\wt G})$
        on $G\times T$
        where $C_{\wt G}$ is the constant perverse sheaf on $\wt G$.
        One can check that the perverse sheaf so defined has the `very central' property,
        as defined in \cite{ChenAVanishingConjecturetheGLnCase},  along the $T$-factor.
        This corresponds, in the $D$-module setting,
        to the fact that the right $D(T)^W$-action on
        $\mathbf{N}\o_{D(T)^W}\mathfrak{J}$ comes from a $\mathfrak{J}$-action.

\section{Proof of \cref{theorem asserting abstract central functor in biWhittaker case}}\label{Hopf Algebroids and Central Functors Section}
\newcommand{\generalHopfAlgebroidWithSourceEqualsTarget}{\mathcal{J}}
\newcommand{\commutativeBaseforGeneralHopfAlgebroid}{Z}
\subsection{}Recall that we have constructed a left $U\LG$-bialgebroid structure (in the sense of \cite[Definition 3.3]{BohmHopfAlgebroids}) on $\differentialOperatorsOnG$ in the proof of \cref{Differential Operators on G Has Hopf Algebroid Structure and Comodules Has Monoidal Stucture}, and used this to construct a left $\Zhbarg$-bialgebroid structure on $\biWhittakerDifferentialOperatorsHBAR$. We have also seen that the coalgebra structure on $\biWhittakerDifferentialOperatorsonG$ cocommutative in \cref{Miura Bimodule is Cocommutative Coalgebra and Isomorphic to DT Upon Tensoring With Field of Fractions}, and that $\biWhittakerDifferentialOperatorsonG$ is flat as a left $Z\LG$-module in \cref{Dpsirhbar and BiWhitHbar Are Flat}. Using the monoidal structure on the category of right $\differentialOperatorsOnG$-modules, one can completely analogously equip $\differentialOperatorsOnG$ with a right $U\LG_{\text{op}}$-structure, and run completely analogous arguments to prove that this induces a right $Z\LG$-bialgebroid structure on $\biWhittakerDifferentialOperatorsonG$ in the sense of \cite[Definition 3.1]{BohmHopfAlgebroids} with source and target map the inclusion of $\Zhbarg$ into $\biWhittakerDifferentialOperatorsHBAR$. 

Observe, moreover, that the ring $\differentialOperatorsOnG$ admits an algebra anti-automorphism $S$ uniquely determined by the conditions that it swaps $U\LG$ and $U\LG_{\op}$ and extends the antipode on $\mathcal{O}(G)$. It is not difficult to check that this equips $\differentialOperatorsOnG$ with the structure of a \textit{Hopf algebroid} with base algebras $U\LG$ and $U\LG_{\text{op}}$ in the sense of \cite[Definition 4.1]{BohmHopfAlgebroids}; since all of the maps in question are $U\LG$-linear or antilinear and restrict to $\mathcal{O}(G) \subseteq \differentialOperatorsOnG$ to the usual coproduct, counit, or antipode, it is not difficult to check the Hopf algebroid axioms hold for $\differentialOperatorsOnG$. Now, using the algebra isomorphism of \eqref{Algebra Isos For Different Descriptions of biWhittaker}, we see that $S$ induces an algebra involution $\overline{S}$ on $\biWhittakerDifferentialOperatorsHBAR$. The Hopf algebroid compatibilities are readily checked. We summarize this discussion in the following Proposition:

\begin{Proposition}\label{biWhit is Hopf Algebroid}
The ring structure on $\biWhittakerDifferentialOperatorsonG$ upgrades to a structure of a Hopf algebroid with both base algebras $Z\LG$ such that the antipode is bijective. Moreover, in either of the $Z\LG$-coalgebra structures on $\biWhittakerDifferentialOperatorsonG$, the comultiplication is cocommutative, and as a left $Z\LG$-module $\biWhittakerDifferentialOperatorsonG$ is flat.
\end{Proposition} 

Using \cref{biWhit is Hopf Algebroid}, we now prove \cref{theorem asserting abstract central functor in biWhittaker case}. In fact, \cref{theorem asserting abstract central functor in biWhittaker case} follows formally from the properties of $\biWhittakerDifferentialOperatorsonG$ listed in \cref{biWhit is Hopf Algebroid}. To exhibit this, we hereafter adapt the following notation: we assume we are given a commutative $K$-algebra $\commutativeBaseforGeneralHopfAlgebroid$ for $K$ some commutative ring and a Hopf algebroid $\generalHopfAlgebroidWithSourceEqualsTarget$ with bijective antipode whose base algebras both are $\commutativeBaseforGeneralHopfAlgebroid$ for which the source map $s: \commutativeBaseforGeneralHopfAlgebroid \to \generalHopfAlgebroidWithSourceEqualsTarget$ for the left bialgebroid structure agrees with the map \begin{equation}\label{Composition of Equals and Target for Left Bialgebroid}\commutativeBaseforGeneralHopfAlgebroid = \commutativeBaseforGeneralHopfAlgebroid_{\text{op}} \xrightarrow{t} \generalHopfAlgebroidWithSourceEqualsTarget\end{equation} induced by the target of the left bialgebroid structure such that the following conditions hold:

\vi Both left and right multiplication by the image of the source map exhibit  $\generalHopfAlgebroidWithSourceEqualsTarget$ as a flat $\commutativeBaseforGeneralHopfAlgebroid$-module.

\vii The left (equivalently right) coring structure is cocommutative.
\vskip 2pt
If $\generalHopfAlgebroidWithSourceEqualsTarget$ is a graded algebra, we will say that $\generalHopfAlgebroidWithSourceEqualsTarget$ is a \textit{graded Hopf algebroid} if $\generalHopfAlgebroidWithSourceEqualsTarget$ is equipped with a grading for which the multiplication map, unit map for the algebra structure, the comultiplication and counit maps, and the antipode all respect the grading. 

As $\generalHopfAlgebroidWithSourceEqualsTarget$ is a Hopf algebroid, its category of modules $\generalHopfAlgebroidWithSourceEqualsTarget$-mod and comodules $\generalHopfAlgebroidWithSourceEqualsTarget\text{-comod}$ for its left bialgebroid of $\generalHopfAlgebroidWithSourceEqualsTarget$ both acquire a monoidal structure and, moreover, the monoidal structure on $\generalHopfAlgebroidWithSourceEqualsTarget$-modules is symmetric, as we review in \cref{Notation for Monoidal Structure on Bialgebroid}. As $\generalHopfAlgebroidWithSourceEqualsTarget\comod$ is monoidal, we may consider its \textit{Drinfeld center} $\Z(\generalHopfAlgebroidWithSourceEqualsTarget\comod)$ (defined in for example \cite[Section 7.13]{EtingofGelakiNikshychOstrikTensorCategories}) whose objects are given by pairs $(C', Z^{C'}_{C})$ where $C' \in \generalHopfAlgebroidWithSourceEqualsTarget\comod$ and $Z^{C'}_{C}$ is a collection of isomorphisms \[Z^{C'}_{C}: C \otimes_{\commutativeBaseforGeneralHopfAlgebroid} C' \xrightarrow{\sim} C' \otimes_{\commutativeBaseforGeneralHopfAlgebroid} C\] natural in $C$. The collection of such objects naturally form a braided monoidal category: indeed, as we also recall in \cref{Generalities for Monoidal Functor Subsection}, the Drinfeld center as defined in \cite{EtingofGelakiNikshychOstrikTensorCategories} is evidently equivalent to the braided monoidal category obtained from the relative Drinfeld center construction $\Z_{\freeFunctorStandin}(\HarishChandraBimoduleStandin)$ defined in \cref{Generalities for Monoidal Functor Subsection} in the special case $\freeFunctorStandin := \text{id}$. 

Observe that our assumptions imply that the comultiplication and counit maps for the left bialgebroid $\generalHopfAlgebroidWithSourceEqualsTarget$ form a $\commutativeBaseforGeneralHopfAlgebroid$-coalgebra. Therefore, to any $X \in \commutativeBaseforGeneralHopfAlgebroid\modwithdash$, we may view $X$ as a trivial comodule for this coalgebra. In particular, restricting scalars along the map $s$, we may view any $\generalHopfAlgebroidWithSourceEqualsTarget$-module as a trivial $\generalHopfAlgebroidWithSourceEqualsTarget$-comodule. With this observation, we may now state our generalization of \cref{theorem asserting abstract central functor in biWhittaker case}, which says that this construction lifts to an embedding of the category of $\generalHopfAlgebroidWithSourceEqualsTarget$-modules into the Drinfeld center of the category of $\generalHopfAlgebroidWithSourceEqualsTarget$-comodules:

\begin{Theorem}\label{In Specific Context Module to Trivial Comodule Gives Central Functor with Explicit Centrality and Inverse}
    There is a braided monoidal, fully faithful functor \begin{equation}\label{Br Mon FF Functor}\generalHopfAlgebroidWithSourceEqualsTarget\mmod \to \mathcal{Z}(\generalHopfAlgebroidWithSourceEqualsTarget\text{-comod})\end{equation} to the Drinfeld center of the category of $\generalHopfAlgebroidWithSourceEqualsTarget$-comodules such that the underlying comodule of $M$ is the $\commutativeBaseforGeneralHopfAlgebroid$-module $M$ (viewed as a trivial $\generalHopfAlgebroidWithSourceEqualsTarget$-comodule) and the central structure $Z_C$ has inverse given by the formula $c \otimes m \mapsto c_{(0)}m \otimes c_{(1)}$ for any $c \in C \in \generalHopfAlgebroidWithSourceEqualsTarget\text{-comod}$ and $m \in M \in \generalHopfAlgebroidWithSourceEqualsTarget\mmod$. If $\generalHopfAlgebroidWithSourceEqualsTarget$ is a graded Hopf algebroid, then there moreover is an analogous functor to \labelcref{Br Mon FF Functor} for the category of graded modules and comodules.
\end{Theorem}

The remainder of \cref{Hopf Algebroids and Central Functors Section} will be devoted to the proof of \cref{In Specific Context Module to Trivial Comodule Gives Central Functor with Explicit Centrality and Inverse}. In proving \cref{In Specific Context Module to Trivial Comodule Gives Central Functor with Explicit Centrality and Inverse}, we will also give an explicit formula \eqref{Defining Equation for ZMC} for the map $Z_C$, using the structure of comodules for Hopf algebroids discussed in \cref{Notation for Comodules of BiWhittaker Differential Operators}. 

\begin{Remark}
The statement of \cref{In Specific Context Module to Trivial Comodule Gives Central Functor with Explicit Centrality and Inverse} can be viewed as a variant of the known result that, for any finite abelian group $A$, there is a fully faithful functor, symmetric monoidal functor $Z_A: \mathcal{O}(A)\modwithdash \to \Z(\text{Rep}(A))$. While the functor in \cref{In Specific Context Module to Trivial Comodule Gives Central Functor with Explicit Centrality and Inverse} is constructed completely analogously to $Z_A$, new technical considerations occur in the proof that this functor is defined and has the desired properties in the Hopf algebroid setting. Therefore, we provide a complete proof of \cref{In Specific Context Module to Trivial Comodule Gives Central Functor with Explicit Centrality and Inverse} below.
\end{Remark}
\subsection{Notation for monoidal structure on bialgebroid}\label{Notation for Monoidal Structure on Bialgebroid}
Recall that, for any $C \in \generalHopfAlgebroidWithSourceEqualsTarget\text{-comod}$, in addition to the underlying (left) $\commutativeBaseforGeneralHopfAlgebroid$-module structure there is also a right $\commutativeBaseforGeneralHopfAlgebroid$-module structure on $C$, given by the formula $cz = \epsilon_L(c_{(0)}z)c_{(1)}$, where we use Sweedler's notation and where $\epsilon_L$ is the counit for the left bialgebroid of $\generalHopfAlgebroidWithSourceEqualsTarget$ such that the induced functor $\generalHopfAlgebroidWithSourceEqualsTarget \to \commutativeBaseforGeneralHopfAlgebroid\text{-bimod}$ is monoidal \cite[Lemma 3.17]{BohmHopfAlgebroids}. Explicitly, in Sweedler's notation, the $\generalHopfAlgebroidWithSourceEqualsTarget$-comodule structure on the bimodule tensor product $B \otimes_{\commutativeBaseforGeneralHopfAlgebroid} C$ is given by the formula 
\begin{equation}\label{General Form for Monoidal Structure on Comodules for Bialgebroid}
b \otimes c \mapsto b_{(0)}c_{(0)} \otimes b_{(1)} \otimes c_{(1)}\end{equation} for any $b \in B, c \in C$, see \cite[Theorem 3.18]{BohmHopfAlgebroids}. If $B, C$ are equipped with compatible gradings, the tensor product also upgrades to a graded $\generalHopfAlgebroidWithSourceEqualsTarget$-comodule.

As $\generalHopfAlgebroidWithSourceEqualsTarget$ is a Hopf algebroid, the category of $\generalHopfAlgebroidWithSourceEqualsTarget$-modules also acquires a monoidal structure such that restriction of scalars along the ring map \[\commutativeBaseforGeneralHopfAlgebroid \otimes_K \commutativeBaseforGeneralHopfAlgebroid_{\text{op}} \to \generalHopfAlgebroidWithSourceEqualsTarget\] gives a monoidal functor to the category of $\commutativeBaseforGeneralHopfAlgebroid$-bimodules \cite[Theorem 3.13]{BohmHopfAlgebroids}. Explicitly, given $\generalHopfAlgebroidWithSourceEqualsTarget$-modules $M$, $N$ we may define a $\generalHopfAlgebroidWithSourceEqualsTarget$-module structure on the tensor product $M \otimes_{\commutativeBaseforGeneralHopfAlgebroid} N$ by the formula\begin{equation}\label{Module Structure on Tensor Product of Modules for Left Bialgebroid}u(m \otimes n) := u_{(0)}m \otimes u_{(1)}n\end{equation} where $m \in M, n \in N$, and $u \in \generalHopfAlgebroidWithSourceEqualsTarget$. Our assumption that $s$ is the composite of the maps in \eqref{Composition of Equals and Target for Left Bialgebroid} implies that this restriction of scalars naturally factors through the monoidal subcategory of $\commutativeBaseforGeneralHopfAlgebroid$\textit{-modules} in $\commutativeBaseforGeneralHopfAlgebroid$-bimodules. Moreover, $\generalHopfAlgebroidWithSourceEqualsTarget$ is cocommutative by assumption, and so from the formula \eqref{Module Structure on Tensor Product of Modules for Left Bialgebroid} it immediately follows that restriction of scalars along $s$ upgrades to a symmetric monoidal functor to the category of $\commutativeBaseforGeneralHopfAlgebroid$-modules. As in the comodule case, if $M, N$ are equipped with compatible gradings, then $M \otimes_{\commutativeBaseforGeneralHopfAlgebroid} N$ is a graded $\generalHopfAlgebroidWithSourceEqualsTarget$-module.

This discussion in particular shows that we may view any $\generalHopfAlgebroidWithSourceEqualsTarget$-module and any $\generalHopfAlgebroidWithSourceEqualsTarget$-comodule both as $\commutativeBaseforGeneralHopfAlgebroid$-bimodules, and \textit{any unadorned tensor product hereafter is taken with respect to the bimodule tensor product.} Additionally, if $P$ and $Q$ are any $\commutativeBaseforGeneralHopfAlgebroid$-bimodules, we use the symbol $P \otimes_{\ell, \ell} Q$ to denote the tensor product of the corresponding left $\commutativeBaseforGeneralHopfAlgebroid$-module structures. We similarly use the notation $P \otimes_{r, r} Q$, $P \otimes_{\ell, r} Q$, and  $P \otimes_{r, \ell} Q = P \otimes Q$.  

\subsection{Comodules of bialgebroid vs. comodules of Hopf algebroid}\label{Notation for Comodules of BiWhittaker Differential Operators}By our flatness assumption, \cite[Theorem 4.8]{BohmHopfAlgebroids} immediately implies the following:
\begin{Proposition}
For any comodule $(C, \rho_L)$ for the left bialgebroid $\generalHopfAlgebroidWithSourceEqualsTarget$, there exists left comodule structure $\rho_R$ for the right bialgebroid for $\generalHopfAlgebroidWithSourceEqualsTarget$ 
such that $\rho_R$ is a comodule map for the left bialgebroid for $\generalHopfAlgebroidWithSourceEqualsTarget$ map and $\rho_L$ is a comodule map for the right bialgebroid of $\generalHopfAlgebroidWithSourceEqualsTarget$, i.e. \begin{equation}\label{CoModule Map for One is Comodule Map for Other}(\text{id}_{\generalHopfAlgebroidWithSourceEqualsTarget}\otimes \rho_R) \circ \rho_L = (\Delta_L \otimes \text{id}_{\generalHopfAlgebroidWithSourceEqualsTarget}) \circ \rho_R\text{, } (\text{id}_{\generalHopfAlgebroidWithSourceEqualsTarget} \otimes \rho_L) \circ \rho_R = (\Delta_R \otimes \text{id}_{\generalHopfAlgebroidWithSourceEqualsTarget}) \circ \rho_L  
\end{equation}
and moreover any morphism of $\generalHopfAlgebroidWithSourceEqualsTarget$-comodules is automatically compatible with $\Delta_R$.
\end{Proposition} 
For any $C \in \generalHopfAlgebroidWithSourceEqualsTarget\text{-comod}$, we use superscripts to refer to the coaction of $\rho_R$, i.e. we set $\rho_R(c) = c^{(0)} \otimes c^{(1)}$ and, as above, we use subscripts to refer to the coaction $\rho_L$. 

\subsection{Construction of central structure $Z_C$}\label{Construction of Central Structure} 
For a fixed $C \in \generalHopfAlgebroidWithSourceEqualsTarget\comod$ and $M \in \generalHopfAlgebroidWithSourceEqualsTarget\modwithdash$, we define a map of $\commutativeBaseforGeneralHopfAlgebroid$-bimodules via the formula $Z_C: M \otimes C \xrightarrow{} C \otimes M$ via the formula
\begin{equation}\label{Defining Equation for ZMC}
    Z_C(m \otimes c) := c^{(1)} \otimes A(c^{(0)})m
\end{equation} where $A$ is the antipode. When we wish to make the dependence on the $\generalHopfAlgebroidWithSourceEqualsTarget$-module $M$ clear, we also denote by $Z_C^M$.

Similarly we may define a map of $\commutativeBaseforGeneralHopfAlgebroid$-bimodules $C \otimes M \xrightarrow{} M \otimes C$ via the formula \begin{equation}\label{Inverse to ZMc in General}c \otimes m \mapsto c_{(0)}m \otimes c_{(1)}\end{equation} which we denote by $Z_C^{-1}$ or, when we wish to make the dependence on $M$ clear, we denote it by $Z_C^{M, -1}$. This notation is justified by the following claim:

\begin{Proposition}\label{Inverse to ZMC Has Explicit Formula}
The map $Z_C^M$ is graded if $\generalHopfAlgebroidWithSourceEqualsTarget$ is graded, and the inverse to $Z_C^M$ is given by the formula \eqref{Inverse to ZMc in General}. 
\end{Proposition}

Before proving \cref{Inverse to ZMC Has Explicit Formula}, we make one (non-circular) observation which will also be used later. Observe that the map given by the formula \eqref{Inverse to ZMc in General} fits into the following commutative diagram:
\begin{equation}\label{Commutative Diagram Reducing General Module Claims to Module of Algebra Itself}
\xymatrix@R+2em@C+2em{
C \otimes_{R, L} \generalHopfAlgebroidWithSourceEqualsTarget \otimes_{\generalHopfAlgebroidWithSourceEqualsTarget} M \ar[r]^{Z_C^{\generalHopfAlgebroidWithSourceEqualsTarget, -1} \otimes \text{id}_M} \ar[d]_{\sim} & \generalHopfAlgebroidWithSourceEqualsTarget \otimes_{L, L} C \otimes_{\generalHopfAlgebroidWithSourceEqualsTarget} M \ar[d]_{\sim}\\
C \otimes_{\commutativeBaseforGeneralHopfAlgebroid} M \ar[r]^{Z_C^{M, -1}} & M \otimes_{\commutativeBaseforGeneralHopfAlgebroid} C
}
\end{equation} where the downward arrows are the maps induced by the $\generalHopfAlgebroidWithSourceEqualsTarget$-module structure,\footnote{In particular, the rightmost arrow is given by the formula $b \otimes c \otimes m \mapsto bm \otimes c$.} and a similar diagram using the maps $Z_C^M$ and $Z_C^{\generalHopfAlgebroidWithSourceEqualsTarget}$ commutes.
\begin{proof}[Proof of \cref{Inverse to ZMC Has Explicit Formula}]
Using the commutative diagram \eqref{Commutative Diagram Reducing General Module Claims to Module of Algebra Itself} and the analogous diagram for $Z_C^M$, it suffices to show this claim when $M = \generalHopfAlgebroidWithSourceEqualsTarget$ itself. Letting $\mu$ denote the multiplication morphism in $\generalHopfAlgebroidWithSourceEqualsTarget$, one readily checks that $Z^{\generalHopfAlgebroidWithSourceEqualsTarget}_C$ identifies with the composite \begin{equation}\label{Inverse to Composite Giving Quantization of j x mapsto j jx}\generalHopfAlgebroidWithSourceEqualsTarget \otimes_{\ell, \ell} C \xrightarrow{A \otimes \text{id}} \generalHopfAlgebroidWithSourceEqualsTarget \otimes_{r, \ell} C \xrightarrow{\text{id} \otimes \rho_R} \generalHopfAlgebroidWithSourceEqualsTarget \otimes_{r, \ell} \generalHopfAlgebroidWithSourceEqualsTarget \otimes_{r, r} C \xrightarrow{\mu \otimes \text{id}}\generalHopfAlgebroidWithSourceEqualsTarget \otimes_{r, r} C \xrightarrow{A \otimes \text{id}} \generalHopfAlgebroidWithSourceEqualsTarget \otimes_{\ell, r} C   \xrightarrow{\text{swap}}C \otimes_{r, \ell} \generalHopfAlgebroidWithSourceEqualsTarget  \end{equation}
    which is in particular right $\generalHopfAlgebroidWithSourceEqualsTarget$-linear. When $M = \generalHopfAlgebroidWithSourceEqualsTarget$, the map \eqref{Inverse to ZMc in General} is right $\generalHopfAlgebroidWithSourceEqualsTarget$-linear as well, as it is induced by the comultiplication map for the left bialgebroid. Therefore it suffices to show that, for any $c \in C$, the formula \eqref{Inverse to ZMc in General} provides a right inverse to $Z_{C}^{\generalHopfAlgebroidWithSourceEqualsTarget}$ on objects of the form $1 \otimes c$, or in other words \begin{equation}\label{Half of Inverse Claim for Hopf Algebroids}Z^{\generalHopfAlgebroidWithSourceEqualsTarget}_C(c_{(0)} \otimes c_{(1)}) = c \otimes 1,\end{equation} and that the formula \eqref{Inverse to ZMc in General} also gives a left inverse to $Z_C^{\generalHopfAlgebroidWithSourceEqualsTarget}$ on objects of the form $1 \otimes c$ for $c \in C$, i.e. if we apply \eqref{Inverse to ZMc in General} to $Z^{\generalHopfAlgebroidWithSourceEqualsTarget}_C(1 \otimes c)$ we obtain $1 \otimes c$. We have equalities \[(\text{swap} \circ Z_C^{\generalHopfAlgebroidWithSourceEqualsTarget}(c_{(0)} \otimes c_{(1)})) = (A \otimes \text{id})(\mu \otimes \text{id})(\text{id} \otimes \rho_R)(A \otimes \text{id}) \rho_L(c) = (A \otimes \text{id})(\mu \otimes \text{id})(A \otimes \rho_R)\rho_L(c)\] \[= (A \otimes \text{id})(\mu \otimes \text{id})(A \otimes \text{id} \otimes \text{id})(\text{id} \otimes \rho_R) \rho_L(c) = (A \otimes \text{id})(\mu \otimes \text{id})(A \otimes \text{id} \otimes \text{id})(\Delta_t \otimes \text{id}_{\generalHopfAlgebroidWithSourceEqualsTarget})\circ \rho_R(c)\] \[= (A \otimes \text{id})\circ (\epsilon_s \otimes \text{id}) \circ \rho_R(c) = (A \otimes \textit{id})(1 \otimes c) = 1 \otimes c\] where the first step uses the description of $Z_C$ as in  \eqref{Inverse to Composite Giving Quantization of j x mapsto j jx} and the definition of $\rho_L(c)$, the second and third steps follow from the definition of the tensor product of maps, the fourth step uses \eqref{CoModule Map for One is Comodule Map for Other}, the fifth step is by the definition of the inverse in a Hopf algebroid, see for example \cite[Definition 4.1(iv)]{BohmHopfAlgebroids}, the sixth step uses that $C$ is a comodule and the seventh step uses the fact the antipode is an algebra morphism. Thus applying the swap map to both sides of this equality in \eqref{Half of Inverse Claim for Hopf Algebroids}. The fact that the formula \eqref{Inverse to ZMc in General} provides a left inverse to $Z_C$ is completely analogous. The fact that $Z_C^M$ is graded follows either by direct computation or by the observation that $Z_C^{M, -1}$ is obviously graded.
\end{proof}

Now, using cocommutativity, we may verify that $Z_C$ actually gives a map of $\generalHopfAlgebroidWithSourceEqualsTarget$-comodules: 

\begin{Proposition}\label{Composite Giving Quantization of j x mapsto j jx is left comodule map}
    Equipping $\generalHopfAlgebroidWithSourceEqualsTarget$ with a trivial left $\generalHopfAlgebroidWithSourceEqualsTarget$-comodule structure, the map $Z_C$ is a $\generalHopfAlgebroidWithSourceEqualsTarget$-comodule map. 
\end{Proposition}

\begin{proof}
    It suffices to show $Z_C^{-1}$ is a $\generalHopfAlgebroidWithSourceEqualsTarget$-comodule map. As above using \eqref{Inverse to ZMc in General} it suffices to prove this in the case $M = \generalHopfAlgebroidWithSourceEqualsTarget$. Let $P_L$ denote the induced comodule structure on $\generalHopfAlgebroidWithSourceEqualsTarget \otimes C$. We wish to show that \begin{equation}\label{Equality to Show ZMC is a Comodule Map}P_L \circ Z_C = (\text{id}_{\generalHopfAlgebroidWithSourceEqualsTarget} \otimes Z_C) \circ P_L\end{equation} for any comodule $C$. Since both maps are right $\generalHopfAlgebroidWithSourceEqualsTarget$-linear, it suffices to show this on objects of the form $c \otimes 1$. In this case, we see that \[P_L \circ Z_C(c \otimes 1) = P_L(c_{(0)} \otimes c_{(1)}) = c_{(1, 0)} \otimes c_{(0)} \otimes c_{(1)} = c_{(0, 1)} \otimes c_{(0, 0)}  \otimes c_{(1)}\] where the second equality uses the formula \eqref{General Form for Monoidal Structure on Comodules for Bialgebroid} for the coaction on the tensor product and the final equality is obtained by swapping the first two factors in the equality \[c_{(0)} \otimes c_{(1, 0)} \otimes c_{(1, 1)} = c_{(0, 0)} \otimes c_{(0, 1)} \otimes c_{(1)}\] which in turn holds since $C$ is a comodule. Similarly using \eqref{General Form for Monoidal Structure on Comodules for Bialgebroid} we have that \[(\text{id}_{\generalHopfAlgebroidWithSourceEqualsTarget} \otimes Z_C) \circ P_L(c \otimes 1) = (\text{id}_{\generalHopfAlgebroidWithSourceEqualsTarget} \otimes Z_C)(c_{(0)} \otimes 1 \otimes c_{(1)}) = c_{(0)} \otimes c_{(1, 0)} \otimes c_{(1, 1)} = c_{(0, 0)} \otimes c_{(0, 1)} \otimes c_{(1)}\] and since $\generalHopfAlgebroidWithSourceEqualsTarget$ is cocommutative assumption we obtain the equality in \eqref{Equality to Show ZMC is a Comodule Map} as desired.
\end{proof}

\subsection{Construction of functor}
Using the isomorphism in \cref{Construction of Central Structure}, we construct the functor in \cref{In Specific Context Module to Trivial Comodule Gives Central Functor with Explicit Centrality and Inverse}: 

\begin{Corollary}\label{Construction of Overline Ngo as Non Braided Functor}
    The assignment $M \mapsto (M, Z_C^M)$ upgrades to a functor $\generalHopfAlgebroidWithSourceEqualsTarget\modwithdash \to \mathcal{Z}(\generalHopfAlgebroidWithSourceEqualsTarget\text{-comod})$, where we view $M$ as a trivial $\generalHopfAlgebroidWithSourceEqualsTarget$-comodule as above. If $\generalHopfAlgebroidWithSourceEqualsTarget$ has a grading compatible with the algebra structure, coalgebra structures, and antipode, then there is a graded analogue.
\end{Corollary}

\begin{proof}
Observe that the map $Z_{C}^{\generalHopfAlgebroidWithSourceEqualsTarget}$ is induced by coaction map of $C$, which is a map of right $\commutativeBaseforGeneralHopfAlgebroid$-modules. Therefore the maps $Z_{C}^{\generalHopfAlgebroidWithSourceEqualsTarget}$ are natural in $C$. Using the diagram \eqref{Commutative Diagram Reducing General Module Claims to Module of Algebra Itself} we therefore see that $Z_{C}^{M}$ is natural in $C$ for any $\generalHopfAlgebroidWithSourceEqualsTarget$-module $M$. 

We wish to show that $(M, Z_C^M)$ gives an object in $\Z(\generalHopfAlgebroidWithSourceEqualsTarget\comod)$; to do this, we must show that, for any $B, C \in \generalHopfAlgebroidWithSourceEqualsTarget\text{-comod}$  that \[Z_{B \otimes C} = (\text{id}_B \otimes Z_C) \circ (Z_B \otimes \text{id}_C).\] Of course, we may show that their inverses agree.  
    
As above, using \eqref{Inverse to ZMc in General} it suffices to prove this equality in the case $M = \generalHopfAlgebroidWithSourceEqualsTarget$. Since all of these maps are right $\generalHopfAlgebroidWithSourceEqualsTarget$-module maps, we may show this equality on objects of the form $b \otimes c \otimes 1$. We see that \[Z_{B \otimes C}^{-1}(b \otimes c \otimes 1) = (b \otimes c)_{(0)} \otimes (b \otimes c)_{(1)} = b_{(0)}c_{(0)} \otimes b_{(1)} \otimes c_{(1)}\] by the definition of the comodule structure on the tensor product as in \eqref{General Form for Monoidal Structure on Comodules for Bialgebroid}. On the other hand, we have \[(Z_B^{-1} \otimes \text{id}_C) \circ (\text{id}_B \otimes Z_C^{-1})(b \otimes c \otimes 1) = (Z_B^{-1} \otimes \text{id}_C)(b \otimes c_{(0)} \otimes c_{(1)}) = b_{(0)}c_{(0)} \otimes b_{(1)} \otimes c_{(1)}\] obtained from twice applying \cref{Inverse to ZMC Has Explicit Formula}, and so we see these maps agree and thus $M$ indeed gives rise to an object of $\Z(\generalHopfAlgebroidWithSourceEqualsTarget\comod)$. Using the commutative diagram \eqref{Commutative Diagram Reducing General Module Claims to Module of Algebra Itself}, one also shows that any map $M_1 \to M_2$ of $\generalHopfAlgebroidWithSourceEqualsTarget$-modules induces a map in $\Z(\generalHopfAlgebroidWithSourceEqualsTarget\comod)$.
\end{proof}

\begin{proof}[Proof of \cref{In Specific Context Module to Trivial Comodule Gives Central Functor with Explicit Centrality and Inverse}]
We have defined the functor in \cref{Construction of Overline Ngo as Non Braided Functor}; it remains to show it is fully faithful and braided monoidal. Since the forgetful functor from the Drinfeld center of a monoidal category to the monoidal category is faithful, the functor of \cref{Construction of Overline Ngo as Non Braided Functor} is faithful. To see that it is full, assume that $M, M' \in \generalHopfAlgebroidWithSourceEqualsTarget\modwithdash$ and assume $f: M \to M'$ is a map in $\Z(\generalHopfAlgebroidWithSourceEqualsTarget\comod)$. We wish to show that $f$ is a map of $\generalHopfAlgebroidWithSourceEqualsTarget$-modules. However, observe that if we view $\generalHopfAlgebroidWithSourceEqualsTarget$ as a left $\generalHopfAlgebroidWithSourceEqualsTarget$-comodule via the left comultiplication of the left bialgebroid, the diagram
\begin{equation*}
\xymatrix@R+2em@C+2em{ 
\generalHopfAlgebroidWithSourceEqualsTarget \otimes_{\commutativeBaseforGeneralHopfAlgebroid} M \ar[r]^{Z_{\generalHopfAlgebroidWithSourceEqualsTarget}^{M, -1}} \ar[d]_{\text{id} \otimes f} & 
M \otimes_{\commutativeBaseforGeneralHopfAlgebroid} \generalHopfAlgebroidWithSourceEqualsTarget \ar[d]_{f \otimes \text{id}} \ar[r]^{\text{id} \otimes \epsilon_L} \ar[r]^{\sim} & 
M \otimes_{\commutativeBaseforGeneralHopfAlgebroid} \commutativeBaseforGeneralHopfAlgebroid \ar[d]^{f \otimes \text{id}} \ar[r]^{\sim} & 
M \ar[d]^f \\
\generalHopfAlgebroidWithSourceEqualsTarget \otimes_{\commutativeBaseforGeneralHopfAlgebroid} M' \ar[r]^{Z_{\generalHopfAlgebroidWithSourceEqualsTarget}^{M, -1}} & 
M' \otimes_{\commutativeBaseforGeneralHopfAlgebroid} \generalHopfAlgebroidWithSourceEqualsTarget \ar[r]^{\text{id} \otimes \epsilon_L} \ar[r]^{\sim} & 
M' \otimes_{\commutativeBaseforGeneralHopfAlgebroid} \commutativeBaseforGeneralHopfAlgebroid \ar[r]^{\sim} & 
M'
}
\end{equation*}
commutes: the leftmost square commutes since $f$ is a map in $\Z(\generalHopfAlgebroidWithSourceEqualsTarget\comod)$, the middle and final square commute since the counit $\epsilon_L$ for the left bialgebroid is $\commutativeBaseforGeneralHopfAlgebroid$-linear, and the unlabeled maps are induced by the module structure and in particular commute. Traversing the three arrows on the top of the diagram, we obtain the map $b \otimes m \mapsto \epsilon_L(b_{(1)})b_{(0)}m = bm$ where this equality holds since $\generalHopfAlgebroidWithSourceEqualsTarget$ is a coalgebra; similarly the bottom composition of the three arrows on the bottom of the diagram give the map $b \otimes m' \mapsto bm'$, and so this diagram shows that $f$ must be a map of $\generalHopfAlgebroidWithSourceEqualsTarget$-modules, as required. 

We now show braided monoidality. Denote the monoidal structure in $\mathcal{Z}(\generalHopfAlgebroidWithSourceEqualsTarget\text{-comod})$ by $\star$. We claim that the monoidality isomorphism \[(M, Z_C^M) \star (N, Z_C^N) \cong (M \otimes N, Z_C^{M \otimes N})\] is simply an equality, which of course is natural in $M$ and $N$. Using an analogous diagram to \eqref{Inverse to ZMc in General} it suffices to prove this when $M = N = \generalHopfAlgebroidWithSourceEqualsTarget$. 

Of course, the underlying (trivial) $\generalHopfAlgebroidWithSourceEqualsTarget$-comodules agree, and so it remains to verify the central structures are the same or, in other words, we wish to show that \begin{equation}\label{Equality of Central Morphisms to Show Monoidality}(Z^{\generalHopfAlgebroidWithSourceEqualsTarget}_C \otimes \text{id}_{\generalHopfAlgebroidWithSourceEqualsTarget}) \circ (\text{id}_{\generalHopfAlgebroidWithSourceEqualsTarget} \otimes Z^{\generalHopfAlgebroidWithSourceEqualsTarget}_C) = Z^{\generalHopfAlgebroidWithSourceEqualsTarget \otimes_{\commutativeBaseforGeneralHopfAlgebroid} \generalHopfAlgebroidWithSourceEqualsTarget}_C\end{equation} for any $C \in \generalHopfAlgebroidWithSourceEqualsTarget$-comod. As above, we may check that the inverse of both maps agree and, since both maps are right $\generalHopfAlgebroidWithSourceEqualsTarget \otimes_k \generalHopfAlgebroidWithSourceEqualsTarget$-module maps, we may check that these inverses agree on objects of the form $c \otimes 1 \otimes 1$. Now one directly checks that the inverse of the left map in \eqref{Equality of Central Morphisms to Show Monoidality} takes $c \otimes 1 \otimes 1$ to $c_{(0)} \otimes c_{(1, 0)} \otimes c_{(1, 1)}$ and that \[Z^{\generalHopfAlgebroidWithSourceEqualsTarget \otimes_{\commutativeBaseforGeneralHopfAlgebroid} \generalHopfAlgebroidWithSourceEqualsTarget, -1}_C(c \otimes 1 \otimes 1) = c_{(0)}(1 \otimes 1) \otimes c_{(1)} = c_{(0, 0)} \otimes c_{(0, 1)} \otimes c_{(1)}\] using the formula \eqref{Module Structure on Tensor Product of Modules for Left Bialgebroid}. Both of these sums agree by the definition of a comodule.

Finally, since our monoidality morphisms are equalities, to show that this functor is braided it suffices to prove that the braiding $Z_{N}^M: M  \otimes  N\to N\mathop{\otimes} M$ is the swap morphism for any $M, N \in \generalHopfAlgebroidWithSourceEqualsTarget\modwithdash$. However, this follows by direct computation: \[Z^M_{N}(m \otimes n) = n^{(1)} \otimes A(n^{(0)})m = n \otimes m\] since the comodule structure on $N$ is trivial. 
\end{proof}
\printbibliography
\end{document}